\documentclass{article}
\usepackage[english]{babel}
\usepackage[letterpaper,top=2cm,bottom=2cm,left=3cm,right=3cm,marginparwidth=1.75cm]{geometry}
\usepackage{threeparttable}
\usepackage{multirow}
\usepackage{placeins}
\usepackage{amsmath}
\usepackage{amssymb}
\usepackage{graphicx}
\usepackage{array}
\usepackage[colorlinks=true, allcolors=blue]{hyperref}
\usepackage{amsthm}
\usepackage{float}
\usepackage{caption}
\usepackage{subcaption}
\usepackage{tikz}
\usetikzlibrary{calc}
\usepackage{algorithm}  
\usepackage{algorithmic} 
\usepackage{bm}     
\newtheorem{remark}{Remark}[section]

\newtheorem{lemma}{Lemma}[section]
\newtheorem{theorem}[lemma]{Theorem}

\numberwithin{equation}{section}
\allowdisplaybreaks[4]

\title{An unconditionally stable second-order Robin partitioned method for fluid--poroelastic structure interaction}
\author{Wenlong He\thanks{Faculty of Science and Technology, Beijing Normal-Hong Kong Baptist University, Zhuhai 519087, China.}\and Thomas Wick\thanks{Institute of Applied Mathematics, Leibniz University Hannover, Welfengarten 1, 30167 Hannover, Germany.}\and Xiaohe Yue\thanks{Corresponding author. School of Mathematical Sciences, East China Normal University, Shanghai 200241, China, Institute of Applied Mathematics, Leibniz University Hannover, Welfengarten 1, 30167 Hannover, Germany.}\and Jiwei Zhang\thanks{School of Mathematics and Statistics, and Hubei Key Laboratory of Computational Science, Wuhan University, Wuhan 430072, China.}\and Haibiao Zheng\thanks{School of Mathematical Sciences, Ministry of Education Key Laboratory of Mathematics and Engineering Applications, Shanghai Key Laboratory of PMMP, East China Normal University, Shanghai, 200241, China.}}
\date{}

\begin{document}
\maketitle

\begin{abstract}
    We study an unsteady fluid-poroelastic structure interaction (FPSI) problem, in which the free fluid region is governed by the incompressible Stokes equations and the poroelastic region by the fully dynamic Biot system. The original physical coupling conditions are equivalently reformulated as Robin-type interface conditions and an auxiliary interface variable is introduced as common Robin data shared by the two subsystems. This reformulation decomposes the coupled FPSI problem into fluid and poroelastic subproblems posed on their respective subdomains. The problem is discretized in space by the finite element method (FEM), while the backward Euler scheme and the BDF2 scheme combined with the second-order interface extrapolation are employed for the temporal discretization. 
    By exploiting the resulting discrete dynamic relations, we derive an update formula for the auxiliary interface variable and obtain fully discrete first- and second-order Robin partitioned schemes in which the two subproblems are solved independently at each time step. The unconditional stability of both schemes is rigorously established.
     Numerical experiments verify the temporal convergence orders of the proposed schemes and investigate the influence of the Robin parameter. 
    A classical pressure-wave benchmark verifies the accuracy of the proposed methods, while supplementary nonlinear moving-domain simulations illustrate their applicability to more complex FPSI configurations.
\end{abstract}


\section{Introduction}

Fluid--poroelastic structure interaction (FPSI) problems arise in a wide range of scientific and engineering applications, including subsurface flow, geological carbon sequestration, hydraulic fracturing, and biological systems such as blood flow through 
deformable tissues~\cite{ambartsumyan2019nonlinear,barbati2016complex, bukac2015effects, ruiz2022biot}. Compared with conventional fluid--structure interaction (FSI) problems, FPSI systems involve additional coupling mechanisms associated with fluid transport through the porous medium, resulting in complex multiphysics systems with multiple spatial and temporal scales. Developing accurate, stable, and computationally efficient numerical methods for such systems therefore remains a challenging task in computational poromechanics.

Existing numerical approaches for FPSI problems are broadly classified as monolithic or partitioned. Monolithic formulations solve all fluid, poroelastic, and interface unknowns simultaneously and have been analyzed with respect to stability, well-posedness, and conservation \cite{ambartsumyan2018lagrange,badia2009coupling,wen2020strongly,wen2021discontinuous}; their cost is the solution of large coupled algebraic systems. Partitioned methods instead solve the two subproblems separately and exchange interface data. Robin- and Nitsche-type interface treatments have been developed for FSI and FPSI problems \cite{badia2008fluid,bukavc2015partitioning,seboldt2021numerical}, while the added-mass stability of partitioned FSI schemes has been extensively studied \cite{causin2005added,fernandez2013explicit,guidoboni2009stable}. Recent Robin--Robin schemes further demonstrate the potential of these methods for FPSI problems \cite{burman2022fully,parrow2025robin,parrow2026stability}.

However, most FPSI schemes use backward Euler time stepping because of its robustness \cite{guo2025fully,he2026locking}. Existing second-order partitioned methods retain one or more limitations: modification of the Biot model \cite{oyekole2020second}, interface subiterations \cite{parrow2025robin}, a time-step restriction depending on physical and trace constants \cite{guo2026stability}, or a CFL-type condition \cite{wang2026second}. Developing a second-order method that is noniterative, unconditionally stable, and consistent with the original Biot formulation therefore remains challenging.

In this work, we develop unconditionally stable first- and second-order Robin partitioned methods for the Stokes--Biot model. An equivalent Robin reformulation and an auxiliary interface variable decompose the coupled problem into fluid and poroelastic subproblems. Backward Euler and BDF2 time discretizations, combined with an update formula derived from the discrete interface relations, yield noniterative first- and second-order schemes. Discrete energy estimates establish their unconditional stability.

The main contributions of this work are summarized as follows:
\begin{itemize}
\item A Robin-based partitioned formulation is developed for the Stokes--Biot fluid--poroelastic interaction problem by introducing an auxiliary interface variable that serves as common Robin data for the two subproblems.
\item Fully discrete first- and second-order Robin partitioned schemes are constructed. The second-order method combines BDF2 time discretization with second-order extrapolation of the transmitted interface quantities and does not require interface subiterations. Rigorous unconditional stability estimates are established for both schemes through discrete energy arguments.
\item Numerical experiments verify the temporal convergence behavior and compare the proposed schemes with monolithic discretizations for a pressure-wave benchmark. Supplementary nonlinear moving-domain computations illustrate applicability beyond the linear fixed-domain setting.
\end{itemize}

The remainder of the paper is organized as follows. Section 2 presents the model and coupling conditions; Section 3 derives the Robin partitioned formulation; Section 4 develops and analyzes the fully discrete schemes; and Section 5 reports the numerical experiments. Section 6 concludes the paper.

\section{The fully coupled Stokes--Biot model}
\label{sec:fully_coupled_model}

Let $\Omega_f,\Omega_p\subset\mathbb{R}^{d}$, $d=2,3$, be the bounded
Lipschitz domains occupied by an incompressible viscous fluid and a
poroelastic medium, respectively. We assume that
$\Omega_f\cap\Omega_p=\emptyset$ and define the fluid--poroelastic
interface by
$
\Gamma=\partial\Omega_f\cap\partial\Omega_p.
$
The remaining parts of the fluid and poroelastic boundaries are denoted
by
$
\Gamma_f=\partial\Omega_f\setminus\overline{\Gamma},
~
\Gamma_p=\partial\Omega_p\setminus\overline{\Gamma},
$
respectively. When needed, these exterior boundaries are further
decomposed according to the prescribed boundary conditions.

The outward unit normal vectors to $\Omega_f$ and $\Omega_p$ are denoted
by $\mathbf{n}_f$ and $\mathbf{n}_p$, respectively. Hence,
$
\mathbf{n}_p=-\mathbf{n}_f~\text{on }\Gamma.
$
Moreover, $\{\boldsymbol{\tau}_j\}_{j=1}^{d-1}$ denotes an orthonormal
basis of the tangent space to $\Gamma$. A schematic representation of
the computational configuration and the interface orientation is shown
in Figure~\ref{fig:FPSI_domain}.

\begin{figure}[t]
\centering
\includegraphics[width=0.4\textwidth,height=0.3\textwidth]{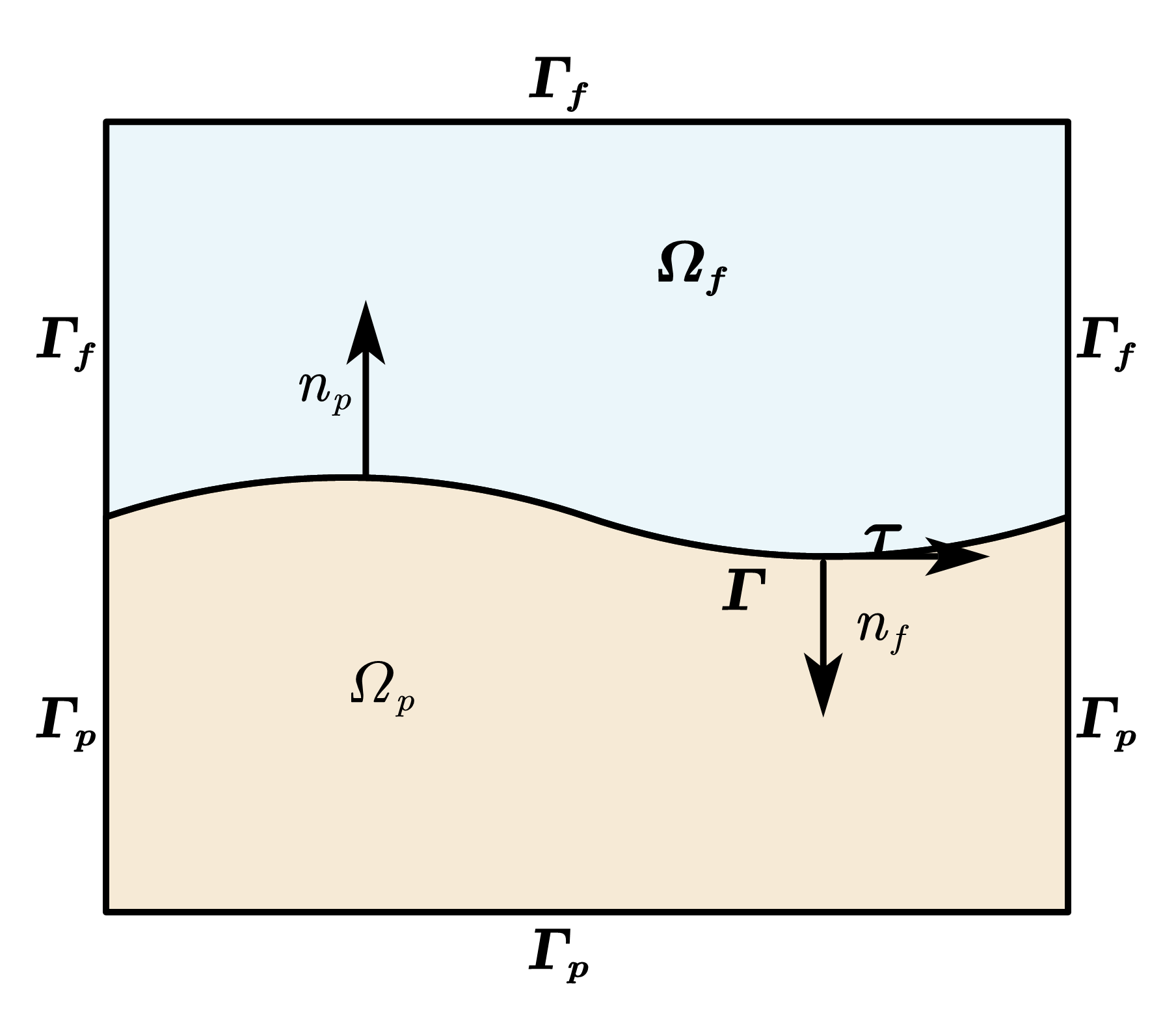}
\caption{Schematic representation of the fluid domain $\Omega_f$, the
poroelastic domain $\Omega_p$, and their common interface $\Gamma$.
The unit normals satisfy $\mathbf{n}_p=-\mathbf{n}_f$ on $\Gamma$, and
$\boldsymbol{\tau}$ denotes a unit tangent vector to the interface.}
\label{fig:FPSI_domain}
\end{figure}

\subsection{Governing equations}
\label{subsec:governing_equations}
The fluid velocity and pressure are denoted by
$\mathbf{u}_f$ and $p_f$, respectively. The fluid motion is governed by the time-dependent Stokes equations
\begin{align}
\rho_f\partial_t\mathbf{u}_f-\nabla\cdot\boldsymbol{\sigma}_f(\mathbf{u}_f,p_f)&=\mathbf{f}_f &&\text{in }\Omega_f\times(0,T], \label{eq:stokes_momentum}\\
\nabla\cdot\mathbf{u}_f&=\phi_f &&\text{in }\Omega_f\times(0,T], \label{eq:stokes_mass}
\end{align}
where $\rho_f>0$ and $\mu_f>0$ are the fluid density and dynamic viscosity, respectively, $\mathbf{f}_f$ is a prescribed body force, and $\phi_f$ is a prescribed source term. The fluid Cauchy stress tensor is defined by
\begin{equation*}
\boldsymbol{\sigma}_f(\mathbf{u}_f,p_f)=2\mu_f\boldsymbol{\varepsilon}(\mathbf{u}_f)-p_f\mathbf{I},
\qquad
\boldsymbol{\varepsilon}(\mathbf{u}_f)=\frac{1}{2}\left(\nabla\mathbf{u}_f+(\nabla\mathbf{u}_f)^T\right).
\label{eq:fluid_stress}
\end{equation*}
For an incompressible fluid without volumetric sources, one has $\phi_f=0$.

In the poroelastic domain $\Omega_p$, the displacement, structure velocity,
Darcy velocity, and pore pressure are denoted by
$\boldsymbol{\eta}_p$, $\boldsymbol{\xi}_p$, $\mathbf{u}_p$, and
$p_p$, respectively. The poroelastic medium is described by the fully dynamic Biot system~\cite{biot1941general, cesmelioglu2017analysis}
\begin{align}
\rho_p\partial_t\boldsymbol{\xi}_p-\nabla\cdot\boldsymbol{\sigma}_p(\boldsymbol{\eta}_p,p_p)&=\mathbf{f}_p &&\text{in }\Omega_p\times(0,T], \label{eq:biot_momentum}\\
\partial_t\boldsymbol{\eta}_p-\boldsymbol{\xi}_p&=\mathbf{0} &&\text{in }\Omega_p\times(0,T], \label{eq:biot_kinematic}\\
\mu_f\mathbf{K}^{-1}\mathbf{u}_p+\nabla p_p&=\mathbf{f}_{d} &&\text{in }\Omega_p\times(0,T], \label{eq:darcy_law}\\
c_0\partial_t p_p+\alpha\nabla\cdot\boldsymbol{\xi}_p+\nabla\cdot\mathbf{u}_p&=\phi_d &&\text{in }\Omega_p\times(0,T]. \label{eq:biot_mass}
\end{align}
Here, $\rho_p>0$ is the density of the poroelastic medium, $c_0\geq0$ is the constrained specific storage coefficient, $\alpha\in(0,1]$ is the Biot--Willis coefficient, $\mathbf{f}_p$ is the body force acting on the poroelastic skeleton, $\mathbf{f}_{d}$ represents the body force acting on the fluid, and $\phi_d$ is a prescribed volumetric source term in the fluid mass
balance equation. The permeability tensor $\mathbf{K}$ is assumed to be symmetric and uniformly positive definite: there exist constants $0<k_{\min}\leq k_{\max}$ such that
\begin{equation*}
k_{\min}|\mathbf{z}|^2\leq \mathbf{z}\cdot\mathbf{K}(\mathbf{x})\mathbf{z}\leq k_{\max}|\mathbf{z}|^2
\qquad\text{for all }\mathbf{z}\in\mathbb{R}^d\text{ and a.e. }\mathbf{x}\in\Omega_p.
\end{equation*}
The total poroelastic stress tensor is given by
\begin{equation*}
\boldsymbol{\sigma}_p(\boldsymbol{\eta}_p,p_p)=\boldsymbol{\sigma}_e(\boldsymbol{\eta}_p)-\alpha p_p\mathbf{I},
\label{eq:poroelastic_total_stress}
\end{equation*}
where the elastic stress tensor is defined by
$
\boldsymbol{\sigma}_e(\boldsymbol{\eta}_p)=2\mu_p\boldsymbol{\varepsilon}(\boldsymbol{\eta}_p)+\lambda_p(\nabla\cdot\boldsymbol{\eta}_p)\mathbf{I}.
$
The Lam\'e coefficients $\lambda_p$ and $\mu_p$ are related to the Young modulus $E$ and the Poisson ratio $\nu$ through
\begin{equation*}
\lambda_p=\frac{E\nu}{(1+\nu)(1-2\nu)},
\qquad
\mu_p=\frac{E}{2(1+\nu)}.
\label{eq:lame_coefficients}
\end{equation*}

\subsection{Interface and boundary conditions}
\label{subsec:interface_boundary_conditions}

The Stokes and Biot systems are coupled through conservation of mass, balance of forces, continuity of the normal fluid stress, and the Beavers--Joseph--Saffman condition~\cite{badia2009coupling,  carraro2013pressure, showalter2005poroelastic}. Conservation of mass across the interface is expressed as
\begin{equation}
\mathbf{u}_f\cdot\mathbf{n}_f+(\boldsymbol{\xi}_p+\mathbf{u}_p)\cdot\mathbf{n}_p=0
\qquad\text{on }\Gamma\times(0,T].
\label{eq:interface_mass}
\end{equation}
The balance of forces is given by
\begin{equation}
\boldsymbol{\sigma}_f(\mathbf{u}_f,p_f)\mathbf{n}_f+\boldsymbol{\sigma}_p(\boldsymbol{\eta}_p,p_p)\mathbf{n}_p=\mathbf{0}
\qquad\text{on }\Gamma\times(0,T].
\label{eq:interface_traction}
\end{equation}
For the normal components, the fluid traction and the pore pressure satisfy
\begin{equation}
-\boldsymbol{\sigma}_f(\mathbf{u}_f,p_f)\mathbf{n}_f\cdot\mathbf{n}_f=p_p
\qquad\text{on }\Gamma\times(0,T].
\label{eq:interface_normal_stress}
\end{equation}
Finally, the tangential slip between the free fluid and the poroelastic skeleton is described by
\begin{equation}
-\boldsymbol{\sigma}_f(\mathbf{u}_f,p_f)
\mathbf{n}_f\cdot\boldsymbol{\tau}_j
=
\gamma_{\mathrm{BJS},j}
(\mathbf{u}_f-\boldsymbol{\xi}_p)
\cdot\boldsymbol{\tau}_j
\qquad
\text{on }\Gamma\times(0,T],
\quad j=1,\ldots,d-1, \label{tangential slip}
\end{equation} 
where $\gamma_{\mathrm{BJS},j}>0$ is the Beavers--Joseph--Saffman friction coefficient and $\boldsymbol{\tau}_j$ denotes an orthonormal set of unit tangent vectors on $\Gamma$. A commonly used expression is
\begin{equation*}
\gamma_{\mathrm{BJS},j}=\frac{\alpha_{\mathrm{BJS}}\mu_f}{\sqrt{\boldsymbol{\tau}_j\cdot\mathbf{K}\boldsymbol{\tau}_j}},
\qquad j=1,\ldots,d-1,
\label{eq:bjs_coefficient}
\end{equation*}
where $d\in\{2,3\}$ denotes the spatial dimension, $j=1,\ldots,d-1$
indexes the tangential directions, and $\alpha_{\mathrm{BJS}}>0$ is a
dimensionless, experimentally determined parameter.

The Stokes--Biot system \eqref{eq:stokes_momentum}--\eqref{tangential slip} must be supplemented with appropriate boundary and initial conditions. To this end, we decompose the exterior boundaries as
\begin{equation*}
\partial\Omega_f\setminus\Gamma=\Gamma_f^D\cup\Gamma_f^N = \Gamma_f,
\qquad
\partial\Omega_p\setminus\Gamma=\Gamma_p^D\cup\Gamma_p^N
=\widetilde{\Gamma}_p^D\cup\widetilde{\Gamma}_p^N = \Gamma_p,
\label{eq:boundary_decomposition}
\end{equation*}
where $\Gamma_p^D\cup\Gamma_p^N$ is associated with the poroelastic skeleton and $\widetilde{\Gamma}_p^D\cup\widetilde{\Gamma}_p^N$ is associated with the Darcy flow. Accordingly, we impose the following boundary conditions for the Stokes--Biot system:
\begin{align*}
\mathbf{u}_f
&=\mathbf{g}_f
\quad\text{on }\Gamma_f^D\times(0,T],
&
\boldsymbol{\sigma}_f(\mathbf{u}_f,p_f)\mathbf{n}_f
&=\mathbf{t}_f
\quad\text{on }\Gamma_f^N\times(0,T],
\\
\boldsymbol{\eta}_p
&=\mathbf{g}_p
\quad\text{on }\Gamma_p^D\times(0,T],
&
\boldsymbol{\sigma}_p(\boldsymbol{\eta}_p,p_p)\mathbf{n}_p
&=\mathbf{t}_p
\quad\text{on }\Gamma_p^N\times(0,T],
\\
p_p
&=p_D
\quad\text{on }\widetilde{\Gamma}_p^D\times(0,T],
&
\mathbf{u}_p\cdot\mathbf{n}_p
&=q_N
\quad\text{on }\widetilde{\Gamma}_p^N\times(0,T].
\end{align*}
The system is further equipped with the initial conditions
\begin{equation*}
\mathbf{u}_f(\cdot,0)=\mathbf{u}_f^0,
\qquad
\boldsymbol{\eta}_p(\cdot,0)=\boldsymbol{\eta}_p^0,
\qquad
\boldsymbol{\xi}_p(\cdot,0)=\boldsymbol{\xi}_p^0,
\qquad
p_p(\cdot,0)=p_p^0.
\label{eq:initial_conditions}
\end{equation*}

\subsection{Variational formulation}
\label{subsec:variational_formulation}

For the sake of notation, we assume homogeneous boundary data in the variational formulation; nonhomogeneous data can be incorporated through standard lifting arguments. We introduce the spaces
\begin{align*}
\mathbf{V}_f&=\left\{\mathbf{v}_f\in\mathbf{H}^1(\Omega_f):\mathbf{v}_f=\mathbf{0}\text{ on }\Gamma_f^D\right\}, \qquad
Q_f=L_0^2(\Omega_f), \\
\mathbf{V}_p&=\{\boldsymbol{\omega}_p\in\mathbf{H}^1(\Omega_p):\boldsymbol{\omega}_p=\mathbf{0}\text{ on }\Gamma_p^D\}, \\
\mathbf{X}_p&=\left\{\mathbf{v}_p\in\mathbf{H}(\operatorname{div};\Omega_p):(\mathbf{v}_p\cdot\mathbf{n}_p)|_{\Gamma}\in L^2(\Gamma),\ \mathbf{v}_p\cdot\mathbf{n}_p=0\text{ on }\widetilde{\Gamma}_p^N\right\}, \\
Q_p&=\{q_p\in H^1(\Omega_p):q_p=0\text{ on }\widetilde{\Gamma}_p^D\}.
\end{align*}
If the fluid pressure is fixed by the boundary conditions, $Q_f$ may instead be taken as $L^2(\Omega_f)$. The $H^1$ regularity of $Q_p$ guarantees that the pressure trace on $\Gamma$ is well defined.

For a domain $D$, the $L^2(D)$ inner product is denoted by $(\cdot,\cdot)_D$, and the $L^2(\Gamma)$ interface pairing is denoted by $\langle\cdot,\cdot\rangle_\Gamma$. We define
\begin{align*}
a_f(\mathbf{u}_f,\mathbf{v}_f)&=2\mu_f\big(\boldsymbol{\varepsilon}(\mathbf{u}_f),\boldsymbol{\varepsilon}(\mathbf{v}_f)\big)_{\Omega_f}, \quad a_d(\mathbf{u}_p,\mathbf{v}_p)=\mu_f\big(\mathbf{K}^{-1}\mathbf{u}_p,\mathbf{v}_p\big)_{\Omega_p}, \\
 a_p(\boldsymbol{\eta}_p,\boldsymbol{\omega}_p)&=2\mu_p\big(\boldsymbol{\varepsilon}(\boldsymbol{\eta}_p),\boldsymbol{\varepsilon}(\boldsymbol{\omega}_p)\big)_{\Omega_p}+\lambda_p\big(\nabla\cdot\boldsymbol{\eta}_p,\nabla\cdot\boldsymbol{\omega}_p\big)_{\Omega_p}, \\
b_f(\mathbf{v}_f,q_f)&=-\big(\nabla\cdot\mathbf{v}_f,q_f\big)_{\Omega_f},\quad b_p(\mathbf{v},q_p)=-\big(\nabla\cdot\mathbf{v},q_p\big)_{\Omega_p}, 
\end{align*}
where $b_p$ is defined for $\mathbf{v}\in\mathbf{V}_p$ or $\mathbf{X}_p$ and $q_p\in Q_p$. The normal and tangential interface forms are defined by
\begin{align*}
c_\Gamma(\mathbf{v}_f,\boldsymbol{\omega}_p,\mathbf{v}_p;q_p)
&=\left\langle\mathbf{v}_f\cdot\mathbf{n}_f+(\boldsymbol{\omega}_p+\mathbf{v}_p)\cdot\mathbf{n}_p,q_p\right\rangle_\Gamma,\nonumber \\
a_{\mathrm{BJS}}(\mathbf{u}_f,\boldsymbol{\xi}_p;\mathbf{v}_f,\boldsymbol{\omega}_p)
&=\sum_{j=1}^{d-1}\left\langle\gamma_{\mathrm{BJS},j}(\mathbf{u}_f-\boldsymbol{\xi}_p)\cdot\boldsymbol{\tau}_j,(\mathbf{v}_f-\boldsymbol{\omega}_p)\cdot\boldsymbol{\tau}_j\right\rangle_\Gamma. 
\end{align*}
We further introduce the bulk operators
\begin{align*}
\mathcal{A}_f\big((\mathbf{u}_f,p_f),(\mathbf{v}_f,q_f)\big)&=\rho_f(\partial_t\mathbf{u}_f,\mathbf{v}_f)_{\Omega_f}+a_f(\mathbf{u}_f,\mathbf{v}_f)+b_f(\mathbf{v}_f,p_f)\\
&\quad-b_f(\mathbf{u}_f,q_f),\\
\mathcal{A}_p\big((\boldsymbol{\eta}_p,\boldsymbol{\xi}_p,\mathbf{u}_p,p_p),(\boldsymbol{\omega}_p,\boldsymbol{\phi}_p,\mathbf{v}_p,q_p)\big)
&=\rho_p(\partial_t\boldsymbol{\xi}_p,\boldsymbol{\omega}_p)_{\Omega_p}+a_p(\boldsymbol{\eta}_p,\boldsymbol{\omega}_p)+\alpha b_p(\boldsymbol{\omega}_p,p_p) \nonumber\\
&\quad+(\partial_t\boldsymbol{\eta}_p-\boldsymbol{\xi}_p,\boldsymbol{\phi}_p)_{\Omega_p} 
+a_d(\mathbf{u}_p,\mathbf{v}_p)+b_p(\mathbf{v}_p,p_p)\nonumber\\
&\quad+c_0(\partial_t p_p,q_p)_{\Omega_p}-\alpha b_p(\boldsymbol{\xi}_p,q_p)-b_p(\mathbf{u}_p,q_p), 
\end{align*}
and the forcing functionals
\begin{align*}
\mathcal{F}_f(\mathbf{v}_f,q_f)&=(\mathbf{f}_f,\mathbf{v}_f)_{\Omega_f}+(\phi_f,q_f)_{\Omega_f}, \\
\mathcal{F}_p(\boldsymbol{\omega}_p, \mathbf{v}_{p},q_p)&=(\mathbf{f}_p,\boldsymbol{\omega}_p)_{\Omega_p}+(\mathbf{f}_{d},\mathbf{v}_p)_{\Omega_p}+(\phi_d,q_p)_{\Omega_p}. 
\end{align*}
The fully coupled variational problem reads as follows: for a.e. $t\in(0,T]$, find
$
(\mathbf{u}_f,p_f,\boldsymbol{\eta}_p,\boldsymbol{\xi}_p,\mathbf{u}_p,p_p)
\in\mathbf{V}_f\times Q_f\times\mathbf{V}_p\times\mathbf{V}_p\times\mathbf{X}_p\times Q_p
$
such that, for every
$
(\mathbf{v}_f,q_f,\boldsymbol{\omega}_p,\boldsymbol{\phi}_p,\\
\mathbf{v}_p,q_p)
\in\mathbf{V}_f\times Q_f\times\mathbf{V}_p\times\mathbf{V}_p\times\mathbf{X}_p\times Q_p,
$
\begin{align*}
&\mathcal{A}_f\big((\mathbf{u}_f,p_f),(\mathbf{v}_f,q_f)\big)
+\mathcal{A}_p\big((\boldsymbol{\eta}_p,\boldsymbol{\xi}_p,\mathbf{u}_p,p_p),(\boldsymbol{\omega}_p,\boldsymbol{\phi}_p,\mathbf{v}_p,q_p)\big) \nonumber\\
&\quad+c_\Gamma(\mathbf{v}_f,\boldsymbol{\omega}_p,\mathbf{v}_p;p_p)
+a_{\mathrm{BJS}}(\mathbf{u}_f,\boldsymbol{\xi}_p;\mathbf{v}_f,\boldsymbol{\omega}_p)
-c_\Gamma(\mathbf{u}_f,\boldsymbol{\xi}_p,\mathbf{u}_p;q_p) \nonumber\\
&=\mathcal{F}_f(\mathbf{v}_f,q_f)+\mathcal{F}_p(\boldsymbol{\omega}_p, q_p).
\end{align*}
The first normal interface term results from the fluid, structure, and Darcy boundary contributions, whereas the second weakly imposes the interface mass-conservation condition. Their skew-symmetric arrangement is convenient for the energy analysis.

Unlike the fluid pressure, the pore pressure is taken in
$H^1(\Omega_p)$ so that its trace on the fluid--poroelastic interface
is well defined, as required by the coupling terms. This choice is
consistent with the variational settings used in
\cite{cesmelioglu2017analysis,showalter2005poroelastic}. The
well-posedness of related coupled Stokes--Biot systems has been
established in these works. In particular, for the linear Stokes case
considered here, the existence and uniqueness result is global in time
and does not require the small-data restriction associated with the
Navier--Stokes nonlinearity. We therefore assume in what follows that
the coupled problem admits a unique sufficiently regular solution.



\section{Robin reformulation and partitioned weak formulation}
\label{sec:robin_weak_form}

In this section, we reformulate the original interface conditions of
the coupled Stokes--Biot problem introduced in Section~\ref{sec:fully_coupled_model}
 into Robin-type transmission
conditions and introduce an auxiliary interface variable, with the aim
of constructing an efficient partitioned solution strategy.

\subsection{Robin transmission conditions}
\label{subsec:robin_conditions}

Let $L>0$ be a given Robin parameter. For clarity, a common parameter is used for all normal Robin conditions; the formulation can be extended to distinct parameters for the fluid, structure, and Darcy conditions. Using \eqref{eq:interface_mass}--\eqref{eq:interface_normal_stress}, the normal coupling conditions can be equivalently rewritten as
\begin{align*}
L\mathbf{u}_f\cdot\mathbf{n}_f+\boldsymbol{\sigma}_f(\mathbf{u}_f,p_f)\mathbf{n}_f\cdot\mathbf{n}_f
&=-L(\boldsymbol{\xi}_p+\mathbf{u}_p)\cdot\mathbf{n}_p+\boldsymbol{\sigma}_p(\boldsymbol{\eta}_p,p_p)\mathbf{n}_p\cdot\mathbf{n}_p, \\
L(\boldsymbol{\xi}_p+\mathbf{u}_p)\cdot\mathbf{n}_p+\boldsymbol{\sigma}_p(\boldsymbol{\eta}_p,p_p)\mathbf{n}_p\cdot\mathbf{n}_p
&=-L\mathbf{u}_f\cdot\mathbf{n}_f+\boldsymbol{\sigma}_f(\mathbf{u}_f,p_f)\mathbf{n}_f\cdot\mathbf{n}_f,\\
L(\boldsymbol{\xi}_p+\mathbf{u}_p)\cdot\mathbf{n}_p-p_p
&=-L\mathbf{u}_f\cdot\mathbf{n}_f+\boldsymbol{\sigma}_f(\mathbf{u}_f,p_f)\mathbf{n}_f\cdot\mathbf{n}_f.
\end{align*}
Indeed, the traction balance and normal-stress condition imply
\begin{equation*}
\boldsymbol{\sigma}_p(\boldsymbol{\eta}_p,p_p)\mathbf{n}_p\cdot\mathbf{n}_p=-p_p
\qquad\text{on }\Gamma.
\label{eq:poroelastic_normal_stress}
\end{equation*}
The  tangential component of the traction balance \eqref{tangential slip} yield
\begin{align*}
\gamma_{\mathrm{BJS},j}\mathbf{u}_f\cdot\boldsymbol{\tau}_j+\boldsymbol{\sigma}_f(\mathbf{u}_f,p_f)\mathbf{n}_f\cdot\boldsymbol{\tau}_j
&=\gamma_{\mathrm{BJS},j}\boldsymbol{\xi}_p\cdot\boldsymbol{\tau}_j, \\
\gamma_{\mathrm{BJS},j}\boldsymbol{\xi}_p\cdot\boldsymbol{\tau}_j+\boldsymbol{\sigma}_p(\boldsymbol{\eta}_p,p_p)\mathbf{n}_p\cdot\boldsymbol{\tau}_j
&=\gamma_{\mathrm{BJS},j}\mathbf{u}_f\cdot\boldsymbol{\tau}_j, 
\end{align*}
for $j=1,\ldots,d-1$.
\begin{remark}
The reformulation of the original coupling conditions into Robin-type
transmission conditions through suitable linear combinations has been
widely used in partitioned methods for coupled problems, including
Stokes--Darcy systems, FSI, and FPSI problems; see, for example,
\cite{burman2022fully,he2026locking, sun2021domain}. The equivalence between the reformulated
problem and the original coupled formulation has been discussed in
related settings in \cite{he2026locking,sun2021domain}. The corresponding equivalence
for the present formulation can be established by analogous arguments
and is therefore omitted here.
\end{remark}
\subsection{Auxiliary interface variable}
\label{subsec:auxiliary_variable}

We introduce an auxiliary scalar interface variable $\theta\in L^2(\Gamma)$ by
\begin{equation}
\theta=-L\mathbf{u}_f\cdot\mathbf{n}_f-p_p=L(\boldsymbol{\xi}_p+\mathbf{u}_p)\cdot\mathbf{n}_p-p_p,
\label{eq:theta_definition}
\end{equation}
where the second identity follows from \eqref{eq:interface_mass}. If the normal traction possesses an $L^2(\Gamma)$ trace, then \eqref{eq:interface_normal_stress} also gives
\begin{align*}
\theta=-L\mathbf{u}_f\cdot\mathbf{n}_f+\boldsymbol{\sigma}_f(\mathbf{u}_f,p_f)\mathbf{n}_f\cdot\mathbf{n}_f.
\end{align*}
Consequently, the normal Robin conditions become
\begin{align*}
L\mathbf{u}_f\cdot\mathbf{n}_f+\boldsymbol{\sigma}_f(\mathbf{u}_f,p_f)\mathbf{n}_f\cdot\mathbf{n}_f
&=\theta-2L(\boldsymbol{\xi}_p+\mathbf{u}_p)\cdot\mathbf{n}_p,\\
L(\boldsymbol{\xi}_p+\mathbf{u}_p)\cdot\mathbf{n}_p+\boldsymbol{\sigma}_p(\boldsymbol{\eta}_p,p_p)\mathbf{n}_p\cdot\mathbf{n}_p
&=\theta, \\
L(\boldsymbol{\xi}_p+\mathbf{u}_p)\cdot\mathbf{n}_p-p_p
&=\theta. 
\end{align*}
Thus, $\theta$ is the common normal Robin datum transmitted between the fluid and poroelastic subproblems. The original interface conditions and the Robin conditions above are equivalent whenever the stated traces are well defined.

We emphasize that $\theta$ should not be interpreted as a Lagrange
multiplier. Unlike a Lagrange multiplier, which is introduced as an
independent unknown to enforce an interface constraint, $\theta$ is
defined algebraically from the physical interface quantities through
\eqref{eq:theta_definition}. It introduces no additional constraint
into the continuous problem and serves only as an auxiliary Robin datum
for transferring normal interface information between the two
subproblems.

\subsection{Robin-split weak formulation with consistent Darcy residual stabilization}
\label{subsec:robin_split_weak_form}

We use the function spaces, bulk operators, and forcing functionals introduced in Section~\ref{sec:fully_coupled_model}. We only define the additional Robin interface forms. For the fluid subproblem, let
\begin{align*}
\mathcal{R}_f(\mathbf{u}_f,\mathbf{v}_f)
&=L\left\langle\mathbf{u}_f\cdot\mathbf{n}_f,\mathbf{v}_f\cdot\mathbf{n}_f\right\rangle_\Gamma+a_{\mathrm{BJS}}(\mathbf{u}_f,\mathbf{0};\mathbf{v}_f,\mathbf{0}), \\
\mathcal{C}_f(\theta,\boldsymbol{\xi}_p,\mathbf{u}_p;\mathbf{v}_f)
&=\left\langle\theta-2L(\boldsymbol{\xi}_p+\mathbf{u}_p)\cdot\mathbf{n}_p,\mathbf{v}_f\cdot\mathbf{n}_f\right\rangle_\Gamma
-a_{\mathrm{BJS}}(\mathbf{0},\boldsymbol{\xi}_p;\mathbf{v}_f,\mathbf{0}). 
\end{align*}
The fluid Robin subproblem is: find $(\mathbf{u}_f,p_f)\in\mathbf{V}_f\times Q_f$ such that
\begin{equation*}
\mathcal{A}_f\big((\mathbf{u}_f,p_f),(\mathbf{v}_f,q_f)\big)+\mathcal{R}_f(\mathbf{u}_f,\mathbf{v}_f)
=\mathcal{F}_f(\mathbf{v}_f,q_f)+\mathcal{C}_f(\theta,\boldsymbol{\xi}_p,\mathbf{u}_p;\mathbf{v}_f)
\end{equation*}
for every $(\mathbf{v}_f,q_f)\in\mathbf{V}_f\times Q_f$.

For the poroelastic subproblem, define
\begin{align*}
\mathcal{R}_p\big((\boldsymbol{\xi}_p,\mathbf{u}_p);(\boldsymbol{\omega}_p,\mathbf{v}_p)\big)
&=L\left\langle(\boldsymbol{\xi}_p+\mathbf{u}_p)\cdot\mathbf{n}_p,(\boldsymbol{\omega}_p+\mathbf{v}_p)\cdot\mathbf{n}_p\right\rangle_\Gamma
+a_{\mathrm{BJS}}(\mathbf{0},\boldsymbol{\xi}_p;\mathbf{0},\boldsymbol{\omega}_p), \\
\mathcal{C}_p(\theta,\mathbf{u}_f;\boldsymbol{\omega}_p,\mathbf{v}_p)
&=\left\langle\theta,(\boldsymbol{\omega}_p+\mathbf{v}_p)\cdot\mathbf{n}_p\right\rangle_\Gamma
-a_{\mathrm{BJS}}(\mathbf{u}_f,\mathbf{0};\mathbf{0},\boldsymbol{\omega}_p). 
\end{align*}
The pore pressure requires an $H^1(\Omega_p)$ trace at the interface, whereas the unstabilized Darcy block uses an $\mathbf H(\operatorname{div};\Omega_p)$-conforming velocity space. To permit standard $H^1$-conforming finite element spaces for both Darcy variables, for $(\mathbf{u}_p,p_p),(\mathbf{v}_p,q_p)\in\mathbf{X}_p\times Q_p$, we introduce the consistent Darcy residual stabilization term
\begin{align*}
s_D\big((\mathbf{u}_p,p_p),(\mathbf{v}_p,q_p)\big)
&=\frac{1}{2}\left(-\mu_f\mathbf{K}^{-1}\mathbf{v}_p+\nabla q_p,\,\mu_f^{-1}\mathbf{K}\big(\mu_f\mathbf{K}^{-1}\mathbf{u}_p+\nabla p_p\big)\right)_{\Omega_p}.
\end{align*}
The residual-based stabilization employed here was originally introduced for the Darcy problem in~\cite{masud2002stabilized}. To the best of our knowledge, this stabilization technique has not previously been applied to the Biot problem in this coupling system. The stabilized poroelastic bulk operator and forcing functional are given by
\begin{align*}
&\mathcal{A}_{p,\mathrm{stab}}\big((\boldsymbol{\eta}_p,\boldsymbol{\xi}_p,\mathbf{u}_p,p_p),(\boldsymbol{\omega}_p,\boldsymbol{\phi}_p,\mathbf{v}_p,q_p)\big)\nonumber\\
&=\mathcal{A}_p\big((\boldsymbol{\eta}_p,\boldsymbol{\xi}_p,\mathbf{u}_p,p_p),(\boldsymbol{\omega}_p,\boldsymbol{\phi}_p,\mathbf{v}_p,q_p)\big) +s_D\big((\mathbf{u}_p,p_p),(\mathbf{v}_p,q_p)\big), \\
&\mathcal{F}_{p,\mathrm{stab}}(\boldsymbol{\omega}_p,\mathbf{v}_p,q_p)
=\mathcal{F}_p(\boldsymbol{\omega}_p,\mathbf{v}_p,q_p)+\frac{1}{2}\left(-\mu_f\mathbf{K}^{-1}\mathbf{v}_p+\nabla q_p,\,\mu_f^{-1}\mathbf{K}\mathbf{f}_d\right)_{\Omega_p}. 
\end{align*}
The stabilized poroelastic Robin subproblem is: find
$
(\boldsymbol{\eta}_p,\boldsymbol{\xi}_p,\mathbf{u}_p,p_p)
\in\mathbf{V}_p\times\mathbf{V}_p\times\mathbf{X}_p\times Q_p
$
such that
\begin{align*}
&\mathcal{A}_{p,\mathrm{stab}}\big((\boldsymbol{\eta}_p,\boldsymbol{\xi}_p,\mathbf{u}_p,p_p),(\boldsymbol{\omega}_p,\boldsymbol{\phi}_p,\mathbf{v}_p,q_p)\big)
+\mathcal{R}_p\big((\boldsymbol{\xi}_p,\mathbf{u}_p);(\boldsymbol{\omega}_p,\mathbf{v}_p)\big) \nonumber\\
&\qquad=\mathcal{F}_{p,\mathrm{stab}}(\boldsymbol{\omega}_p,\mathbf{v}_p,q_p)+\mathcal{C}_p(\theta,\mathbf{u}_f;\boldsymbol{\omega}_p,\mathbf{v}_p),
\end{align*}
for every $(\boldsymbol{\omega}_p,\boldsymbol{\phi}_p,\mathbf{v}_p,q_p)\in\mathbf{V}_p\times\mathbf{V}_p\times\mathbf{X}_p\times Q_p$. At the continuous level, the fluid and stabilized poroelastic Robin subproblems remain coupled through their interface data; the time-discrete schemes below evaluate these data explicitly from previous time steps.

Note that the stabilization is consistent because the exact Darcy solution satisfies
\begin{equation*}
\mu_f\mathbf{K}^{-1}\mathbf{u}_p+\nabla p_p=\mathbf{f}_d
\qquad\text{in }\Omega_p.
\label{eq:darcy_residual_identity}
\end{equation*}
Moreover, choosing $(\mathbf{v}_p,q_p)=(\mathbf{u}_p,p_p)$ gives
\begin{align*}
&a_d(\mathbf{u}_p,\mathbf{u}_p)+b_p(\mathbf{u}_p,p_p)-b_p(\mathbf{u}_p,p_p)+s_D\big((\mathbf{u}_p,p_p),(\mathbf{u}_p,p_p)\big)\\
 &=\frac{\mu_f}{2}\|\mathbf{K}^{-1/2}\mathbf{u}_p\|_{\Omega_p}^2+\frac{1}{2\mu_f}\|\mathbf{K}^{1/2}\nabla p_p\|_{\Omega_p}^2.
\end{align*}
Thus, the stabilized formulation controls both the Darcy velocity and the pressure gradient and permits standard continuous finite element subspaces of $\mathbf{X}_{p}\times Q_{p}$, without requiring an $\mathbf{H}(\operatorname{div})$-conforming discrete velocity space. For $\mathbf{f}_d=\rho_f\mathbf{g}$, its additional forcing term is
$$
\frac{1}{2}\left(-\mu_f\mathbf{K}^{-1}\mathbf{v}_p+\nabla q_p,\,\mu_f^{-1}\mathbf{K}\rho_f\mathbf{g}\right)_{\Omega_p}.
$$

\section{Fully discrete Robin partitioned schemes}
\label{sec:fully_discrete_schemes}

In this section, we introduce the spatial and temporal discretizations shared by the first- and second-order schemes. We first derive the time-discrete Robin transmission conditions and the corresponding update formulas for the auxiliary interface variable. The first- and second-order schemes, together with their stability analysis, will be presented in Subsections~\ref{subsec:first_order_scheme} and \ref{subsec:second_order_scheme}, respectively.

Let $\mathcal{T}_{f,h}$ and $\mathcal{T}_{p,h}$ be shape-regular and quasi-uniform simplicial triangulations of $\Omega_f$ and $\Omega_p$, respectively, consisting of triangles when $d=2$ and tetrahedra when $d=3$. We denote the corresponding mesh sizes by $h_f$ and $h_p$, and set $h=\max\{h_f,h_p\}$. The two triangulations are assumed to be matching on $\Gamma$, and the induced interface triangulation is denoted by $\mathcal{T}_{\Gamma,h}$.

For an integer $k\geq1$, we introduce the conforming finite element spaces
\begin{align*}
\mathbf{V}_{f,h}
&=\left\{\mathbf{v}_{f,h}\in\mathbf{V}_f\cap[C^0(\overline{\Omega}_f)]^d:
\mathbf{v}_{f,h}|_K\in[\mathbb{P}_{k+1}(K)]^d
\quad\forall K\in\mathcal{T}_{f,h}\right\}, \\
Q_{f,h}
&=\left\{q_{f,h}\in Q_f\cap C^0(\overline{\Omega}_f):
q_{f,h}|_K\in\mathbb{P}_k(K)
\quad\forall K\in\mathcal{T}_{f,h}\right\}, \\
\mathbf{V}_{p,h}
&=\left\{\boldsymbol{\omega}_{p,h}\in\mathbf{V}_p\cap[C^0(\overline{\Omega}_p)]^d:
\boldsymbol{\omega}_{p,h}|_K\in[\mathbb{P}_k(K)]^d
\quad\forall K\in\mathcal{T}_{p,h}\right\}, \\
\mathbf{X}_{p,h}
&=\left\{\mathbf{v}_{p,h}\in\mathbf{X}_p\cap[C^0(\overline{\Omega}_p)]^d:
\mathbf{v}_{p,h}|_K\in[\mathbb{P}_k(K)]^d
\quad\forall K\in\mathcal{T}_{p,h}\right\}, \\
Q_{p,h}
&=\left\{q_{p,h}\in Q_p\cap C^0(\overline{\Omega}_p):
q_{p,h}|_K\in\mathbb{P}_k(K)
\quad\forall K\in\mathcal{T}_{p,h}\right\}. 
\end{align*}
Thus, the Taylor--Hood pair is used for the fluid velocity and pressure, whereas continuous piecewise polynomial spaces of degree $k$ are used for the structure displacement, structure velocity, Darcy velocity, and pore pressure. We assume that $\mathbf{V}_{f,h}\times Q_{f,h}$ satisfies the discrete inf--sup condition
\begin{equation}
\inf_{q_{f,h}\in Q_{f,h}}\sup_{\mathbf{v}_{f,h}\in\mathbf{V}_{f,h}}
\frac{b_f(\mathbf{v}_{f,h},q_{f,h})}
{\|\mathbf{v}_{f,h}\|_{H^1(\Omega_f)}\|q_{f,h}\|_{L^2(\Omega_f)}}
\geq\beta_0,
\label{eq:discrete_fluid_inf_sup}
\end{equation}
where $\beta_0>0$ is independent of $h$. Owing to the residual stabilization introduced in Subsection~\ref{subsec:robin_split_weak_form}, no discrete Darcy inf--sup condition is imposed on $\mathbf{X}_{p,h}\times Q_{p,h}$; see \cite{masud2002stabilized} for further details.

To approximate the auxiliary interface variable, we define
\begin{align*}
\Sigma_h
=\left\{\psi_h\in L^2(\Gamma):
\psi_h|_e\in\mathbb{P}_k(e)
\quad\forall e\in\mathcal{T}_{\Gamma,h}\right\}.
\end{align*}
We also denote by $\Pi_\Sigma:L^2(\Gamma)\rightarrow\Sigma_h$ the $L^2(\Gamma)$-orthogonal projection satisfying
\begin{equation}
\left\langle\Pi_\Sigma\varphi-\varphi,\psi_h\right\rangle_\Gamma=0,
\qquad\forall\psi_h\in\Sigma_h.
\label{eq:interface_projection}
\end{equation}
Since the fluid velocity is approximated by polynomials of degree $k+1$, its normal trace does not necessarily belong to $\Sigma_h$. Therefore, the update of the auxiliary interface variable will be imposed in the projected weak form \eqref{eq:interface_projection}.

Let $N\in\mathbb{N}$, $\Delta t=T/N$, and $t^n=n\Delta t$ for $n=0,\ldots,N$. For a time-dependent function $z$, we write $z^n=z(t^n)$. The backward Euler and BDF2 difference operators are defined by
\begin{align*}
d_t z^{n+1}
&=\frac{z^{n+1}-z^n}{\Delta t},
\qquad n\geq0, \\
d_t^{(2)}z^{n+1}
&=\frac{3z^{n+1}-4z^n+z^{n-1}}{2\Delta t},
\qquad n\geq1. 
\end{align*}

The spatial bilinear forms introduced in Sections~\ref{sec:fully_coupled_model} and \ref{sec:robin_weak_form} are restricted to the corresponding finite element spaces without changing their notation. For $\delta_t=d_t$ or $\delta_t=d_t^{(2)}$, we define the discrete fluid bulk operator by
\begin{align*}
\mathcal{A}_{f,h}^{\delta_t}
\big((\mathbf{u}_{f,h}^{n+1},p_{f,h}^{n+1}),
(\mathbf{v}_{f,h},q_{f,h})\big)
&=\rho_f(\delta_t\mathbf{u}_{f,h}^{n+1},\mathbf{v}_{f,h})_{\Omega_f}
+a_f(\mathbf{u}_{f,h}^{n+1},\mathbf{v}_{f,h}) \nonumber\\
&\quad+b_f(\mathbf{v}_{f,h},p_{f,h}^{n+1})
-b_f(\mathbf{u}_{f,h}^{n+1},q_{f,h}).
\end{align*}
The stabilized poroelastic bulk operator is defined by
\begin{align*}
&\mathcal{A}_{p,h,\mathrm{stab}}^{\delta_t}
\big((\boldsymbol{\eta}_{p,h}^{n+1},\boldsymbol{\xi}_{p,h}^{n+1},
\mathbf{u}_{p,h}^{n+1},p_{p,h}^{n+1}),
(\boldsymbol{\omega}_{p,h},\boldsymbol{\phi}_{p,h},
\mathbf{v}_{p,h},q_{p,h})\big) \nonumber\\
&\quad=\rho_p(\delta_t\boldsymbol{\xi}_{p,h}^{n+1},
\boldsymbol{\omega}_{p,h})_{\Omega_p}
+a_p(\boldsymbol{\eta}_{p,h}^{n+1},\boldsymbol{\omega}_{p,h})
+\alpha b_p(\boldsymbol{\omega}_{p,h},p_{p,h}^{n+1}) \nonumber\\
&\qquad+(\delta_t\boldsymbol{\eta}_{p,h}^{n+1}
-\boldsymbol{\xi}_{p,h}^{n+1},\boldsymbol{\phi}_{p,h})_{\Omega_p}
+a_d(\mathbf{u}_{p,h}^{n+1},\mathbf{v}_{p,h})
+b_p(\mathbf{v}_{p,h},p_{p,h}^{n+1}) \nonumber\\
&\qquad+c_0(\delta_t p_{p,h}^{n+1},q_{p,h})_{\Omega_p}
-\alpha b_p(\boldsymbol{\xi}_{p,h}^{n+1},q_{p,h})
-b_p(\mathbf{u}_{p,h}^{n+1},q_{p,h}) \nonumber\\
&\qquad+s_D\big((\mathbf{u}_{p,h}^{n+1},p_{p,h}^{n+1}),
(\mathbf{v}_{p,h},q_{p,h})\big).
\end{align*}
The corresponding time-dependent forcing functionals are
\begin{align*}
\mathcal{F}_{f,h}^{n+1}(\mathbf{v}_{f,h},q_{f,h})
&=(\mathbf{f}_f^{n+1},\mathbf{v}_{f,h})_{\Omega_f}
+(\phi_f^{n+1},q_{f,h})_{\Omega_f}, \\
\mathcal{F}_{p,h,\mathrm{stab}}^{n+1}
(\boldsymbol{\omega}_{p,h},\mathbf{v}_{p,h},q_{p,h})
&=(\mathbf{f}_p^{n+1},\boldsymbol{\omega}_{p,h})_{\Omega_p}
+(\mathbf{f}_d^{n+1},\mathbf{v}_{p,h})_{\Omega_p}
+(\phi_d^{n+1},q_{p,h})_{\Omega_p} \nonumber\\
&\quad+\frac{1}{2}\left(
-\mu_f\mathbf{K}^{-1}\mathbf{v}_{p,h}+\nabla q_{p,h},
\mu_f^{-1}\mathbf{K}\mathbf{f}_{d}
\right)_{\Omega_p}.
\end{align*}

\subsection{Time-discrete Robin conditions and interface update}
\label{subsec:discrete_robin_update}

We first derive the time-discrete Robin transmission conditions used in the first-order scheme. At time $t^{n+1}$, the normal conditions are written as
\begin{align}
L\mathbf{u}_{f,h}^{n+1}\cdot\mathbf{n}_f
+\boldsymbol{\sigma}_f(\mathbf{u}_{f,h}^{n+1},p_{f,h}^{n+1})
\mathbf{n}_f\cdot\mathbf{n}_f
&=\theta_h^n
-2L(\boldsymbol{\xi}_{p,h}^{n+1}+\mathbf{u}_{p,h}^{n+1})
\cdot\mathbf{n}_p, \nonumber\\
L(\boldsymbol{\xi}_{p,h}^{n+1}+\mathbf{u}_{p,h}^{n+1})
\cdot\mathbf{n}_p
+\boldsymbol{\sigma}_p(\boldsymbol{\eta}_{p,h}^{n+1},p_{p,h}^{n+1})
\mathbf{n}_p\cdot\mathbf{n}_p
&=\theta_h^n, \nonumber\\
L(\boldsymbol{\xi}_{p,h}^{n+1}+\mathbf{u}_{p,h}^{n+1})
\cdot\mathbf{n}_p-p_{p,h}^{n+1}
&=\theta_h^n. \label{eq:first_order_darcy_normal_robin}
\end{align}
The last two conditions are equivalent by the normal-stress relation
$\boldsymbol{\sigma}_p\mathbf{n}_p\cdot\mathbf{n}_p=-p_p$ on $\Gamma$. The tangential conditions are discretized explicitly as
\begin{align*}
\gamma_{\mathrm{BJS},j}\mathbf{u}_{f,h}^{n+1}\cdot\boldsymbol{\tau}_j
+\boldsymbol{\sigma}_f(\mathbf{u}_{f,h}^{n+1},p_{f,h}^{n+1})
\mathbf{n}_f\cdot\boldsymbol{\tau}_j
&=\gamma_{\mathrm{BJS},j}
\boldsymbol{\xi}_{p,h}^{n}\cdot\boldsymbol{\tau}_j,
\\
\gamma_{\mathrm{BJS},j}\boldsymbol{\xi}_{p,h}^{n+1}\cdot\boldsymbol{\tau}_j
+\boldsymbol{\sigma}_p(\boldsymbol{\eta}_{p,h}^{n+1},p_{p,h}^{n+1})
\mathbf{n}_p\cdot\boldsymbol{\tau}_j
&=\gamma_{\mathrm{BJS},j}
\mathbf{u}_{f,h}^{n}\cdot\boldsymbol{\tau}_j,
\end{align*}
for $j=1,\ldots,d-1$.

It remains to update the auxiliary interface variable. From
\eqref{eq:first_order_darcy_normal_robin} and the continuous definition \eqref{eq:theta_definition}, we obtain 
\begin{align*}
\theta_h^{n+1}
=\theta_h^n
-L\Big(
\mathbf{u}_{f,h}^{n+1}\cdot\mathbf{n}_f
+(\boldsymbol{\xi}_{p,h}^{n+1}+\mathbf{u}_{p,h}^{n+1})
\cdot\mathbf{n}_p
\Big).
\end{align*}
Since $\theta_h^{n+1}\in\Sigma_h$, the update is imposed with the projected \eqref{eq:interface_projection} as follows:
\begin{align}
\langle\theta_h^{n+1},\psi_h\rangle_\Gamma
&=\langle\theta_h^n,\psi_h\rangle_\Gamma
-L\langle
\Pi_{\Sigma}\mathbf{u}_{f,h}^{n+1}\cdot\mathbf{n}_f
+(\boldsymbol{\xi}_{p,h}^{n+1}+\mathbf{u}_{p,h}^{n+1})
\cdot\mathbf{n}_p,\psi_h
\rangle_\Gamma
\label{eq:first_order_theta_update}
\end{align}
for every $\psi_h\in\Sigma_h$.

For the second-order scheme, the interface data required by the poroelastic subproblem are approximated by second-order extrapolation. The normal Robin conditions become
\begin{align}
L\mathbf{u}_{f,h}^{n+1}\cdot\mathbf{n}_f
+\boldsymbol{\sigma}_f(\mathbf{u}_{f,h}^{n+1},p_{f,h}^{n+1})
\mathbf{n}_f\cdot\mathbf{n}_f
&=2\theta_h^n-\theta_h^{n-1}
-2L(\boldsymbol{\xi}_{p,h}^{n+1}+\mathbf{u}_{p,h}^{n+1})
\cdot\mathbf{n}_p, \nonumber\\
L(\boldsymbol{\xi}_{p,h}^{n+1}+\mathbf{u}_{p,h}^{n+1})
\cdot\mathbf{n}_p
+\boldsymbol{\sigma}_p(\boldsymbol{\eta}_{p,h}^{n+1},p_{p,h}^{n+1})
\mathbf{n}_p\cdot\mathbf{n}_p
&=2\theta_h^n-\theta_h^{n-1}, \nonumber\\
L(\boldsymbol{\xi}_{p,h}^{n+1}+\mathbf{u}_{p,h}^{n+1})
\cdot\mathbf{n}_p-p_{p,h}^{n+1}
&=2\theta_h^n-\theta_h^{n-1}. \nonumber
\end{align}
The corresponding tangential conditions are
\begin{align*}
\gamma_{\mathrm{BJS},j}\mathbf{u}_{f,h}^{n+1}\cdot\boldsymbol{\tau}_j
+\boldsymbol{\sigma}_f(\mathbf{u}_{f,h}^{n+1},p_{f,h}^{n+1})
\mathbf{n}_f\cdot\boldsymbol{\tau}_j
&=\gamma_{\mathrm{BJS},j}
(\boldsymbol{\xi}_{p,h}^{n+1})\cdot\boldsymbol{\tau}_j,
\\
\gamma_{\mathrm{BJS},j}\boldsymbol{\xi}_{p,h}^{n+1}\cdot\boldsymbol{\tau}_j
+\boldsymbol{\sigma}_p(\boldsymbol{\eta}_{p,h}^{n+1},p_{p,h}^{n+1})
\mathbf{n}_p\cdot\boldsymbol{\tau}_j
&=\gamma_{\mathrm{BJS},j}
(2\mathbf{u}_{f,h}^{n}-\mathbf{u}_{f,h}^{n-1})
\cdot\boldsymbol{\tau}_j,
\end{align*}
for $j=1,\ldots,d-1$.

 Furthermore, similar to \eqref{eq:first_order_theta_update}, the projected update for the auxiliary variable $\theta$ is given by
\begin{align*}
\langle\theta_h^{n+1},\psi_h\rangle_\Gamma
&=\langle2\theta_h^n-\theta_h^{n-1},\psi_h\rangle_\Gamma
-L\langle
\Pi_{\Sigma}\mathbf{u}_{f,h}^{n+1}\cdot\mathbf{n}_f
+(\boldsymbol{\xi}_{p,h}^{n+1}+\mathbf{u}_{p,h}^{n+1})
\cdot\mathbf{n}_p,\psi_h\rangle_\Gamma
\end{align*}
for every $\psi_h\in\Sigma_h$.

Finally, the initial interface variable is defined by
\begin{align}
\left\langle\theta_h^0,\psi_h\right\rangle_\Gamma
=\left\langle
-L\mathbf{u}_{f,h}^0\cdot\mathbf{n}_f-p_{p,h}^0,\psi_h
\right\rangle_\Gamma
\qquad\forall\psi_h\in\Sigma_h.
\label{eq:initial_discrete_theta}
\end{align}
The value $\theta_h^1$ required by the second-order scheme is computed using the first-order update \eqref{eq:first_order_theta_update}.

\subsection{A first-order Robin partitioned scheme}
\label{subsec:first_order_scheme}

We now present the first-order scheme obtained by combining the backward Euler method with the time-discrete Robin conditions derived in Subsection~\ref{subsec:discrete_robin_update}. Let the initial approximations $\mathbf{u}_{f,h}^{0}$, $\boldsymbol{\eta}_{p,h}^{0}$, $\boldsymbol{\xi}_{p,h}^{0}$, and $p_{p,h}^{0}$ be given in the corresponding finite element spaces. The initial interface variable $\theta_h^0\in\Sigma_h$ is defined by \eqref{eq:initial_discrete_theta}.

\begin{algorithm}[H]
\scriptsize
\caption{First-order Robin partitioned scheme}
\label{alg:first_order_scheme}
For $n=0,\ldots,N-1$, perform the following steps.
\begin{enumerate}
\item Find
$(\boldsymbol{\eta}_{p,h}^{n+1},\boldsymbol{\xi}_{p,h}^{n+1},
\mathbf{u}_{p,h}^{n+1},p_{p,h}^{n+1})
\in\mathbf{V}_{p,h}\times\mathbf{V}_{p,h}\times
\mathbf{X}_{p,h}\times Q_{p,h},
$
with $\boldsymbol{\xi}_{p,h}^{n+1}=d_t\boldsymbol{\eta}_{p,h}^{n+1}$, such that, for every
$(\boldsymbol{\omega}_{p,h},\mathbf{v}_{p,h},q_{p,h})
\in\mathbf{V}_{p,h}\times\mathbf{X}_{p,h}\times Q_{p,h}$,
\begin{align}
&\mathcal{A}_{p,h,\mathrm{stab}}^{d_t}
\big((\boldsymbol{\eta}_{p,h}^{n+1},\boldsymbol{\xi}_{p,h}^{n+1},
\mathbf{u}_{p,h}^{n+1},p_{p,h}^{n+1}),
(\boldsymbol{\omega}_{p,h},\mathbf{0},\mathbf{v}_{p,h},q_{p,h})\big)\nonumber\\
&\quad+\mathcal{R}_p\big((\boldsymbol{\xi}_{p,h}^{n+1},\mathbf{u}_{p,h}^{n+1});
(\boldsymbol{\omega}_{p,h},\mathbf{v}_{p,h})\big)
\nonumber\\
&=\mathcal{F}_{p,h,\mathrm{stab}}^{n+1}
(\boldsymbol{\omega}_{p,h},\mathbf{v}_{p,h},q_{p,h})
+\left\langle\theta_h^n,
(\boldsymbol{\omega}_{p,h}+\mathbf{v}_{p,h})\cdot\mathbf{n}_p\right\rangle_\Gamma-a_{\mathrm{BJS}}(\mathbf{u}_{f,h}^{n},\mathbf{0};
\mathbf{0},\boldsymbol{\omega}_{p,h}).
\label{eq:first_order_poroelastic_problem}
\end{align}

\item Find $(\mathbf{u}_{f,h}^{n+1},p_{f,h}^{n+1})
\in\mathbf{V}_{f,h}\times Q_{f,h}$ such that, for every
$(\mathbf{v}_{f,h},q_{f,h})\in\mathbf{V}_{f,h}\times Q_{f,h}$,
\begin{align}
&\mathcal{A}_{f,h}^{d_t}
\big((\mathbf{u}_{f,h}^{n+1},p_{f,h}^{n+1}),
(\mathbf{v}_{f,h},q_{f,h})\big)
+\mathcal{R}_f(\mathbf{u}_{f,h}^{n+1},\mathbf{v}_{f,h})
\nonumber\\
&=\mathcal{F}_{f,h}^{n+1}(\mathbf{v}_{f,h},q_{f,h})
+\left\langle
\theta_h^n-2L(\boldsymbol{\xi}_{p,h}^{n+1}+\mathbf{u}_{p,h}^{n+1})\cdot\mathbf{n}_p,
\mathbf{v}_{f,h}\cdot\mathbf{n}_f
\right\rangle_\Gamma
\nonumber\\
&\qquad\quad-a_{\mathrm{BJS}}(\mathbf{0},\boldsymbol{\xi}_{p,h}^{n};
\mathbf{v}_{f,h},\mathbf{0}).
\label{eq:first_order_fluid_problem}
\end{align}

\item Update $\theta_h^{n+1}\in\Sigma_h$ by
\begin{align}
\left\langle\theta_h^{n+1},\psi_h\right\rangle_\Gamma
&=\left\langle\theta_h^n,\psi_h\right\rangle_\Gamma
-L\left\langle
\Pi_{\Sigma}\mathbf{u}_{f,h}^{n+1}\cdot\mathbf{n}_f
+(\boldsymbol{\xi}_{p,h}^{n+1}+\mathbf{u}_{p,h}^{n+1})\cdot\mathbf{n}_p,
\psi_h\right\rangle_\Gamma
\label{eq:first_order_interface_update_algorithm}
\end{align}
for every $\psi_h\in\Sigma_h$.
\end{enumerate}
\end{algorithm}


\paragraph{Stability analysis}
For the stability analysis, we use the algebraic identities
\begin{align}
&2a(a-b)=a^{2}-b^{2}+(a-b)^{2},\label{eq:first_order_algebraic_identity}\\
&ab=\frac{1}{4}\big((a+b)^{2}-(a-b)^{2}\big).\label{eq:first_order_polarization_identity}
\end{align}
We define the discrete energy and dissipation terms by
\begin{align*}
\mathcal{E}_{fluid}^{n+1}
&=\frac{\rho_f}{2}\|\mathbf{u}_{f,h}^{n+1}\|_{L^2(\Omega_f)}^2,\quad\mathcal{J}_{fluid}^{n+1}=2\mu_f\|\boldsymbol{\varepsilon}(\mathbf{u}_{f,h}^{n+1})\|_{L^2(\Omega_f)}^2+\frac{\rho_f\Delta t}{2}\|d_t\mathbf{u}_{f,h}^{n+1}\|_{L^2(\Omega_f)}^2,\\
\mathcal{E}_{poro}^{n+1}
&=\frac{\rho_p}{2}\|\boldsymbol{\xi}_{p,h}^{n+1}\|_{L^2(\Omega_p)}^2
+\mu_p\|\boldsymbol{\varepsilon}(\boldsymbol{\eta}_{p,h}^{n+1})\|_{L^2(\Omega_p)}^2+\frac{\lambda_p}{2}\|\nabla\cdot\boldsymbol{\eta}_{p,h}^{n+1}\|_{L^2(\Omega_p)}^2\\
&\quad
+\frac{c_0}{2}\|p_{p,h}^{n+1}\|_{L^2(\Omega_p)}^2,\\
\mathcal{J}_{poro}^{n+1}
&=\frac{\rho_p\Delta t}{2}\|d_t\boldsymbol{\xi}_{p,h}^{n+1}\|_{L^2(\Omega_p)}^2+\mu_p\Delta t\|\boldsymbol{\varepsilon}(\boldsymbol{\xi}_{p,h}^{n+1})\|_{L^2(\Omega_p)}^2
+\frac{\lambda_p\Delta t}{2}\|\nabla\cdot\boldsymbol{\xi}_{p,h}^{n+1}\|_{L^2(\Omega_p)}^2
\\
&\quad+\frac{c_0\Delta t}{2}\|d_t p_{p,h}^{n+1}\|_{L^2(\Omega_p)}^2+\frac{1}{2\mu_f}\|\mathbf{K}^{1/2}\nabla p_{p,h}^{n+1}\|_{L^2(\Omega_p)}^2
+\frac{\mu_f}{2}\|\mathbf{K}^{-1/2}\mathbf{u}_{p,h}^{n+1}\|_{L^2(\Omega_p)}^2,\\
\mathcal{E}_{h}^{n+1}&=\mathcal{E}_{fluid}^{n+1}+\mathcal{E}_{poro}^{n+1},\quad\mathcal{J}_{h}^{n+1}=\mathcal{J}_{fluid}^{n+1}+\mathcal{J}_{poro}^{n+1},\\
\mathcal{I}_{h}^{n+1}
&=\frac{1}{2L}\|\theta_h^{n+1}\|_{L^2(\Gamma)}^2
+\frac{1}{2}\sum_{j=1}^{d-1}\gamma_{\mathrm{BJS},j}
\left(
\|\mathbf{u}_{f,h}^{n+1}\cdot\boldsymbol{\tau}_j\|_{L^2(\Gamma)}^2
+\|\boldsymbol{\xi}_{p,h}^{n+1}\cdot\boldsymbol{\tau}_j\|_{L^2(\Gamma)}^2
\right),\\[0.3em]
\mathcal{D}_{h}^{n+1}
&=\frac{L}{2}\left\|
\mathbf{u}_{f,h}^{n+1}\cdot\mathbf{n}_f
+(\boldsymbol{\xi}_{p,h}^{n+1}+\mathbf{u}_{p,h}^{n+1})\cdot\mathbf{n}_p
\right\|_{L^2(\Gamma)}^2\\
&\quad+\frac{1}{2}\sum_{j=1}^{d-1}\gamma_{\mathrm{BJS},j}
\left(
\|(\mathbf{u}_{f,h}^{n+1}-\boldsymbol{\xi}_{p,h}^{n})\cdot\boldsymbol{\tau}_j\|_{L^2(\Gamma)}^2
+\|(\boldsymbol{\xi}_{p,h}^{n+1}-\mathbf{u}_{f,h}^{n})\cdot\boldsymbol{\tau}_j\|_{L^2(\Gamma)}^2
\right).
\end{align*} 

For simplicity, we consider the homogeneous case $\mathbf{f}_{f}=\mathbf{f}_{p}=\mathbf{f}_{d}=\mathbf{0}$ and $\phi_{f}=\phi_{d}=0$. The result can be extended to nonhomogeneous data by applying the Cauchy--Schwarz, Poincar\'e, Korn, and Young inequalities, together with the discrete inf--sup condition \eqref{eq:discrete_fluid_inf_sup}. We omit these standard estimates and focus on the energy balance of the proposed scheme.
\begin{theorem}[Stability of the first-order scheme]\label{thm:first_order_stability}
	Let $(\mathbf{u}_{f,h}^{n+1},p_{f,h}^{n+1},\boldsymbol{\xi}_{p,h}^{n+1},\boldsymbol{\eta}_{p,h}^{n+1},\\ \mathbf{u}_{p,h}^{n+1},p_{p,h}^{n+1},\theta_{h}^{n+1})$ be defined by Algorithm~\ref{alg:first_order_scheme}. Then, for any $0\leq\ell\leq N-1$, the following energy inequality holds: 
	\begin{equation}
\mathcal{E}_{h}^{\ell+1}
+\Delta t\,\mathcal{I}_{h}^{\ell+1}
+\Delta t\sum_{n=0}^{\ell}
\left(\mathcal{J}_{h}^{n+1}+\mathcal{D}_{h}^{n+1}\right)\leq\mathcal{E}_{h}^{0}+\Delta t\,\mathcal{I}_{h}^{0}.
\label{eq:first_order_stability_estimate}
\end{equation}
\end{theorem}

For the proof of Theorem \ref{thm:first_order_stability}, the interested reader can refer to Appendix A.

\begin{remark}
The interface update is understood through the $L^2(\Gamma)$-orthogonal projection $\Pi_\Sigma$. Its orthogonality contributes a nonnegative projection-dissipation term to the one-step energy balance. Hence, the mismatch between the trace of the $P_{k+1}$ fluid velocity and the $P_k$ interface space does not affect unconditional stability. If the quantity being projected belongs to $\Sigma_h$, the projection is exact and the energy inequality reduces to an equality.
\end{remark}

\subsection{A second-order Robin partitioned scheme}
\label{subsec:second_order_scheme}

We now construct a second-order time discretization by combining the BDF2 formula with second-order extrapolation of the transmitted interface data. The values at $t^1$ are computed using Algorithm~\ref{alg:first_order_scheme}. Hence, the discrete functions
$$
\mathbf{u}_{f,h}^0,\ \mathbf{u}_{f,h}^1,\ 
\boldsymbol{\eta}_{p,h}^0,\ \boldsymbol{\eta}_{p,h}^1,\ 
\boldsymbol{\xi}_{p,h}^0,\ \boldsymbol{\xi}_{p,h}^1,\ 
\mathbf{u}_{p,h}^0,\ \mathbf{u}_{p,h}^1,\ 
p_{p,h}^0,\ p_{p,h}^1,\ 
\theta_h^0,\ \theta_h^1
$$
are assumed to be available.

\begin{algorithm}[H]
\scriptsize
\caption{Second-order Robin partitioned scheme}
\label{alg:second_order_scheme}
For $n=1,\ldots,N-1$, perform the following steps.
\begin{enumerate}
\item Find
$(\boldsymbol{\eta}_{p,h}^{n+1},\boldsymbol{\xi}_{p,h}^{n+1},
\mathbf{u}_{p,h}^{n+1},p_{p,h}^{n+1})
\in\mathbf{V}_{p,h}\times\mathbf{V}_{p,h}\times
\mathbf{X}_{p,h}\times\widehat{Q}_{p,h}
$
with $\boldsymbol{\xi}_{p,h}^{n+1}=d_t^{(2)}\boldsymbol{\eta}_{p,h}^{n+1}$ such that, for every
$(\boldsymbol{\omega}_{p,h},
\mathbf{v}_{p,h},q_{p,h})
\in\mathbf{V}_{p,h}\times
\mathbf{X}_{p,h}\times\widehat{Q}_{p,h}$,
\begin{align}
&\mathcal{A}_{p,h,\mathrm{stab}}^{d_t^{(2)}}
\big((\boldsymbol{\eta}_{p,h}^{n+1},\boldsymbol{\xi}_{p,h}^{n+1},
\mathbf{u}_{p,h}^{n+1},p_{p,h}^{n+1}),
(\boldsymbol{\omega}_{p,h},\mathbf{0},
\mathbf{v}_{p,h},q_{p,h})\big)
\nonumber\\
&\quad+\mathcal{R}_p
\big((\boldsymbol{\xi}_{p,h}^{n+1},\mathbf{u}_{p,h}^{n+1});
(\boldsymbol{\omega}_{p,h},\mathbf{v}_{p,h})\big)
\nonumber\\
&=\mathcal{F}_{p,h,\mathrm{stab}}^{n+1}
(\boldsymbol{\omega}_{p,h},\mathbf{v}_{p,h},q_{p,h})
+\left\langle
2\theta_h^n-\theta_h^{n-1},
(\boldsymbol{\omega}_{p,h}+\mathbf{v}_{p,h})\cdot\mathbf{n}_p
\right\rangle_\Gamma
\nonumber\\
&\quad-a_{\mathrm{BJS}}
(2\mathbf{u}_{f,h}^{n}-\mathbf{u}_{f,h}^{n-1},\mathbf{0};
\mathbf{0},\boldsymbol{\omega}_{p,h}).
\label{eq:second_order_poroelastic_problem}
\end{align}

\item Find $(\mathbf{u}_{f,h}^{n+1},p_{f,h}^{n+1})
\in\mathbf{V}_{f,h}\times Q_{f,h}$ such that, for every
$(\mathbf{v}_{f,h},q_{f,h})\in\mathbf{V}_{f,h}\times Q_{f,h}$,
\begin{align}
&\mathcal{A}_{f,h}^{d_t^{(2)}}
\big((\mathbf{u}_{f,h}^{n+1},p_{f,h}^{n+1}),
(\mathbf{v}_{f,h},q_{f,h})\big)
+\mathcal{R}_f(\mathbf{u}_{f,h}^{n+1},\mathbf{v}_{f,h})
\nonumber\\
&=\mathcal{F}_{f,h}^{n+1}(\mathbf{v}_{f,h},q_{f,h})
+\left\langle
2\theta_h^n-\theta_h^{n-1}
-2L(\boldsymbol{\xi}_{p,h}^{n+1}
+\mathbf{u}_{p,h}^{n+1})\cdot\mathbf{n}_p,
\mathbf{v}_{f,h}\cdot\mathbf{n}_f
\right\rangle_\Gamma
\nonumber\\
&\quad-a_{\mathrm{BJS}}
(\mathbf{0},\boldsymbol{\xi}_{p,h}^{n+1};
\mathbf{v}_{f,h},\mathbf{0}).
\label{eq:second_order_fluid_problem}
\end{align}

\item Update $\theta_h^{n+1}\in\Sigma_h$ according to
\begin{align}
\left\langle\theta_h^{n+1},\psi_h\right\rangle_\Gamma
&=\left\langle2\theta_h^n-\theta_h^{n-1},\psi_h\right\rangle_\Gamma-L\left\langle
\Pi_{\Sigma}\mathbf{u}_{f,h}^{n+1}\cdot\mathbf{n}_f
+(\boldsymbol{\xi}_{p,h}^{n+1}
+\mathbf{u}_{p,h}^{n+1})\cdot\mathbf{n}_p,
\psi_h
\right\rangle_\Gamma
\label{eq:second_order_interface_update_algorithm}
\end{align}
for every $\psi_h\in\Sigma_h$.
\end{enumerate}
\end{algorithm}

The poroelastic subproblem is first solved using the extrapolated fluid tangential velocity and the extrapolated interface datum. Its normal flux and current structure velocity are then used in the fluid subproblem, after which the interface variable is updated. Thus, Algorithm~\ref{alg:second_order_scheme} remains a sequential Robin partitioned method.

\paragraph{Stability analysis}
We firstly introduce the BDF2 identity
\begin{equation}
2(3a-4b+c,a)
=\|a\|^2-\|b\|^2
+\|2a-b\|^2-\|2b-c\|^2
+\|a-2b+c\|^2,
\label{eq:bdf2_energy_identity}
\end{equation}
where $(\cdot,\cdot)$ and $\|\cdot\|$ denote an inner product and its induced norm. Define the BDF2 discrete energy by
\begin{align*}
\widetilde{\mathcal{E}}_h^{n+1}=\widetilde{\mathcal{E}}_{fluid}^{n+1}+\widetilde{\mathcal{E}}_{poro}^{n+1},
\end{align*}
where
\begin{align*}
\widetilde{\mathcal{E}}_{fluid}^{n+1}
&=\frac{\rho_f}{4}
\left(
\|\mathbf{u}_{f,h}^{n+1}\|_{L^2(\Omega_f)}^2
+\|2\mathbf{u}_{f,h}^{n+1}-\mathbf{u}_{f,h}^{n}\|_{L^2(\Omega_f)}^2
\right),\nonumber\\
\widetilde{\mathcal{E}}_{poro}^{n+1}&=\frac{\rho_p}{4}
\|\boldsymbol{\xi}_{p,h}^{n+1}\|_{L^2(\Omega_p)}^2
+\frac{\mu_p}{2}\|\boldsymbol{\varepsilon}(\boldsymbol{\eta}_{p,h}^{n+1})\|_{L^2(\Omega_p)}^2+\frac{\lambda_p}{4}\|\nabla\cdot\boldsymbol{\eta}_{p,h}^{n+1}\|_{L^2(\Omega_p)}^2+\frac{c_0}{4}\|p_{p,h}^{n+1}\|_{L^2(\Omega_p)}^2
\\
&\quad+\frac{\rho_p}{4}\|2\boldsymbol{\xi}_{p,h}^{n+1}
-\boldsymbol{\xi}_{p,h}^{n}\|_{L^2(\Omega_p)}^2+\frac{\mu_p}{2}\|2\boldsymbol{\varepsilon}(\boldsymbol{\eta}_{p,h}^{n+1})
-\boldsymbol{\varepsilon}(\boldsymbol{\eta}_{p,h}^{n})\|_{L^2(\Omega_p)}^2\\
&\quad+\frac{\lambda_p}{4}\|2\nabla\cdot\boldsymbol{\eta}_{p,h}^{n+1}
-\nabla\cdot\boldsymbol{\eta}_{p,h}^{n}\|_{L^2(\Omega_p)}^2+\frac{c_0}{4}\|2p_{p,h}^{n+1}-p_{p,h}^{n}\|_{L^2(\Omega_p)}^2.
\end{align*}
The corresponding bulk dissipation is
\begin{align*}
\widetilde{\mathcal{J}}_{h}^{n+1}&=\widetilde{\mathcal{J}}_{fluid}^{n+1}+\widetilde{\mathcal{J}}_{poro}^{n+1},
\end{align*}
where
\begin{align*}
\widetilde{\mathcal{J}}_{fluid}^{n+1}
&=2\mu_f
\|\boldsymbol{\varepsilon}(\mathbf{u}_{f,h}^{n+1})\|_{L^2(\Omega_f)}^2
+\frac{\rho_f}{4\Delta t}
\|\mathbf{u}_{f,h}^{n+1}-2\mathbf{u}_{f,h}^{n}+\mathbf{u}_{f,h}^{n-1}\|_{L^2(\Omega_f)}^2,
\nonumber\\
\widetilde{\mathcal{J}}_{poro}^{n+1}&=\frac{\rho_p}{4\Delta t}
\|\boldsymbol{\xi}_{p,h}^{n+1}-2\boldsymbol{\xi}_{p,h}^{n}+\boldsymbol{\xi}_{p,h}^{n-1}\|_{L^2(\Omega_p)}^2+\frac{\mu_p}{2\Delta t}
\|\boldsymbol{\varepsilon}(\boldsymbol{\eta}_{p,h}^{n+1}-2\boldsymbol{\eta}_{p,h}^{n}+\boldsymbol{\eta}_{p,h}^{n-1})\|_{L^2(\Omega_p)}^2
\nonumber\\
&\quad+\frac{\lambda_p}{4\Delta t}
\|\nabla\cdot(\boldsymbol{\eta}_{p,h}^{n+1}-2\boldsymbol{\eta}_{p,h}^{n}+\boldsymbol{\eta}_{p,h}^{n-1})\|_{L^2(\Omega_p)}^2+\frac{c_0}{4\Delta t}
\|p_{p,h}^{n+1}-2p_{p,h}^{n}+p_{p,h}^{n-1}\|_{L^2(\Omega_p)}^2\nonumber\\
&\quad+\frac{\mu_f}{2}
\|\mathbf{K}^{-1/2}\mathbf{u}_{p,h}^{n+1}\|_{L^2(\Omega_p)}^2
+\frac{1}{2\mu_f}
\|\mathbf{K}^{1/2}\nabla p_{p,h}^{n+1}\|_{L^2(\Omega_p)}^2.
\end{align*}
For the final estimate, we also introduce
\begin{align*}
\widehat{\mathcal{J}}_h^{n+1}
&=\widetilde{\mathcal{J}}_h^{n+1}
-\mu_f\|\boldsymbol{\varepsilon}(\mathbf{u}_{f,h}^{n+1})\|_{L^2(\Omega_f)}^2-\frac{1}{4\mu_f}
\|\mathbf{K}^{1/2}\nabla p_{p,h}^{n+1}\|_{L^2(\Omega_p)}^2.
\end{align*}
Clearly, $\widehat{\mathcal{J}}_h^{n+1}\geq0$.

As in the first-order analysis, we consider homogeneous forcing and source terms:
$\mathbf{f}_f=\mathbf{f}_p=\mathbf{f}_d=\mathbf{0},~\phi_{f}=\phi_{d}=0$.

\begin{theorem}[Stability of the second-order scheme]
\label{thm:second_order_stability}
Let $(\mathbf{u}_{f,h}^{n+1},p_{f,h}^{n+1},\boldsymbol{\xi}_{p,h}^{n+1},\boldsymbol{\eta}_{p,h}^{n+1},\\
\mathbf{u}_{p,h}^{n+1},p_{p,h}^{n+1},\theta_{h}^{n+1})$ be defined by Algorithm~\ref{alg:second_order_scheme}. Then, there exists a constant $L_0>0$, independent of $h$ and $\Delta t$, such that, for every $L\geq L_0$, the solution generated by Algorithm~\ref{alg:second_order_scheme} satisfies, for $1\leq\ell\leq N-1$,
\begin{align}
&\widetilde{\mathcal{E}}_h^{\ell+1}
+\frac{\Delta t^{3}}{2L}
\|d_{t}\theta_h^{\ell+1}\|_{L^2(\Gamma)}^2
+\Delta t\sum_{n=1}^{\ell}
\Big(
\widehat{\mathcal{J}}_h^{n+1}
+\frac{1}{2}
\sum_{j=1}^{d-1}\gamma_{\mathrm{BJS},j}
\|\mathbf{u}_{f,h}^{n+1}\cdot\boldsymbol{\tau}_j\|_{L^2(\Gamma)}^2
\Big)
\nonumber\\
&\leq\exp\!\left(C(1+L^2)T\right)
\Big(
\widetilde{\mathcal{E}}_h^1
+\frac{\Delta t^{3}}{2L}\|d_{t}\theta_h^1\|_{L^2(\Gamma)}^2+\Delta t\mu_f\|\boldsymbol{\varepsilon}(\mathbf{u}_{f,h}^{1})\|_{L^2(\Omega_f)}^2
\nonumber\\
&\quad+\Delta t\mu_f\|\boldsymbol{\varepsilon}(\mathbf{u}_{f,h}^{0})\|_{L^2(\Omega_f)}^2\Big),
\label{eq:second_order_stability_estimate}
\end{align}
where $C>0$ is independent of $h$, $\Delta t$, and $N$. Consequently, the second-order scheme is stable without a CFL-type condition.
\end{theorem}

\begin{proof}
Choosing
$
(\boldsymbol{\omega}_{p,h},
\mathbf{v}_{p,h},q_{p,h})
=
(\boldsymbol{\xi}_{p,h}^{n+1},
\mathbf{u}_{p,h}^{n+1},p_{p,h}^{n+1})
$
as the test functions in \eqref{eq:second_order_poroelastic_problem} and using
$\boldsymbol{\xi}_{p,h}^{n+1}
=d_t^{(2)}\boldsymbol{\eta}_{p,h}^{n+1}$
and the identity \eqref{eq:bdf2_energy_identity}, we obtain
\begin{align}
&d_t\widetilde{\mathcal{E}}_{poro}^{n+1}+\widetilde{\mathcal{J}}_{poro}^{n+1}+L
\|(\boldsymbol{\xi}_{p,h}^{n+1}
+\mathbf{u}_{p,h}^{n+1})\cdot\mathbf{n}_p\|_{L^2(\Gamma)}^2+\sum_{j=1}^{d-1}\gamma_{\mathrm{BJS},j}
\|\boldsymbol{\xi}_{p,h}^{n+1}\cdot\boldsymbol{\tau}_j\|_{L^2(\Gamma)}^2\nonumber\\
&=\left\langle
2\theta_h^n-\theta_h^{n-1},
(\boldsymbol{\xi}_{p,h}^{n+1}
+\mathbf{u}_{p,h}^{n+1})\cdot\mathbf{n}_p
\right\rangle_\Gamma\nonumber\\
&\quad+\sum_{j=1}^{d-1}\gamma_{\mathrm{BJS},j}
\left\langle
(2\mathbf{u}_{f,h}^{n}-\mathbf{u}_{f,h}^{n-1})
\cdot\boldsymbol{\tau}_j,
\boldsymbol{\xi}_{p,h}^{n+1}\cdot\boldsymbol{\tau}_j
\right\rangle_\Gamma.
\label{eq:second_order_poroelastic_energy}
\end{align}
Similarly, choosing
$(\mathbf{v}_{f,h},q_{f,h})
=(\mathbf{u}_{f,h}^{n+1},p_{f,h}^{n+1})$ as the test functions
in \eqref{eq:second_order_fluid_problem} gives
\begin{align}
&d_t\widetilde{\mathcal{E}}_{fluid}^{n+1}+\widetilde{\mathcal{J}}_{fluid}^{n+1}+L\|\mathbf{u}_{f,h}^{n+1}\cdot\mathbf{n}_f\|_{L^2(\Gamma)}^2
+\sum_{j=1}^{d-1}\gamma_{\mathrm{BJS},j}
\|\mathbf{u}_{f,h}^{n+1}\cdot\boldsymbol{\tau}_j\|_{L^2(\Gamma)}^2
\nonumber\\
&=
\left\langle
2\theta_h^n-\theta_h^{n-1}
-2L(\boldsymbol{\xi}_{p,h}^{n+1}
+\mathbf{u}_{p,h}^{n+1})\cdot\mathbf{n}_p,
\mathbf{u}_{f,h}^{n+1}\cdot\mathbf{n}_f
\right\rangle_\Gamma\nonumber\\
&\quad+
\sum_{j=1}^{d-1}\gamma_{\mathrm{BJS},j}
\left\langle
\boldsymbol{\xi}_{p,h}^{n+1}\cdot\boldsymbol{\tau}_j,
\mathbf{u}_{f,h}^{n+1}\cdot\boldsymbol{\tau}_j
\right\rangle_\Gamma.
\label{eq:second_order_fluid_energy}
\end{align}
Since $
2\theta_h^n-\theta_h^{n-1}
=\theta_h^n+(\theta_h^n-\theta_h^{n-1})$,
the interface update \eqref{eq:second_order_interface_update_algorithm} can be written as
\begin{equation}
\langle\theta_h^{n+1}-\theta_h^n,\psi_{h}\rangle_{\Gamma}
=
\langle
\theta_h^n-\theta_h^{n-1}
-L
\Pi_\Sigma\mathbf{u}_{f,h}^{n+1}\cdot\mathbf{n}_f
-L(\boldsymbol{\xi}_{p,h}^{n+1}
+\mathbf{u}_{p,h}^{n+1})\cdot\mathbf{n}_p
,\psi_{h}\rangle_{\Gamma}.
\label{eq:second_order_projected_update}
\end{equation}
Then the following inequality holds for \eqref{eq:second_order_projected_update}:
\begin{align}
&\|\theta_h^{n+1}-\theta_h^n\|_{L^2(\Gamma)}^2\nonumber\\
&\leq
\|
\theta_h^n-\theta_h^{n-1}
-L\mathbf{u}_{f,h}^{n+1}\cdot\mathbf{n}_f
-L(\boldsymbol{\xi}_{p,h}^{n+1}
+\mathbf{u}_{p,h}^{n+1})\cdot\mathbf{n}_p
\|_{L^2(\Gamma)}\|\theta_h^{n+1}-\theta_h^n\|_{L^2(\Gamma)}.
\label{eq:second_order_theta_projection_bound}
\end{align}
Using the polarization identity \eqref{eq:first_order_polarization_identity}, we further obtain
\begin{align}
&\langle
\theta_h^n-\theta_h^{n-1}
-\frac{L}{2}
(\boldsymbol{\xi}_{p,h}^{n+1}
+\mathbf{u}_{p,h}^{n+1})\cdot\mathbf{n}_p,
(\boldsymbol{\xi}_{p,h}^{n+1}
+\mathbf{u}_{p,h}^{n+1})\cdot\mathbf{n}_p
\rangle_\Gamma
\nonumber\\
&\qquad=
\frac{1}{2L}
\Big(
\|\theta_h^n-\theta_h^{n-1}\|_{L^2(\Gamma)}^2
-\|
\theta_h^n-\theta_h^{n-1}
-L(\boldsymbol{\xi}_{p,h}^{n+1}
+\mathbf{u}_{p,h}^{n+1})\cdot\mathbf{n}_p
\|_{L^2(\Gamma)}^2
\Big),
\label{eq:second_order_normal_identity_one}\\
&\langle
\theta_h^n-\theta_h^{n-1}
-L(\boldsymbol{\xi}_{p,h}^{n+1}
+\mathbf{u}_{p,h}^{n+1})\cdot\mathbf{n}_p
-\frac{L}{2}\mathbf{u}_{f,h}^{n+1}\cdot\mathbf{n}_f,
\mathbf{u}_{f,h}^{n+1}\cdot\mathbf{n}_f
\rangle_\Gamma
\nonumber\\
&\qquad=
\frac{1}{2L}
\Big(\|
\theta_h^n-\theta_h^{n-1}
-L(\boldsymbol{\xi}_{p,h}^{n+1}
+\mathbf{u}_{p,h}^{n+1})\cdot\mathbf{n}_p\|_{L^2(\Gamma)}^2
\nonumber\\
&\qquad\quad
-\|
\theta_h^n-\theta_h^{n-1}
-L\mathbf{u}_{f,h}^{n+1}\cdot\mathbf{n}_f
-L(\boldsymbol{\xi}_{p,h}^{n+1}
+\mathbf{u}_{p,h}^{n+1})\cdot\mathbf{n}_p
\|_{L^2(\Gamma)}^2
\Big).
\label{eq:second_order_normal_identity_two}
\end{align}
Moreover,
\begin{align}
&\frac{L}{2}
\|\mathbf{u}_{f,h}^{n+1}\cdot\mathbf{n}_f\|_{L^2(\Gamma)}^2
+\frac{L}{2}
\|(\boldsymbol{\xi}_{p,h}^{n+1}
+\mathbf{u}_{p,h}^{n+1})\cdot\mathbf{n}_p\|_{L^2(\Gamma)}^2\nonumber\\
&\quad+
L\langle
(\boldsymbol{\xi}_{p,h}^{n+1}
+\mathbf{u}_{p,h}^{n+1})\cdot\mathbf{n}_p,
\mathbf{u}_{f,h}^{n+1}\cdot\mathbf{n}_f
\rangle_\Gamma
\nonumber\\
&=
\frac{L}{2}
\|
\mathbf{u}_{f,h}^{n+1}\cdot\mathbf{n}_f
+(\boldsymbol{\xi}_{p,h}^{n+1}
+\mathbf{u}_{p,h}^{n+1})\cdot\mathbf{n}_p
\|_{L^2(\Gamma)}^2.
\label{eq:second_order_normal_residual_identity}
\end{align}

Adding \eqref{eq:second_order_poroelastic_energy} and
\eqref{eq:second_order_fluid_energy}, then applying
\eqref{eq:second_order_theta_projection_bound}--\eqref{eq:second_order_normal_residual_identity} and the summation operator $\Delta t\sum_{n=1}^{\ell}$ yields
\begin{align}
&
\widetilde{\mathcal{E}}_h^{\ell+1}
+\frac{\Delta t^{3}}{2L}
\|d_{t}\theta_h^{\ell+1}\|_{L^2(\Gamma)}^2
+\Delta t\sum_{n=1}^{\ell}\Big[\frac{L}{2}\|
\mathbf{u}_{f,h}^{n+1}\cdot\mathbf{n}_f
+(\boldsymbol{\xi}_{p,h}^{n+1}
+\mathbf{u}_{p,h}^{n+1})\cdot\mathbf{n}_p\|_{L^2(\Gamma)}^2\nonumber\\
&\quad+\Delta t\sum_{n=1}^{\ell}\Big[\widetilde{\mathcal{J}}_h^{n+1}+
\sum_{j=1}^{d-1}\gamma_{\mathrm{BJS},j}
\Big(
\|\mathbf{u}_{f,h}^{n+1}\cdot\boldsymbol{\tau}_j\|_{L^2(\Gamma)}^2
+\|\boldsymbol{\xi}_{p,h}^{n+1}\cdot\boldsymbol{\tau}_j\|_{L^2(\Gamma)}^2
\Big)\Big]
\nonumber\\
&\leq
\widetilde{\mathcal{E}}_h^{1}
+\frac{\Delta t^{3}}{2L}
\|d_{t}\theta_h^{1}\|_{L^2(\Gamma)}^2+\Delta t\sum_{n=1}^{\ell}\Big[\langle
\theta_h^n,
\mathbf{u}_{f,h}^{n+1}\cdot\mathbf{n}_f
+(\boldsymbol{\xi}_{p,h}^{n+1}
+\mathbf{u}_{p,h}^{n+1})\cdot\mathbf{n}_p\rangle_\Gamma\nonumber\\
&\quad+\sum_{j=1}^{d-1}\gamma_{\mathrm{BJS},j}\Big(\langle
(2\mathbf{u}_{f,h}^{n}-\mathbf{u}_{f,h}^{n-1})
\cdot\boldsymbol{\tau}_j,
\boldsymbol{\xi}_{p,h}^{n+1}\cdot\boldsymbol{\tau}_j\rangle_\Gamma+\langle
\boldsymbol{\xi}_{p,h}^{n+1}\cdot\boldsymbol{\tau}_j,
\mathbf{u}_{f,h}^{n+1}\cdot\boldsymbol{\tau}_j\rangle_\Gamma\Big)\Big].
\label{eq:second_order_preliminary_energy}
\end{align}
Next, we estimate the first BJS cross term on the right side of \eqref{eq:second_order_preliminary_energy} by the Cauchy-Schwarz, Trace, Korn, and Young inequalities as follows:
\begin{align}
&\Delta t\sum_{n=1}^{\ell}
\sum_{j=1}^{d-1}\gamma_{\mathrm{BJS},j}
\left\langle
(2\mathbf{u}_{f,h}^{n}-\mathbf{u}_{f,h}^{n-1})
\cdot\boldsymbol{\tau}_j,
\boldsymbol{\xi}_{p,h}^{n+1}\cdot\boldsymbol{\tau}_j
\right\rangle_\Gamma
\nonumber\\
&\leq
C\Delta t\sum_{n=0}^{\ell}
\|\mathbf{u}_{f,h}^{n}\|_{L^2(\Omega_f)}^2
+\frac{\mu_f\Delta t}{2}
\sum_{n=0}^{\ell}
\|\boldsymbol{\varepsilon}(\mathbf{u}_{f,h}^{n})\|_{L^2(\Omega_f)}^2\nonumber\\
&\quad+
\frac{\Delta t}{2}
\sum_{n=1}^{\ell}
\sum_{j=1}^{d-1}\gamma_{\mathrm{BJS},j}
\|\boldsymbol{\xi}_{p,h}^{n+1}\cdot\boldsymbol{\tau}_j\|_{L^2(\Gamma)}^2.
\label{eq:second_order_bjs_extrapolation_estimate}
\end{align}
And it is easy to check that the second BJS cross term on the right side of \eqref{eq:second_order_preliminary_energy} satisfies
\begin{align}
&\Delta t\sum_{n=1}^{\ell}
\sum_{j=1}^{d-1}\gamma_{\mathrm{BJS},j}
\langle
\boldsymbol{\xi}_{p,h}^{n+1}\cdot\boldsymbol{\tau}_j,
\mathbf{u}_{f,h}^{n+1}\cdot\boldsymbol{\tau}_j
\rangle_\Gamma\nonumber\\
&\leq
\frac{\Delta t}{2}
\sum_{n=1}^{\ell}
\sum_{j=1}^{d-1}\gamma_{\mathrm{BJS},j}
\Big(
\|\boldsymbol{\xi}_{p,h}^{n+1}\cdot\boldsymbol{\tau}_j\|_{L^2(\Gamma)}^2
+\|\mathbf{u}_{f,h}^{n+1}\cdot\boldsymbol{\tau}_j\|_{L^2(\Gamma)}^2
\Big).
\label{eq:second_order_bjs_current_estimate}
\end{align}

By the definition of the interface variable and the $L^2(\Gamma)$ stability of $\Pi_\Sigma$, we have
\begin{equation}
\theta_h^n
=\Pi_\Sigma\left(
-L\mathbf{u}_{f,h}^{n}\cdot\mathbf{n}_f-p_{p,h}^{n}
\right).
\label{eq:second_order_theta_representation}
\end{equation}
Then the trace, Poincar\'e, Korn, and Young inequalities immediately imply
\begin{align}
&\langle
\theta_h^n,
\mathbf{u}_{f,h}^{n+1}\cdot\mathbf{n}_f
+(\boldsymbol{\xi}_{p,h}^{n+1}
+\mathbf{u}_{p,h}^{n+1})\cdot\mathbf{n}_p\rangle_\Gamma
\nonumber\\
&\leq
CL^2\|\mathbf{u}_{f,h}^{n}\|_{L^2(\Omega_f)}^2
+\frac{\mu_f}{2}
\|\boldsymbol{\varepsilon}(\mathbf{u}_{f,h}^{n})\|_{L^2(\Omega_f)}^2
+\frac{C}{L}
\|\nabla p_{p,h}^{n}\|_{L^2(\Omega_p)}^2
\nonumber\\
&\quad+
\frac{L}{2}
\left\|
\mathbf{u}_{f,h}^{n+1}\cdot\mathbf{n}_f
+(\boldsymbol{\xi}_{p,h}^{n+1}
+\mathbf{u}_{p,h}^{n+1})\cdot\mathbf{n}_p
\right\|_{L^2(\Gamma)}^2.
\label{eq:second_order_theta_residual_estimate}
\end{align}

Choosing
$$
L\geq L_0:=\frac{4\mu_f C}{k_{\min}}
$$
allows the pressure-gradient contribution in
\eqref{eq:second_order_theta_residual_estimate} to be absorbed by the Darcy dissipation. Substituting
\eqref{eq:second_order_bjs_extrapolation_estimate}--\eqref{eq:second_order_theta_residual_estimate}
into the \eqref{eq:second_order_preliminary_energy}, we obtain
\begin{align}
&\widetilde{\mathcal{E}}_h^{\ell+1}
+\frac{\Delta t^{3}}{2L}
\|d_{t}\theta_h^{\ell+1}\|_{L^2(\Gamma)}^2+\Delta t\sum_{n=1}^{\ell}
\Big(
\widehat{\mathcal{J}}_h^{n+1}
+\frac{1}{2}
\sum_{j=1}^{d-1}\gamma_{\mathrm{BJS},j}
\|\mathbf{u}_{f,h}^{n+1}\cdot\boldsymbol{\tau}_j\|_{L^2(\Gamma)}^2
\Big)
\nonumber\\
&\leq
\widetilde{\mathcal{E}}_h^1
+\frac{\Delta t^{3}}{2L}\|d_{t}\theta_h^1\|_{L^2(\Gamma)}^2
+\mu_f\Delta t
\|\boldsymbol{\varepsilon}(\mathbf{u}_{f,h}^{1})\|_{L^2(\Omega_f)}^2+\mu_f\Delta t
\|\boldsymbol{\varepsilon}(\mathbf{u}_{f,h}^{0})\|_{L^2(\Omega_f)}^2\nonumber\\
&\quad+
C(1+L^2)\Delta t
\sum_{n=0}^{\ell}
\widetilde{\mathcal{E}}_h^{n}.
\label{eq:second_order_gronwall_form}
\end{align}
An application of the discrete Gr\"onwall inequality to
\eqref{eq:second_order_gronwall_form} gives
\eqref{eq:second_order_stability_estimate}. This completes the proof.
\end{proof}

\begin{remark}
To apply discrete Gr\"onwall inequality to \eqref{eq:second_order_gronwall_form}, we require the condition
\begin{align}
C(1+L^{2})\Delta t \le 1. \label{condition}
\end{align}
Therefore, for any sufficiently small $\Delta t$, one may choose $L=O(1/\sqrt{\Delta t})$ so that \eqref{condition} is satisfied. Together with the requirement $L\ge L_0$, this implies $L_0 \le L \lesssim 1/\sqrt{\Delta t}$, which provides a guideline for selecting $L$ in the numerical experiments.
\end{remark}

\section{Numerical experiments}
\label{sec:numerical_experiments}

This section verifies the accuracy and stability of the proposed first- and second-order Robin partitioned schemes. A manufactured solution examines temporal convergence and the influence of $L$, and a standard FPSI benchmark compares the partitioned and monolithic methods. Supplementary Material reports nonlinear moving-domain computations that illustrate applicability to more complex configurations.

The computations use FEniCS and FEniCSx~\cite{alnaes2015fenics,baratta2023dolfinx,logg2012automated} on conforming simplicial meshes. $[P_2]^{d}$--$P_1$ Taylor--Hood elements are used for the fluid variables, while continuous piecewise linear elements are used for the poroelastic displacement and velocity, Darcy velocity, and pore pressure. The consistent Darcy residual stabilization of Subsection~\ref{subsec:robin_split_weak_form} is used throughout, and the BDF2 start-up step is computed with the first-order method. Problem-specific parameters, meshes, time steps, and solvers are given below.

\subsection{Temporal convergence and influence of the Robin parameter}
\label{subsec:temporal_convergence}

\paragraph{Problem setting}
We first consider a manufactured solution to verify the temporal accuracy of the proposed first- and second-order Robin partitioned schemes and to investigate the influence of the Robin parameter. Let
$\Omega_f=(0,1)\times(0,1),~
\Omega_p=(0,1)\times(-1,0),~
\Gamma=(0,1)\times\{0\}.
\label{eq:manufactured_domains} $ The analytical solution is given by
\begin{align*}
\mathbf{u}_f
&=\pi\cos(\pi t)
\begin{pmatrix}
-3x+\cos y\\
y+1
\end{pmatrix},
&
p_f
&=e^t\sin(\pi x)\cos\left(\frac{\pi y}{2}\right)
+2\pi\cos(\pi t),\\
\boldsymbol{\eta}_p
&=\sin(\pi t)
\begin{pmatrix}
-3x+\cos y
\\
y+1
\end{pmatrix},
&
p_p
&=e^t\sin(\pi x)\cos\left(\frac{\pi y}{2}\right).
\end{align*}
It is easy to check that these analytical fields satisfy the interface conditions on $\Gamma$. The forcing terms, source functions, initial conditions, and exterior boundary data are obtained by substituting the analytical solution into the governing equations and the corresponding boundary conditions.

Moreover, the physical parameters are chosen as $
\rho_f=\mu_f=\rho_p=\mu_p=\lambda_p=c_0=\alpha=1,
~
\mathbf{K}=\mathbf{I},
~
\gamma_{\mathrm{BJS},j}=1$. And the final time is $T=0.1 \text{s}$. To reduce the influence of the spatial discretization error, the spatial mesh is fixed at $h=1/300$ for the first-order scheme and at $h=1/500$ for the second-order scheme. The time-step size is selected as $
\Delta t=0.1/N_t,~N_t=10,20,40,80$. Computations are performed for $L=0.01,~1,~100,~1000$.

\paragraph{Error measures}
Since this experiment is designed primarily to examine temporal convergence, we report the $L^2$ errors of four representative variables:
\begin{align*}
e_{\mathbf{u}_f}
&=\|\mathbf{u}_f(T)-\mathbf{u}_{f,h}^{N_t}\|_{L^2(\Omega_f)},
&
e_{p_f}
&=\|p_f(T)-p_{f,h}^{N_t}\|_{L^2(\Omega_f)},
\nonumber\\
e_{\boldsymbol{\eta}_p}
&=\|\boldsymbol{\eta}_p(T)-\boldsymbol{\eta}_{p,h}^{N_t}\|_{L^2(\Omega_p)},
&
e_{p_p}
&=\|p_p(T)-p_{p,h}^{N_t}\|_{L^2(\Omega_p)}.
\end{align*}
We also define the normal interface residual at the final time as
\begin{equation*}
\mathcal{R}_{\Gamma,h}^{N_t}
=
\|
\mathbf{u}_{f,h}^{N_t}\cdot\mathbf{n}_f
+
(
\boldsymbol{\xi}_{p,h}^{N_t}
+\mathbf{u}_{p,h}^{N_t}
)\cdot\mathbf{n}_p\|_{L^2(\Gamma)}.
\label{eq:final_interface_residual}
\end{equation*}
The errors of the poroelastic velocity and Darcy velocity, together with the $H^1$ errors of the fluid velocity and structure displacement, exhibit similar convergence behavior and are omitted from the main presentation for brevity.

\begin{figure}[!htbp]
\centering
\includegraphics[width=0.98\textwidth]{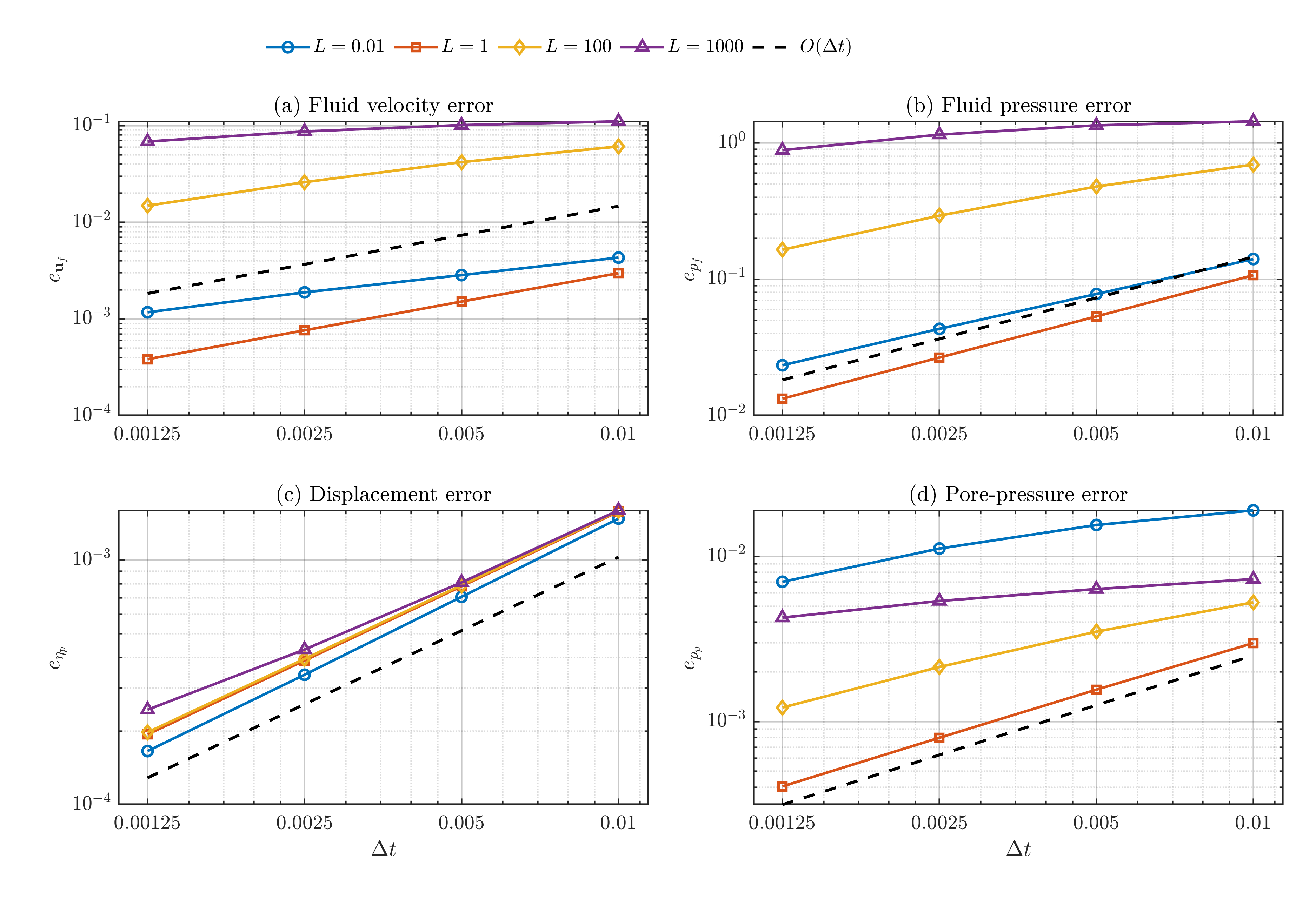}
\caption{Temporal errors of the first-order Robin partitioned scheme for different values of the Robin parameter $L$. The dashed lines indicate the reference rate $O(\Delta t)$.}
\label{fig:first_order_temporal_convergence}
\end{figure}

\paragraph{Temporal convergence of the primary variables}
Figures~\ref{fig:first_order_temporal_convergence} and \ref{fig:second_order_temporal_convergence} report the four errors against $\Delta t$; dashed lines indicate $O(\Delta t)$ and $O(\Delta t^2)$, respectively. The expected first-order behavior is clearest for $L=1$. The displacement error is comparatively insensitive to $L$, whereas the other errors show larger constants and occasional pre-asymptotic behavior for $L=0.01$ and $1000$.

\begin{figure}[!htbp]
\centering
\includegraphics[width=0.98\textwidth]{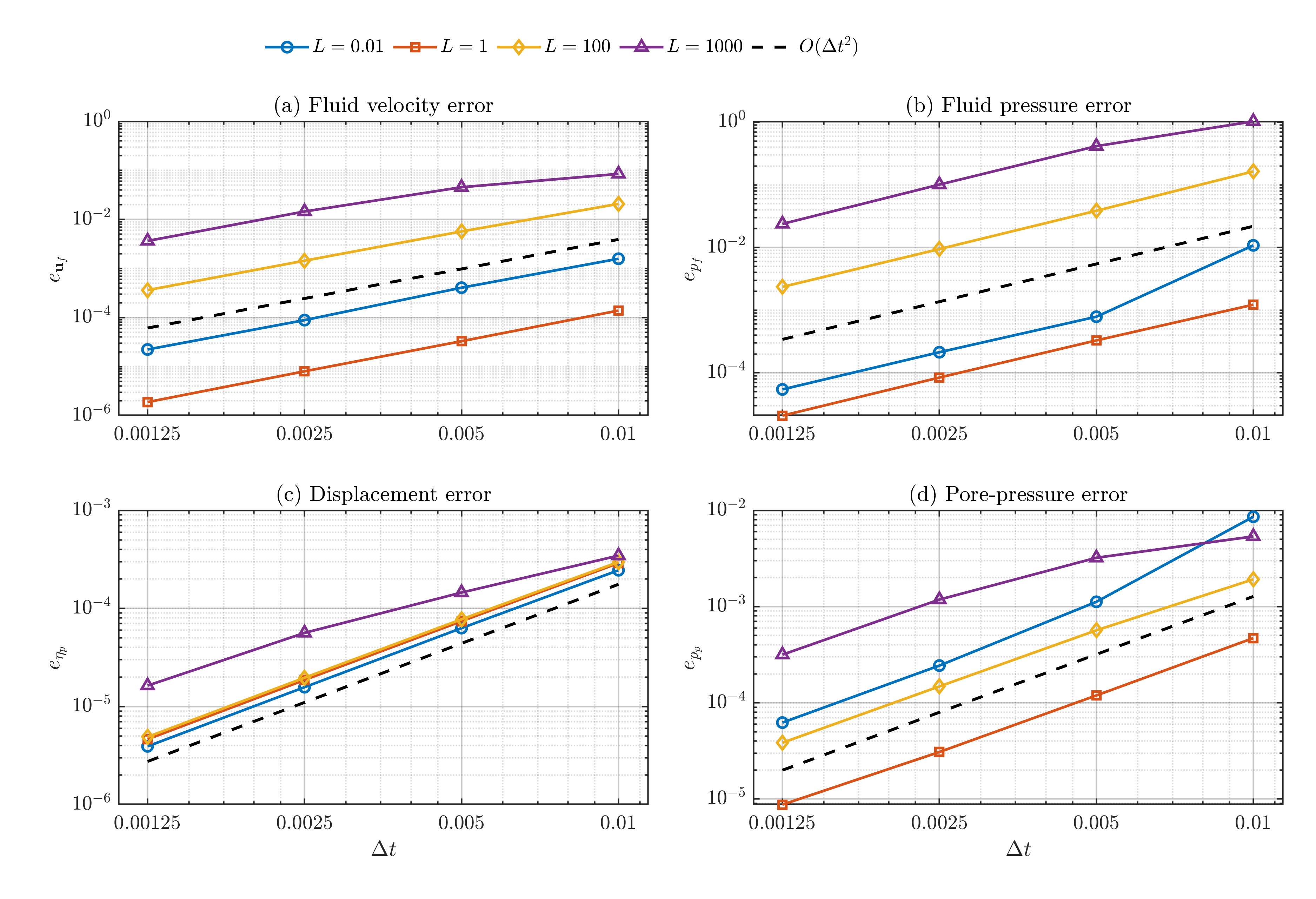}
\caption{Temporal errors of the second-order Robin partitioned scheme for different values of the Robin parameter $L$. The dashed lines indicate the reference rate $O(\Delta t^2)$.}
\label{fig:second_order_temporal_convergence}
\end{figure}

For the second-order method, the principal errors decay approximately quadratically for moderate $L$, especially for $L=1$ and $100$. Extreme values produce larger error constants, but the second-order trend becomes clear as $\Delta t$ decreases.

\paragraph{Convergence of the interface residual}
Figure~\ref{fig:interface_residual_convergence} shows the normal interface residual. For $L=1$, the first-order method has the expected linear decay, while the second-order method reduces the residual substantially faster and is close to the quadratic reference rate for $L=1$, $100$, and $1000$. The case $L=0.01$ is more sensitive on coarse time levels.

\begin{figure}[!htbp]
\centering
\includegraphics[width=0.98\textwidth]{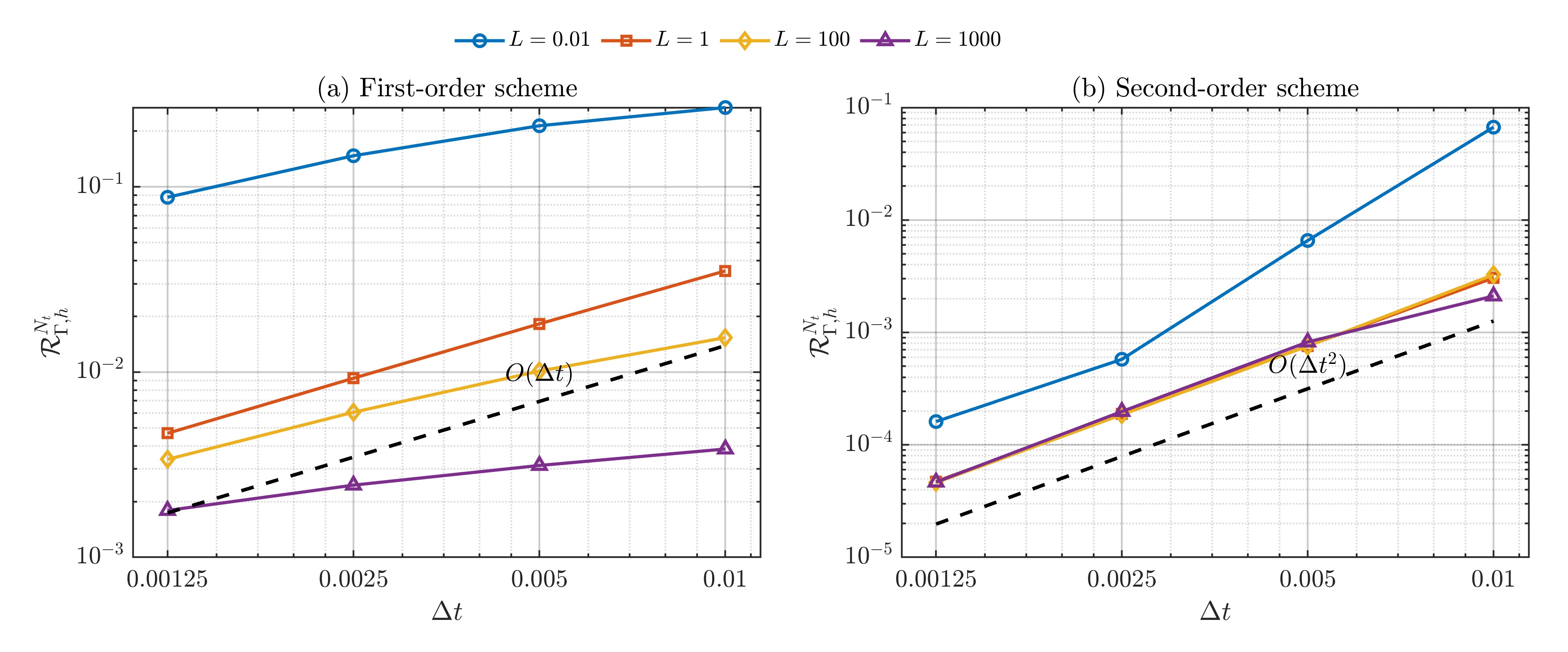}
\caption{Convergence of the normal interface residual for the first- and second-order Robin partitioned schemes. The dashed lines indicate the reference rates $O(\Delta t)$ and $O(\Delta t^2)$, respectively.}
\label{fig:interface_residual_convergence}
\end{figure}

\paragraph{Influence of the Robin parameter}
Table~\ref{tab:robin_parameter_comparison} reports the principal errors and interface residual at the finest time step, $N_t=80$, which confirms the trade-off in selecting $L$: a small $L$ increases the pore-pressure and interface errors, while a very large $L$ increases the fluid errors. A smaller interface residual therefore does not by itself imply a more accurate coupled solution. Among the tested values, $L=1$ gives the most balanced accuracy, and the second-order method is markedly more accurate at the same $N_t$.

\begin{table}[!htbp]
\centering
\caption{Errors and normal interface residual at $T=0.1$ with $N_t=80$ for different values of $L$.}
\label{tab:robin_parameter_comparison}
\small
\setlength{\tabcolsep}{3.8pt}
\resizebox{\textwidth}{!}{%
\begin{tabular}{c c c c c c c}
\hline
Scheme
& $L$
& $e_{\mathbf{u}_f}$
& $e_{p_f}$
& $e_{\boldsymbol{\eta}_p}$
& $e_{p_p}$
& $\mathcal{R}_{\Gamma,h}^{N_t}$
\\
\hline
First order
& $0.01$
& $1.1735\times10^{-3}$
& $2.3435\times10^{-2}$
& $1.6560\times10^{-4}$
& $7.0185\times10^{-3}$
& $8.7826\times10^{-2}$
\\
& $1$
& $3.8348\times10^{-4}$
& $1.3331\times10^{-2}$
& $1.9371\times10^{-4}$
& $4.0487\times10^{-4}$
& $4.6713\times10^{-3}$
\\
& $100$
& $1.4863\times10^{-2}$
& $1.6480\times10^{-1}$
& $1.9769\times10^{-4}$
& $1.2162\times10^{-3}$
& $3.3796\times10^{-3}$
\\
& $1000$
& $6.8544\times10^{-2}$
& $8.8567\times10^{-1}$
& $2.4448\times10^{-4}$
& $4.2592\times10^{-3}$
& $1.7907\times10^{-3}$
\\
\hline
Second order
& $0.01$
& $2.2324\times10^{-5}$
& $5.4321\times10^{-5}$
& $3.9310\times10^{-6}$
& $6.2258\times10^{-5}$
& $1.6112\times10^{-4}$
\\
& $1$
& $1.8902\times10^{-6}$
& $2.0699\times10^{-5}$
& $4.6428\times10^{-6}$
& $8.7369\times10^{-6}$
& $4.7237\times10^{-5}$
\\
& $100$
& $3.6274\times10^{-4}$
& $2.3510\times10^{-3}$
& $4.9042\times10^{-6}$
& $3.8443\times10^{-5}$
& $4.6195\times10^{-5}$
\\
& $1000$
& $3.6496\times10^{-3}$
& $2.3818\times10^{-2}$
& $1.6344\times10^{-5}$
& $3.1622\times10^{-4}$
& $4.6664\times10^{-5}$
\\
\hline
\end{tabular}%
}
\end{table}

\FloatBarrier

\FloatBarrier

\subsection{Pressure-wave benchmark and comparison with the monolithic scheme}
\label{subsec:pressure_wave_benchmark}

\paragraph{Geometry and model}
We consider the classical pressure-wave benchmark for an idealized arterial segment~\cite{he2026locking,parrow2026stability}. A time-dependent inlet pressure drives a downstream pulse, whose fluid traction deforms the poroelastic wall. By symmetry, only the upper channel and upper wall are modeled; see Figure~\ref{fig:pressure_wave_geometry}. The subdomains are $\Omega_f=(0,H)\times(0,R)$ and $\Omega_p=(0,H)\times(R,R+r_p)$, with $H=6~\mathrm{cm}$, $R=0.5~\mathrm{cm}$, and $r_p=0.1~\mathrm{cm}$. The interface is $\Gamma=(0,H)\times\{R\}$, the symmetry boundary is $\Gamma_f^{\mathrm{sym}}=(0,H)\times\{0\}$, the fluid inlet and outlet are $\Gamma_f^{\mathrm{in}}=\{0\}\times(0,R)$ and $\Gamma_f^{\mathrm{out}}=\{H\}\times(0,R)$, and the wall boundaries are $\Gamma_p^D=(\{0\}\cup\{H\})\times(R,R+r_p)$ and $\Gamma_p^{\mathrm{ext}}=(0,H)\times\{R+r_p\}$.

\begin{figure}[!htbp]
\centering
\includegraphics[width=0.94\textwidth]{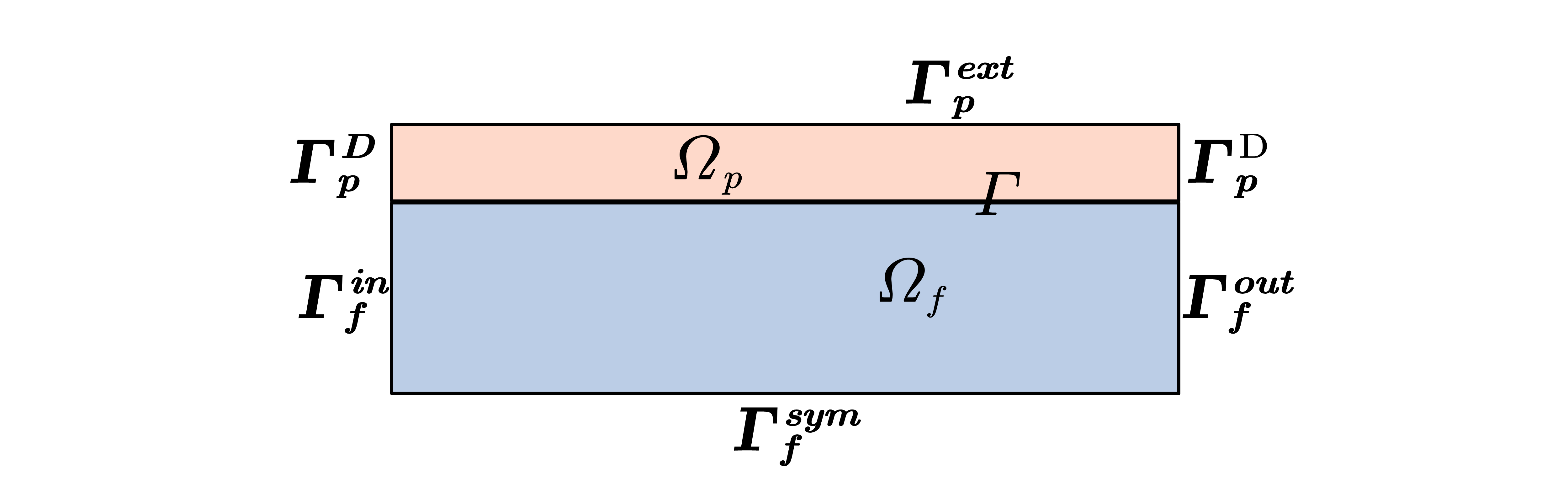}
\caption{Computational geometry and boundary conditions for the pressure-wave benchmark.}
\label{fig:pressure_wave_geometry}
\end{figure}

\FloatBarrier

To account for the circumferential recoil of the three-dimensional cylindrical wall in the two-dimensional formulation, the poroelastic momentum equation is augmented by a linear spring term:
\begin{align*}
\rho_p\partial_t\boldsymbol{\xi}_p
-\nabla\cdot\boldsymbol{\sigma}_p(\boldsymbol{\eta}_p,p_p)
+\beta\boldsymbol{\eta}_p
=\mathbf{f}_p
\qquad\text{in }\Omega_p\times(0,T].
\end{align*}
The coefficient $\beta>0$ represents the elastic recoil associated with the circumferential strain in the corresponding axisymmetric three-dimensional configuration. All the remaining fluid, Darcy, mass-conservation, and interface equations are the same as those introduced in Section~\ref{sec:fully_coupled_model}. The body forces, and volumetric source terms are set to zero, i.e. $\mathbf{f}_f=\mathbf{f}_p=\mathbf{f}_{d}=\mathbf{0},~
\phi_f=\phi_d=0$.

\paragraph{Initial and boundary conditions}
Initially, the system is at rest:
\begin{equation*}
\mathbf{u}_f(\cdot,0)=\mathbf{0},
\qquad
\boldsymbol{\eta}_p(\cdot,0)=\mathbf{0},
\qquad
\boldsymbol{\xi}_p(\cdot,0)=\mathbf{0},
\qquad
p_p(\cdot,0)=0.
\label{eq:benchmark_initial_conditions}
\end{equation*}
The final time is set to $T=0.014~\mathrm{s}$, and the numerical results are reported at the four observation times $t=0.0035,~ 0.007,~0.0105,~0.014~\mathrm{s}$,
which capture both the generation and the downstream propagation of the pressure pulse.

At the fluid inlet, a time-dependent pressure pulse is prescribed, whereas a homogeneous traction condition is imposed at the outlet:
\begin{align*}
\boldsymbol{\sigma}_f(\mathbf{u}_f,p_f)\mathbf{n}_f
&=-p_{\mathrm{in}}(t)\mathbf{n}_f
&&\text{on }\Gamma_f^{\mathrm{in}}\times(0,T],
\\
\boldsymbol{\sigma}_f(\mathbf{u}_f,p_f)\mathbf{n}_f
&=\mathbf{0}
&&\text{on }\Gamma_f^{\mathrm{out}}\times(0,T].
\end{align*}
The inlet pressure is given by
\begin{equation*}
p_{\mathrm{in}}(t)=
\begin{cases}
\dfrac{P_{\max}}{2}
\left[
1-\cos\left(\dfrac{2\pi t}{T_{\max}}\right)
\right],
&0\leq t\leq T_{\max},\\[1.2ex]
0,
&t>T_{\max},
\end{cases}
\end{equation*}
with $P_{\max}=13334~\mathrm{dyn}/\mathrm{cm}^2,~
T_{\max}=0.003~\mathrm{s}$.
For the poroelastic structure, clamping conditions are enforced at both ends:
\begin{equation*}
\boldsymbol{\eta}_p=\mathbf{0}
\qquad
\text{on }\Gamma_p^D\times(0,T].
\end{equation*}
The exterior boundaries of the porous medium are assumed to be impermeable:
\begin{equation*}
\mathbf{u}_p\cdot\mathbf{n}_p=0
\qquad
\text{on }
\left(
\Gamma_p^D\cup\Gamma_p^{\mathrm{ext}}
\right)\times(0,T].
\end{equation*}
On the lower fluid boundary, we impose symmetry conditions:
\begin{equation*}
\mathbf{u}_f\cdot\mathbf{n}_f=0,
\qquad
\boldsymbol{\sigma}_f(\mathbf{u}_f,p_f)\mathbf{n}_f
\cdot\boldsymbol{\tau}_f=0
\qquad\text{on }\Gamma_f^{\mathrm{sym}}\times(0,T].
\label{eq:benchmark_symmetry_conditions}
\end{equation*}

\paragraph{Physical and numerical parameters}
The physical parameters are summarized in Table~\ref{tab:pressure_wave_parameters}. The first-order Robin partitioned scheme is compared with the backward-Euler monolithic discretization, whereas the second-order Robin partitioned scheme is compared with the corresponding BDF2 monolithic discretization. The strongly coupled reference solution is computed using the
monolithic formulation of~\cite{cesmelioglu2017analysis}. The same physical parameters and temporal resolution are used in each pair of comparisons. Unless otherwise specified, $
\Delta t=10^{-4}~\mathrm{s},~L=500$.

\begin{table}[!htbp]
\centering
\caption{Physical parameters for the pressure-wave benchmark}
\label{tab:pressure_wave_parameters}
\small
\begin{tabular}{l c c c}
\hline
Parameter & Symbol & Unit & Value\\
\hline
Fluid density
& $\rho_f$
& $\mathrm{g}/\mathrm{cm}^3$
& $1.0$\\
Poroelastic wall density
& $\rho_p$
& $\mathrm{g}/\mathrm{cm}^3$
& $1.1$\\
Dynamic viscosity
& $\mu_f$
& $\mathrm{g}/(\mathrm{cm}\,\mathrm{s})$
& $0.035$\\
Spring coefficient
& $\beta$
& $\mathrm{dyn}/\mathrm{cm}^4$
& $4\times10^6$\\
Storage coefficient
& $c_0$
& $\mathrm{cm}^2/\mathrm{dyn}$
& $10^{-3}$\\
Permeability tensor
& $\mathbf{K}$
& $\mathrm{cm}^2$
& $10^{-6}\mathbf{I}$\\
Shear modulus
& $\mu_p$
& $\mathrm{dyn}/\mathrm{cm}^2$
& $5.575\times10^5$\\
First Lam\'e coefficient
& $\lambda_p$
& $\mathrm{dyn}/\mathrm{cm}^2$
& $1.7\times10^6$\\
Biot--Willis coefficient
& $\alpha$
& --
& $1$\\
BJS coefficient
& $\gamma_{\mathrm{BJS}}$
& --
& $1$\\
\hline
\end{tabular}
\end{table}

\FloatBarrier

\begin{figure}[!htbp]
\centering
\includegraphics[width=\textwidth,height=0.5\textwidth]
{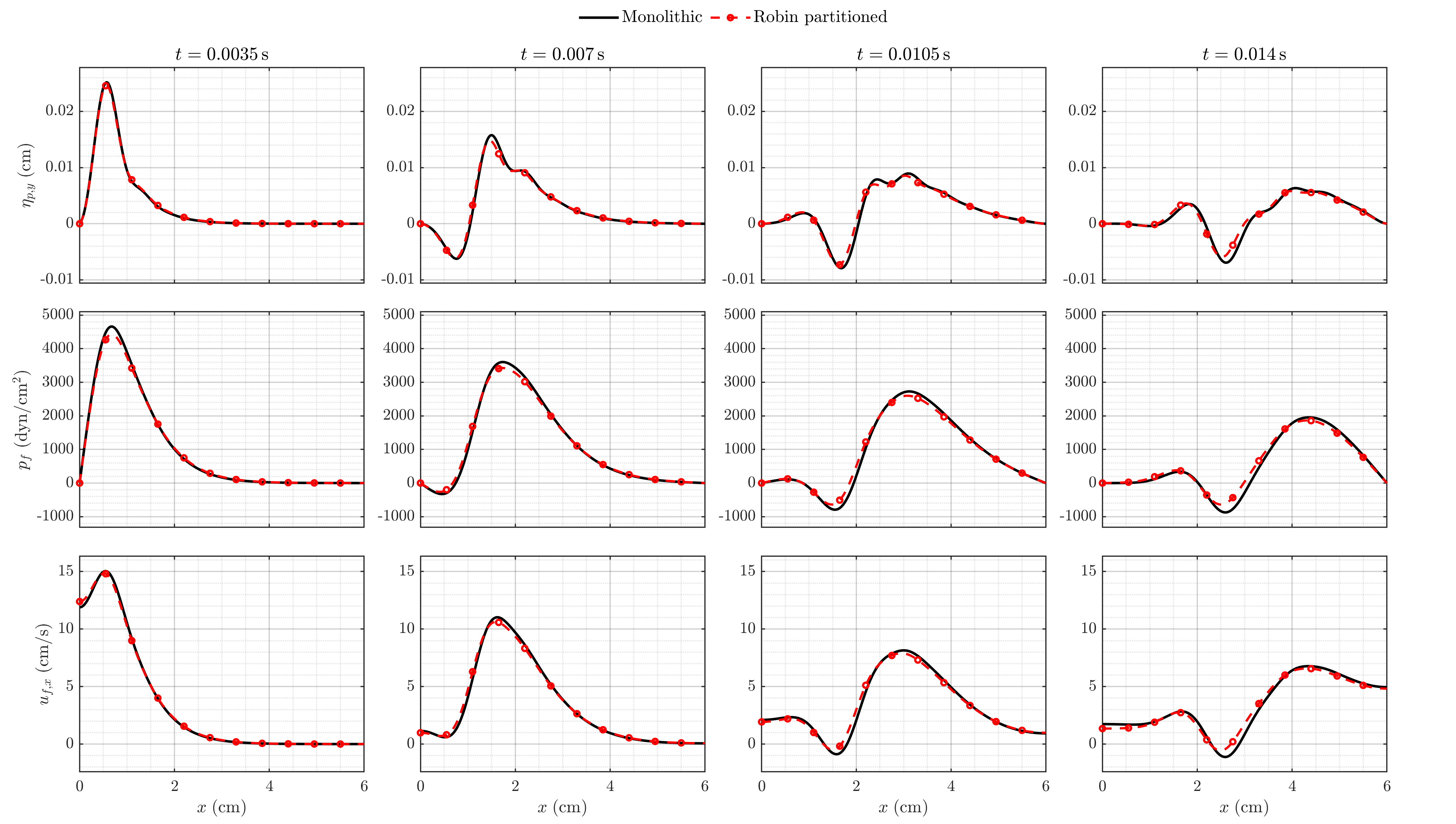}
\caption{Comparison of the first-order Robin partitioned scheme and the backward-Euler monolithic method for the pressure-wave benchmark. From top to bottom, the rows show the vertical interface displacement $\eta_{p,y}$, the fluid pressure $p_f$, and the axial fluid velocity $u_{f,x}$ at four representative time instants.}
\label{fig:benchmark_first_order_comparison}
\end{figure}
\FloatBarrier

\paragraph{Comparison with the monolithic schemes}
We compare the vertical interface displacement $\eta_{p,y}$, and the fluid pressure $p_f$ and axial velocity $u_{f,x}$ along the symmetry line at the four observation times.

Figure~\ref{fig:benchmark_first_order_comparison} shows close agreement between the first-order partitioned and backward-Euler monolithic profiles. The partitioned method reproduces the wave location, phase, and overall shape; only small differences near late-time extrema are visible.

\begin{figure}[!htbp]
\centering
\includegraphics[width=\textwidth,height=0.5\textwidth]{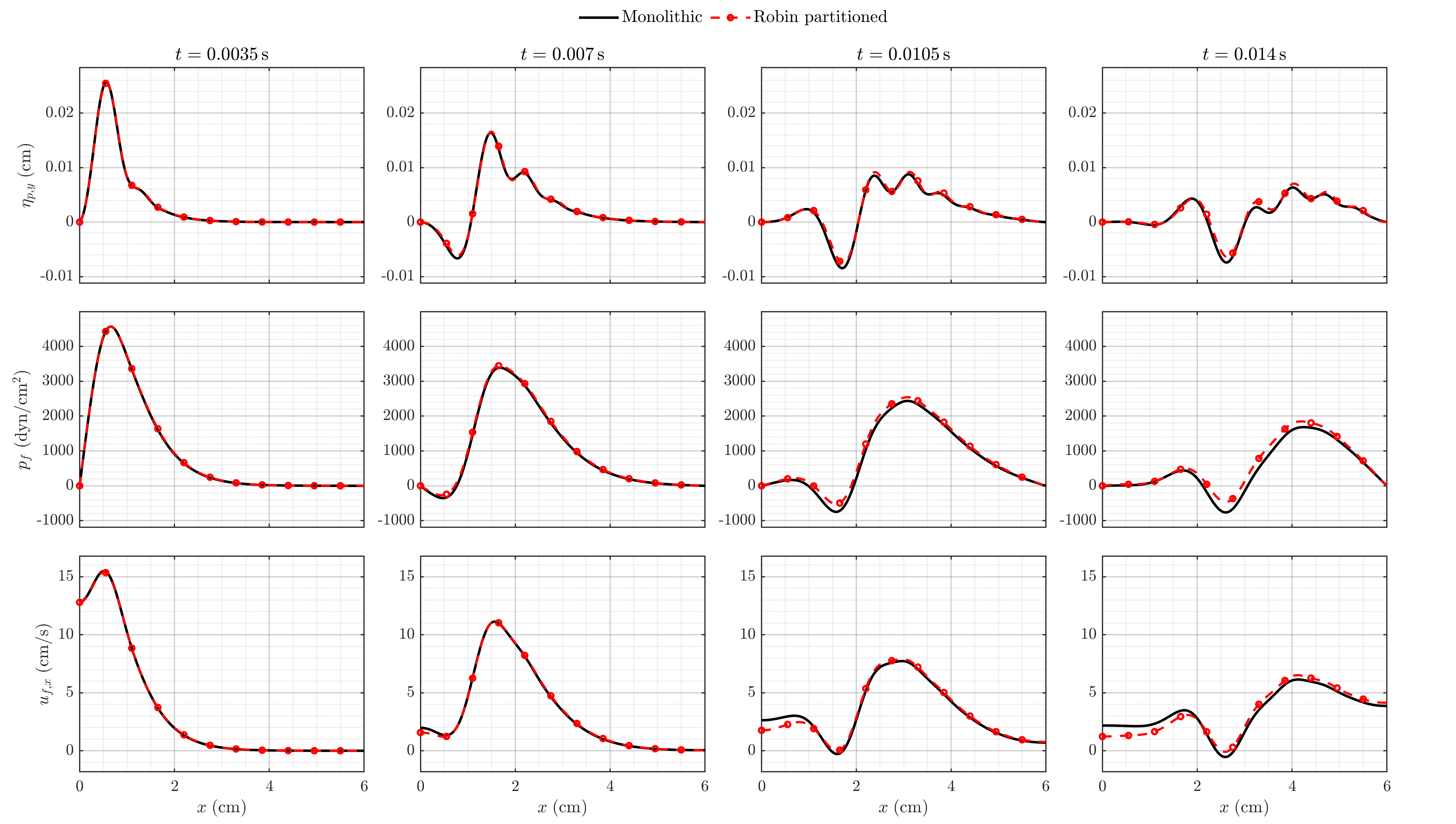}
\caption{Comparison of the second-order Robin partitioned scheme and the BDF2 monolithic method for the pressure-wave benchmark. From top to bottom, the rows show the vertical interface displacement $\eta_{p,y}$, the fluid pressure $p_f$, and the axial fluid velocity $u_{f,x}$ at four representative time instants.}
\label{fig:benchmark_second_order_comparison}
\end{figure}
\FloatBarrier
Figure~\ref{fig:benchmark_second_order_comparison} likewise shows close agreement with the BDF2 monolithic solution: the displacement features and the principal pressure and velocity peaks occur at essentially the same locations. At the latest time, the pulse is near the outlet, so local details may be affected by the outflow treatment. Overall, both partitioned schemes reproduce the principal pressure-wave dynamics without a simultaneous fully coupled solve.

\FloatBarrier

\section{Conclusions}

We developed decoupled first- and second-order Robin partitioned methods for FPSI problems. Backward Euler and BDF2 discretizations, together with explicit interface extrapolation and an auxiliary-variable update, yield noniterative schemes without direct evaluation of the interface stress. The discrete energy analysis proves unconditional stability; in particular, it gives an unconditional stability result for a second-order noniterative FPSI partitioned scheme. A rigorous error analysis of the second-order method remains open.

The numerical tests recover the expected temporal orders, clarify the influence of the Robin parameter, and show close agreement with monolithic solutions for the pressure-wave benchmark. Supplementary nonlinear moving-domain simulations further illustrate the applicability of the algorithms to more complex FPSI configurations.

\section*{Acknowledgements}
X. Yue and T. Wick were partially supported by the Deutsche Forschungsgemeinschaft (DFG, German Research Foundation) under Grant No.~548064929. J. Zhang was supported by the National Natural Science Foundation of China (NSFC) under Grant No.~12171376, the Fundamental Research Funds for the Central Universities under Grant No.~2042021kf0050, and WHU-2022-SYJS-0002. H. Zheng was partially supported by the NSFC under Grant No.~12471406 and the Science and Technology Commission of Shanghai Municipality under Grant No.~22DZ2229014.

\appendix
\setcounter{figure}{0}
\setcounter{table}{0}
\renewcommand{\thefigure}{S\arabic{figure}}
\renewcommand{\thetable}{S\arabic{table}}
\section{Proof of Theorem \ref{thm:first_order_stability}}
We first choose $(\boldsymbol{\omega}_{p,h},\mathbf{v}_{p,h},q_{p,h})=(\boldsymbol{\xi}_{p,h}^{n+1},\mathbf{u}_{p,h}^{n+1},p_{p,h}^{n+1})$ as the test functions in \eqref{eq:first_order_poroelastic_problem}. Then, applying the identity \eqref{eq:first_order_algebraic_identity} to the resulting equation, we obtain
\begin{align}
&d_{t}\mathcal{E}_{poro}^{n+1}+\mathcal{J}_{poro}^{n+1}+L\langle(\boldsymbol{\xi}_{p,h}^{n+1}+\mathbf{u}_{p,h}^{n+1})\cdot\mathbf{n}_{p},(\boldsymbol{\xi}_{p,h}^{n+1}+\mathbf{u}_{p,h}^{n+1})\cdot\mathbf{n}_{p}\rangle_{\Gamma}\nonumber\\
		&=\langle\theta_{h}^{n},(\boldsymbol{\xi}_{p,h}^{n+1}+\mathbf{u}_{p,h}^{n+1})\cdot\mathbf{n}_{p}\rangle_{\Gamma}
        -\sum_{j=1}^{d-1}\gamma_{\mathrm{BJS},j}
        \left\langle
        (\boldsymbol{\xi}_{p,h}^{n+1}-\mathbf{u}_{f,h}^{n})\cdot\boldsymbol{\tau}_{j},
        \boldsymbol{\xi}_{p,h}^{n+1}\cdot\boldsymbol{\tau}_{j}
        \right\rangle_{\Gamma}.\label{eq:first_order_poroelastic_energy}
\end{align}
Similarly, choosing $(\mathbf{v}_{f,h},q_{f,h})=(\mathbf{u}_{f,h}^{n+1},p_{f,h}^{n+1})$ as the test functions in \eqref{eq:first_order_fluid_problem} yields
\begin{align}
		&d_{t}\mathcal{E}_{fluid}^{n+1}+\mathcal{J}_{fluid}^{n+1}+L\|\mathbf{u}_{f,h}^{n+1}\cdot\mathbf{n}_{f}\|_{L^{2}(\Gamma)}^{2}\nonumber\\
      &=\langle\theta_{h}^{n}-2L(\boldsymbol{\xi}_{p,h}^{n+1}
        +\mathbf{u}_{p,h}^{n+1})\cdot\mathbf{n}_{p},\mathbf{u}_{f,h}^{n+1}\cdot\mathbf{n}_{f}\rangle_{\Gamma}\nonumber\\
        &\quad
        -\sum_{j=1}^{d-1}\gamma_{\mathrm{BJS},j}
        \left\langle
        (\mathbf{u}_{f,h}^{n+1}-\boldsymbol{\xi}_{p,h}^{n})\cdot\boldsymbol{\tau}_{j},
        \mathbf{u}_{f,h}^{n+1}\cdot\boldsymbol{\tau}_{j}
        \right\rangle_{\Gamma}.\label{eq:first_order_fluid_energy}
\end{align}
Next, choosing $\psi_{h}=\theta_{h}^{n+1}$ as the test function in \eqref{eq:first_order_interface_update_algorithm}, then using \eqref{eq:interface_projection} and Cauchy-Schwarz inequality for the resulting equation, we obtain
\begin{align}
\|\theta_{h}^{n+1}\|_{L^{2}(\Gamma)}^{2}&=\langle\theta_{h}^{n}-L(\boldsymbol{\xi}_{p,h}^{n+1}+\mathbf{u}_{p,h}^{n+1})\cdot\mathbf{n}_{p}-L\Pi_{\Sigma}\mathbf{u}_{f,h}^{n+1}\cdot\mathbf{n}_{f},\theta_{h}^{n+1}\rangle_{\Gamma}\nonumber\\
&=\langle\theta_{h}^{n}-L(\boldsymbol{\xi}_{p,h}^{n+1}+\mathbf{u}_{p,h}^{n+1})\cdot\mathbf{n}_{p}-L\mathbf{u}_{f,h}^{n+1}\cdot\mathbf{n}_{f},\theta_{h}^{n+1}\rangle_{\Gamma}\nonumber\\
&\leq\|\theta_{h}^{n}-L(\boldsymbol{\xi}_{p,h}^{n+1}+\mathbf{u}_{p,h}^{n+1})\cdot\mathbf{n}_{p}-L\mathbf{u}_{f,h}^{n+1}\cdot\mathbf{n}_{f}\|_{L^{2}(\Gamma)}\|\theta_{h}^{n+1}\|_{L^{2}(\Gamma)}.\label{eq:first_order_theta_norm_inequality}
\end{align}
Using identities \eqref{eq:first_order_algebraic_identity}-\eqref{eq:first_order_polarization_identity}, we obtain the following equalities:
\begin{align}
     &\langle\theta_{h}^{n}-\frac{L}{2}(\boldsymbol{\xi}_{p,h}^{n+1}+\mathbf{u}_{p,h}^{n+1})\cdot\mathbf{n}_{p},(\boldsymbol{\xi}_{p,h}^{n+1}+\mathbf{u}_{p,h}^{n+1})\cdot\mathbf{n}_{p}\rangle_{\Gamma}\nonumber\\
    &=\frac{1}{2L}\big(\|\theta_{h}^{n}\|_{L^{2}(\Gamma)}^{2}-\|\theta_{h}^{n}-L(\boldsymbol{\xi}_{p,h}^{n+1}+\mathbf{u}_{p,h}^{n+1})\cdot\mathbf{n}_{p}\|_{L^{2}(\Gamma)}^{2}\big),\label{eq:first_order_normal_identity_one}\\
    &\langle\theta_{h}^{n}-L(\boldsymbol{\xi}_{p,h}^{n+1}
        +\mathbf{u}_{p,h}^{n+1})\cdot\mathbf{n}_{p}-\frac{L}{2}\mathbf{u}_{f,h}^{n+1}\cdot\mathbf{n}_{f},\mathbf{u}_{f,h}^{n+1}\cdot\mathbf{n}_{f}\rangle_{\Gamma}\nonumber\\
        &=\frac{1}{2L}\big(\|\theta_{h}^{n}-L(\boldsymbol{\xi}_{p,h}^{n+1}
        +\mathbf{u}_{p,h}^{n+1})\cdot\mathbf{n}_{p}\|_{L^{2}(\Gamma)}^{2}\nonumber\\
        &\quad-\|\theta_{h}^{n}-L\mathbf{u}_{f,h}^{n+1}\cdot\mathbf{n}_{f}-L(\boldsymbol{\xi}_{p,h}^{n+1}
        +\mathbf{u}_{p,h}^{n+1})\cdot\mathbf{n}_{p}\|_{L^{2}(\Gamma)}^{2}\big),\label{eq:first_order_normal_identity_two}\\
    &\sum_{j=1}^{d-1}\gamma_{\mathrm{BJS},j}
    \left\langle
    (\boldsymbol{\xi}_{p,h}^{n+1}-\mathbf{u}_{f,h}^{n})\cdot\boldsymbol{\tau}_{j},
    \boldsymbol{\xi}_{p,h}^{n+1}\cdot\boldsymbol{\tau}_{j}
    \right\rangle_{\Gamma}\nonumber\\
	 	&=\frac{1}{2}\sum_{j=1}^{d-1}\gamma_{\mathrm{BJS},j}\big(\|\boldsymbol{\xi}_{p,h}^{n+1}\cdot\boldsymbol{\tau}_{j}\|_{L^{2}(\Gamma)}^{2}-\|\mathbf{u}_{f,h}^{n}\cdot\boldsymbol{\tau}_{j}\|_{L^{2}(\Gamma)}^{2}\nonumber\\
        &\quad+\|(\boldsymbol{\xi}_{p,h}^{n+1}-\mathbf{u}_{f,h}^{n})\cdot\boldsymbol{\tau}_{j}\|_{L^{2}(\Gamma)}^{2}\big),\label{eq:first_order_bjs_identity_poroelastic}\\
    &\sum_{j=1}^{d-1}\gamma_{\mathrm{BJS},j}
    \left\langle
    (\mathbf{u}_{f,h}^{n+1}-\boldsymbol{\xi}_{p,h}^{n})\cdot\boldsymbol{\tau}_{j},
    \mathbf{u}_{f,h}^{n+1}\cdot\boldsymbol{\tau}_{j}
    \right\rangle_{\Gamma}\nonumber\\
	 	&=\frac{1}{2}\sum_{j=1}^{d-1}\gamma_{\mathrm{BJS},j}\big(\|\mathbf{u}_{f,h}^{n+1}\cdot\boldsymbol{\tau}_{j}\|_{L^{2}(\Gamma)}^{2}-\|\boldsymbol{\xi}_{p,h}^{n}\cdot\boldsymbol{\tau}_{j}\|_{L^{2}(\Gamma)}^{2}\nonumber\\
        &\quad+\|(\mathbf{u}_{f,h}^{n+1}-\boldsymbol{\xi}_{p,h}^{n})\cdot\boldsymbol{\tau}_{j}\|_{L^{2}(\Gamma)}^{2}\big).\label{eq:first_order_bjs_identity_fluid}
\end{align}
Moreover, it easy to check that the following identity holds:
\begin{align}
    &\frac{L}{2}\langle\mathbf{u}_{f,h}^{n+1}\cdot\mathbf{n}_{f},\mathbf{u}_{f,h}^{n+1}\cdot\mathbf{n}_{f}\rangle_{\Gamma}+\frac{L}{2}\langle(\boldsymbol{\xi}_{p,h}^{n+1}+\mathbf{u}_{p,h}^{n+1})\cdot\mathbf{n}_{p},(\boldsymbol{\xi}_{p,h}^{n+1}+\mathbf{u}_{p,h}^{n+1})\cdot\mathbf{n}_{p}\rangle_{\Gamma}\nonumber\\
    &\quad+L\langle(\boldsymbol{\xi}_{p,h}^{n+1}+\mathbf{u}_{p,h}^{n+1})\cdot\mathbf{n}_{p},\mathbf{u}_{f,h}^{n+1}\cdot\mathbf{n}_{f}\rangle_{\Gamma}\nonumber\\
    &=\frac{L}{2}\|\mathbf{u}_{f,h}^{n+1}\cdot\mathbf{n}_{f}+(\boldsymbol{\xi}_{p,h}^{n+1}+\mathbf{u}_{p,h}^{n+1})\cdot\mathbf{n}_{p}\|_{L^{2}(\Gamma)}^{2}.\label{eq:first_order_normal_residual_identity}
\end{align}
Adding \eqref{eq:first_order_poroelastic_energy} and \eqref{eq:first_order_fluid_energy}, and then substituting \eqref{eq:first_order_theta_norm_inequality}--\eqref{eq:first_order_normal_residual_identity} into the resulting identity, we arrive at
\begin{align}
	 	d_{t}(\mathcal{E}_{h}^{n+1}+\Delta t\mathcal{I}_{h}^{n+1})+\mathcal{J}_{h}^{n+1}+\mathcal{D}_{h}^{n+1}\leq 0.\label{eq:first_order_one_step_energy}
\end{align}
Finally, applying $\Delta t\sum_{n=0}^{\ell}$ to both sides of \eqref{eq:first_order_one_step_energy} yields \eqref{eq:first_order_stability_estimate}. This completes the proof.

\section{Nonlinear simulations on moving fluid domains}
\label{subsec:nonlinear_moving_domains}

The preceding numerical experiments concern the linear Stokes--Biot
system posed on fixed computational domains. In this subsection, we
investigate the applicability of the proposed Robin partitioning
strategy to nonlinear fluid--poroelastic interaction problems with
moving fluid domains. The fluid motion is governed by the
incompressible Navier--Stokes equations, and the motion of the fluid
mesh is described by an arbitrary Lagrangian--Eulerian (ALE) mapping~\cite{donea2004arbitrary,tezduyar2001finite}.

Let $\widehat{\Omega}_f$ denote the reference fluid domain and let
$\widehat{\Gamma}$ be the fluid--poroelastic interface in the reference
configuration. The current fluid domain is defined by
\begin{equation*}
\Omega_f(t)
=
\mathcal{A}_t(\widehat{\Omega}_f),
\qquad
\mathcal{A}_t(\widehat{\mathbf{x}})
=
\widehat{\mathbf{x}}
+
\boldsymbol{\eta}_{\mathrm{ALE}}
(\widehat{\mathbf{x}},t),
\end{equation*}
where $\boldsymbol{\eta}_{\mathrm{ALE}}$ denotes the fluid mesh
displacement. The associated mesh velocity is
\begin{equation*}
\mathbf{w}_f
=
\partial_t\mathcal{A}_t
\circ\mathcal{A}_t^{-1}.
\end{equation*}
Accordingly, the fluid acceleration in the current configuration is
written in ALE form as
\begin{equation*}
\left.
\partial_t\mathbf{u}_f
\right|_{\widehat{\mathbf{x}}}
+
\big(
(\mathbf{u}_f-\mathbf{w}_f)\cdot\nabla
\big)\mathbf{u}_f.
\end{equation*}
The Navier--Stokes equations on the moving fluid domain therefore take
the form
\begin{align*}
\rho_f
\left[
\left.
\partial_t\mathbf{u}_f
\right|_{\widehat{\mathbf{x}}}
+
\big(
(\mathbf{u}_f-\mathbf{w}_f)\cdot\nabla
\big)\mathbf{u}_f
\right]
-\nabla\cdot
\boldsymbol{\sigma}_f(\mathbf{u}_f,p_f)
&=\mathbf{f}_f
&&\text{in }\Omega_f(t),\\
\nabla\cdot\mathbf{u}_f
&=0
&&\text{in }\Omega_f(t).
\end{align*}

The fluid mesh displacement is obtained by extending the poroelastic
interface displacement into the reference fluid domain. In generic
form, the mesh-extension problem used in the computations can be
written as
\begin{align*}
-\widehat{\nabla}\cdot
\left(
\alpha_{\mathrm{ALE}}
\widehat{\nabla}
\boldsymbol{\eta}_{\mathrm{ALE}}
\right)
&=\mathbf{0}
&&\text{in }\widehat{\Omega}_f,\\
\boldsymbol{\eta}_{\mathrm{ALE}}
&=\boldsymbol{\eta}_p
&&\text{on }\widehat{\Gamma},\\
\boldsymbol{\eta}_{\mathrm{ALE}}
&=\mathbf{0}
&&\text{on the fixed part of }
\partial\widehat{\Omega}_f.
\end{align*}
Here $\alpha_{\mathrm{ALE}}>0$ is the mesh-extension coefficient; the
choice $\alpha_{\mathrm{ALE}}=1$ corresponds to a harmonic extension,
whereas a spatially varying coefficient may be employed to protect the
smaller elements from excessive distortion.

The poroelastic equations and the physical interface conditions retain
the form introduced in Section~\ref{sec:fully_coupled_model}, with the
fluid quantities evaluated on the current interface
$\Gamma(t)=\mathcal{A}_t(\widehat{\Gamma})$. The Robin data are
constructed from previously available time levels, so the fluid and
poroelastic subproblems preserve the independent solution structure of
the proposed partitioned schemes.

We emphasize that the stability and error analysis developed in the
preceding sections is established for the linear Stokes--Biot problem
on fixed domains. The computations below constitute a numerical
extension of the proposed first- and second-order Robin partitioned
schemes to nonlinear Navier--Stokes--Biot systems on moving domains. A
rigorous analysis of the fully nonlinear ALE formulation is beyond the
scope of the present work.

The same families of finite element spaces as in the preceding
experiments are employed. The fluid velocity and pressure are
approximated by the $[P_2]^{d}$--$P_1$ Taylor--Hood pair, while continuous
piecewise linear elements are used for the poroelastic displacement,
structure velocity, Darcy velocity, pore pressure, and ALE mesh
displacement. The consistent Darcy residual stabilization is retained.
All physical quantities are given in the CGS unit system, and the common parameter set is
\begin{equation*}
\begin{aligned}
\rho_f &= 1~\mathrm{g/cm^3}, &
\mu_f &= 5\times10^{-3}~\mathrm{g/(cm\,s)}, &
\rho_p &= 1.2~\mathrm{g/cm^3},\\
\mu_p &= 5.575\times10^{3}~\mathrm{dyne/cm^2}, &
\lambda_p &= 1.7\times10^{4}~\mathrm{dyne/cm^2}, &
\alpha &= 1,\\
C_0 &= 10^{-3}~\mathrm{cm^2/dyne}, &
K &= 10^{-4}~\mathrm{cm^3\,s/g}, &
C_{\mathrm{BJS}} &= 1~\mathrm{g/(cm^2\,s)}.
\end{aligned}
\end{equation*}
The first two examples also employ the same temporal and Robin parameters,
\begin{equation*}
\Delta t=5\times10^{-4}~\mathrm{s},
\qquad
T=5~\mathrm{s},
\qquad
L_1=L_2=100~\mathrm{g/(cm^2\,s)}.
\end{equation*}
Thus, apart from the different inlet velocities specified separately below,
the first two examples use the same physical and numerical parameter sets.
The following three examples below involve different geometries,
loading conditions, and simulation intervals. Specifically, the first two examples are
intended to demonstrate the applicability of the two partitioned
schemes rather than to provide a direct comparison between their
accuracies. The third example extends the example in Subsection \ref{subsec:pressure_wave_benchmark} to three dimensions, thereby further demonstrating the applicability and effectiveness of the proposed algorithm.

\subsection{Flow through a deformable poroelastic chamber}
\label{subsubsec:nonlinear_chamber}

The first example considers flow through a central chamber connected
to narrower lower and upper passages. Two deformable poroelastic
components form the lateral boundaries of the chamber, as illustrated
in Figure~\ref{fig:nonlinear_chamber_geometry}. The configuration
contains abrupt changes in cross-section and is designed to generate
pronounced interaction between the nonlinear fluid motion and the
deformable porous structures.

\paragraph{Boundary conditions.}
A portion of the upper boundary of the fluid domain is prescribed as the fluid inlet,
$\Gamma_f^{\mathrm{in}}
=\{(x,y):-2<x<2,\ y=3.3\}$,
where the following velocity condition is imposed:
\begin{equation*}
\mathbf{u}_f
=
\begin{cases}
\bigl(0,-2\sin(\frac{1}{2}\pi t)\bigr)^T,
& 0\leq t<1,\\
(0,-2)^T,
& t\geq1.
\end{cases}
\end{equation*}
The lower boundary of the fluid domain is taken as the outlet
$\Gamma_f^{\mathrm{out}}$, where a stress-free boundary condition is
prescribed. No-slip conditions are imposed on the left and right
boundaries of the fluid domain. On the exterior left and right boundaries of
the poroelastic components, homogeneous Dirichlet conditions are
prescribed for the displacement, while a no-penetration condition is
imposed on the Darcy velocity. On the remaining fluid--poroelastic
interfaces, the coupling conditions introduced in
Section~\ref{sec:fully_coupled_model} are enforced.

\begin{figure}[!htbp]
    \centering
    \begin{tikzpicture}
        \node[anchor=south west, inner sep=0pt] (A) at (0,0) {
            \begin{subfigure}[b]{6.15cm}
                \centering
                \includegraphics[
                    height=6.3cm,
                    trim=500 0 500 0,
                    clip
                ]{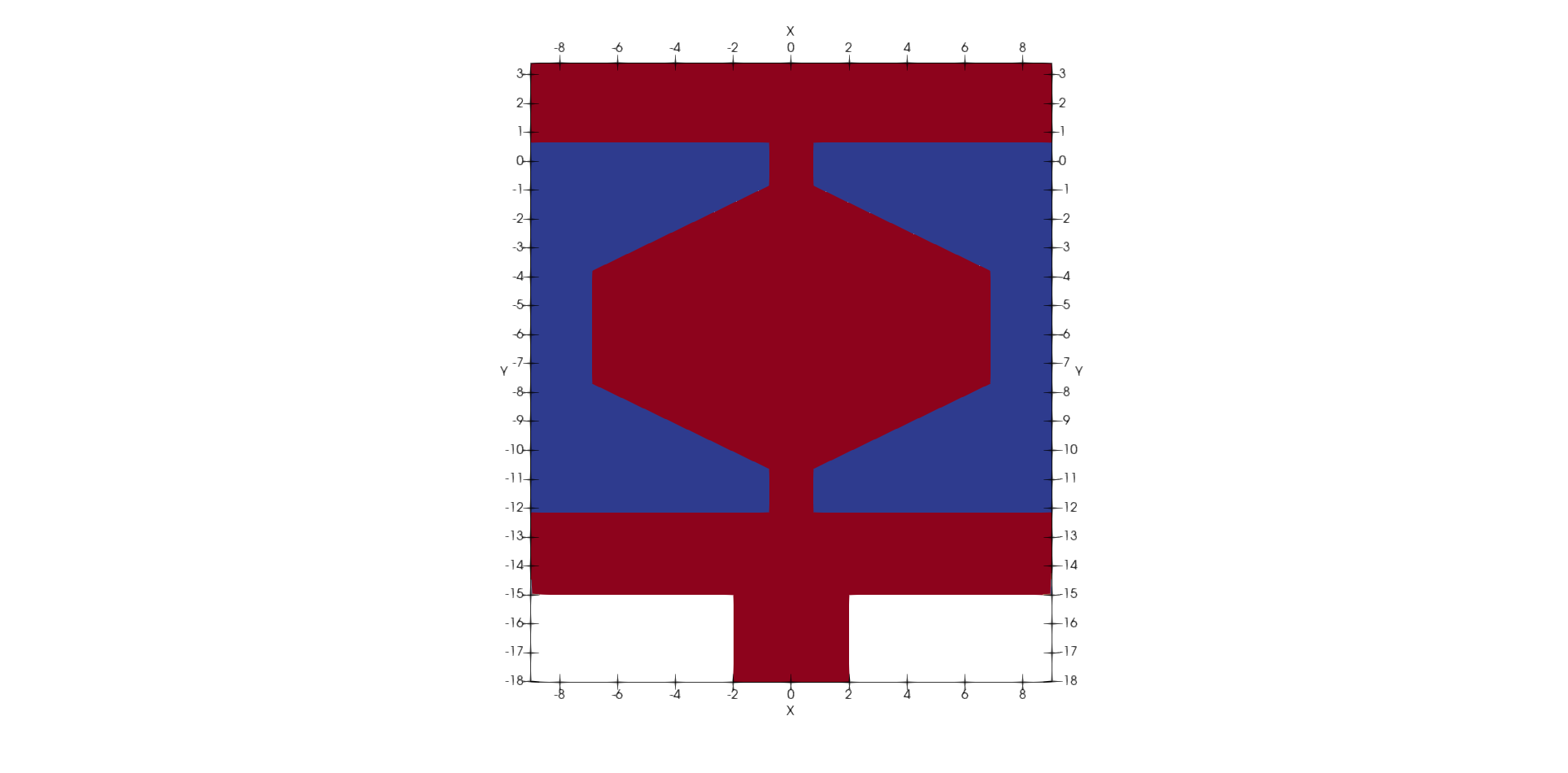}
                \caption{Computational domain.}
                \label{fig:domain_global}
            \end{subfigure}
        };
        \node[anchor=south west, inner sep=0pt] (B) at (6.75cm,0) {
            \begin{subfigure}[b]{6.15cm}
                \centering
                \includegraphics[
                    height=3.4cm,
                    trim=300 0 300 200,
                    clip
                ]{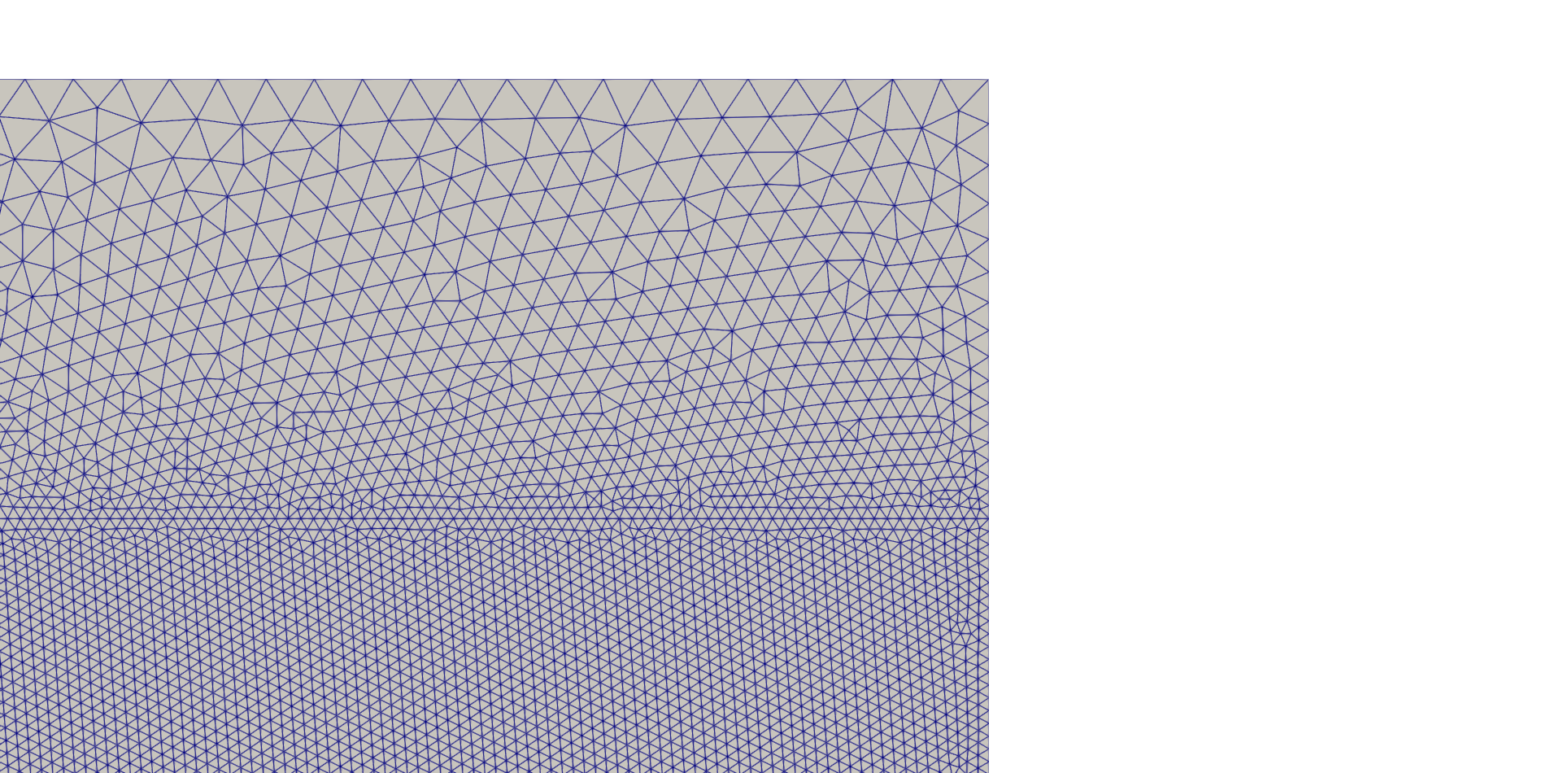}
                \caption{Local view of the refined mesh.}
                \label{fig:domain_zoom}
            \end{subfigure}
        };
        \draw[red, dashed, thick]
            ($(A.north west)+(4.50cm,-0.85cm)$)
            rectangle
            ($(A.north west)+(5.10cm,-1.45cm)$);
        \draw[red, thick, ->]
            ($(A.north west)+(5.10cm,-1.15cm)$)
            --
            ($(B.north west)+(0.12cm,-1.15cm)$);
    \end{tikzpicture}

    \caption{
    Initial computational geometry and mesh configuration for the
    deformable poroelastic chamber. The red region represents the fluid
    domain and the blue regions represent the poroelastic media.
    Panel~(b) shows a magnified view of the locally refined mesh.
    The fluid and poroelastic meshes contain $49\,060$ and $46\,912$
    cells, respectively.
    }
    \label{fig:nonlinear_chamber_geometry}

\end{figure}

The first-order Robin partitioned scheme is used for this computation.
Figure~\ref{fig:nonlinear_chamber_fields} shows the poroelastic
displacement magnitude, fluid pressure, and fluid velocity magnitude at
$t=1$, $25$, and $5~\mathrm{s}$. The white curves in the velocity panels
represent instantaneous streamlines.

The displacement is concentrated near the narrower connections between
the poroelastic components and the surrounding fixed boundaries. As
the flow evolves, the locations of the largest deformation change
smoothly while the overall response remains compatible with the
geometry of the chamber. The pressure field undergoes a substantial
redistribution between the upper and lower passages and the central
region. The velocity plots show jets through the narrow openings and
recirculating structures inside the enlarged chamber. The changing
streamline pattern demonstrates the influence of the moving
poroelastic boundaries on the nonlinear fluid motion.

\begin{figure}[!p]
\centering

\setlength{\tabcolsep}{0pt}

\newcommand{\etafig}[1]{%
  \includegraphics[
    width=0.325\textwidth,
    trim=470bp 170bp 450bp 140bp,
    clip
  ]{#1.png}%
}

\newcommand{\pfig}[1]{%
  \includegraphics[
    width=0.325\textwidth,
    trim=520bp 10bp 410bp 30bp,
    clip
  ]{#1.png}%
}

\newcommand{\ufig}[1]{%
  \includegraphics[
    width=0.325\textwidth,
    trim=460bp 15bp 530bp 15bp,
    clip
  ]{#1.png}%
}

\begin{tabular}{@{}ccc@{}}
$t=1$
&
$t=2.5$
&
$t=5$
\\[0.3em]

\etafig{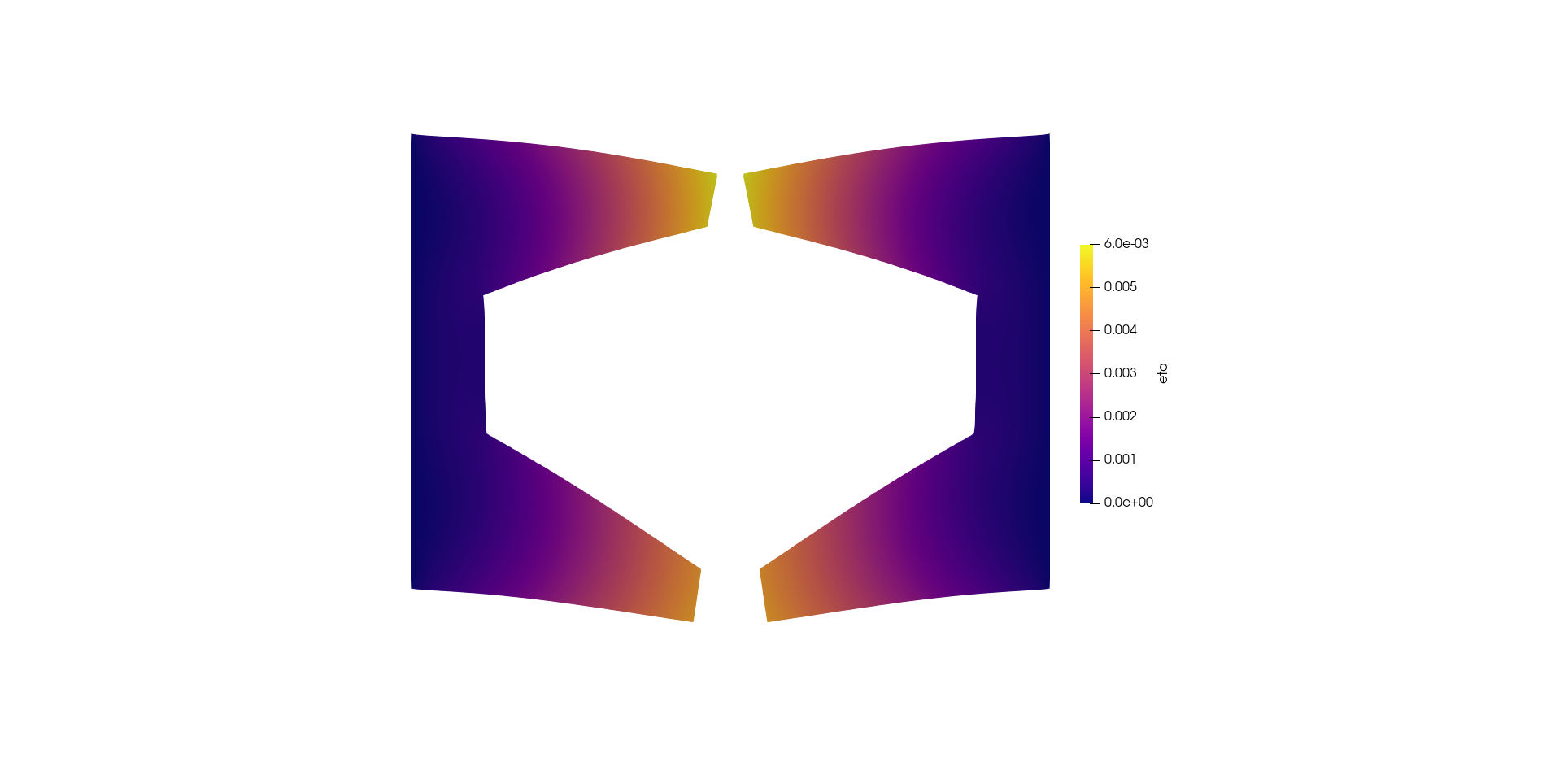}
&
\etafig{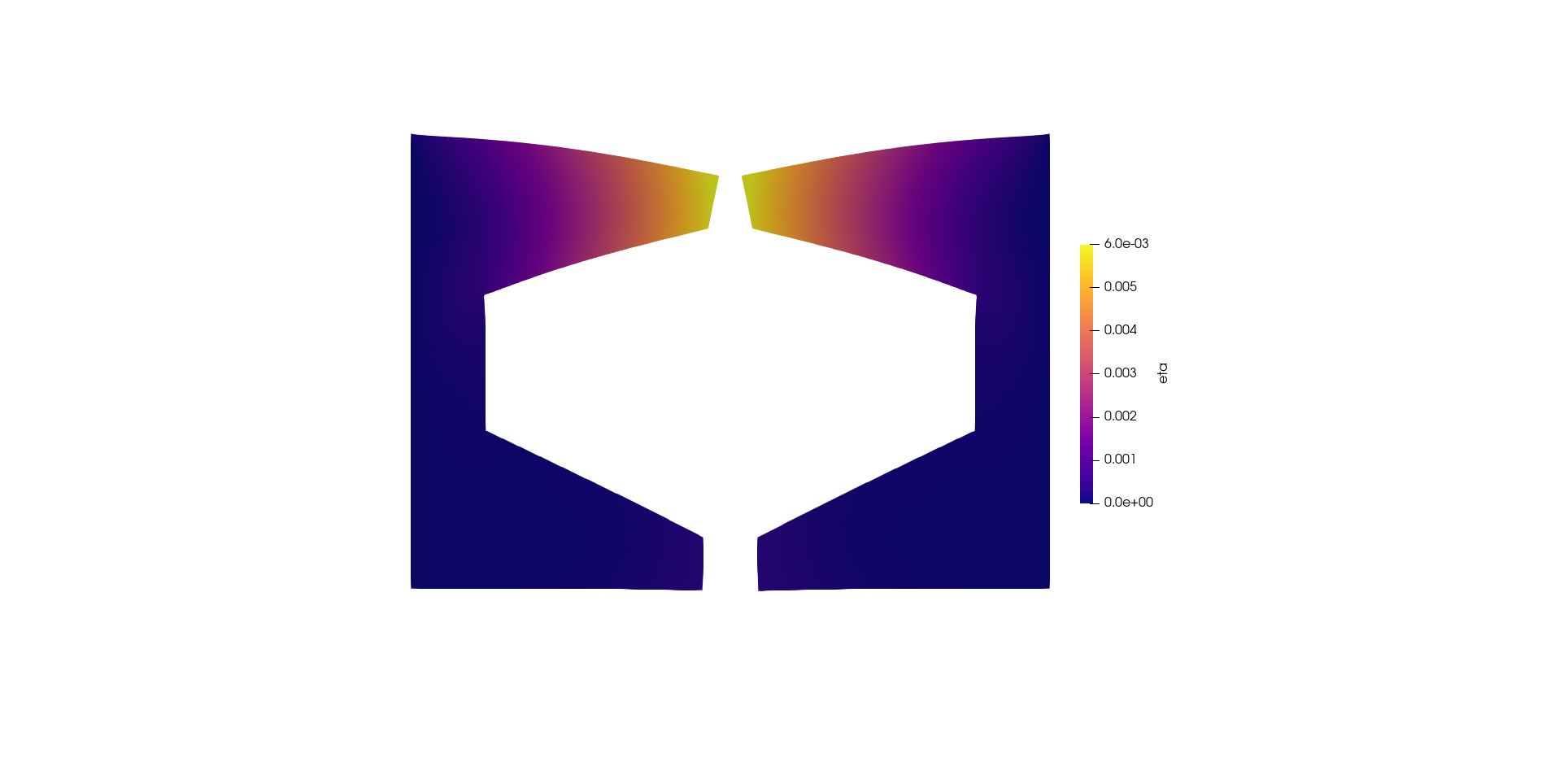}
&
\etafig{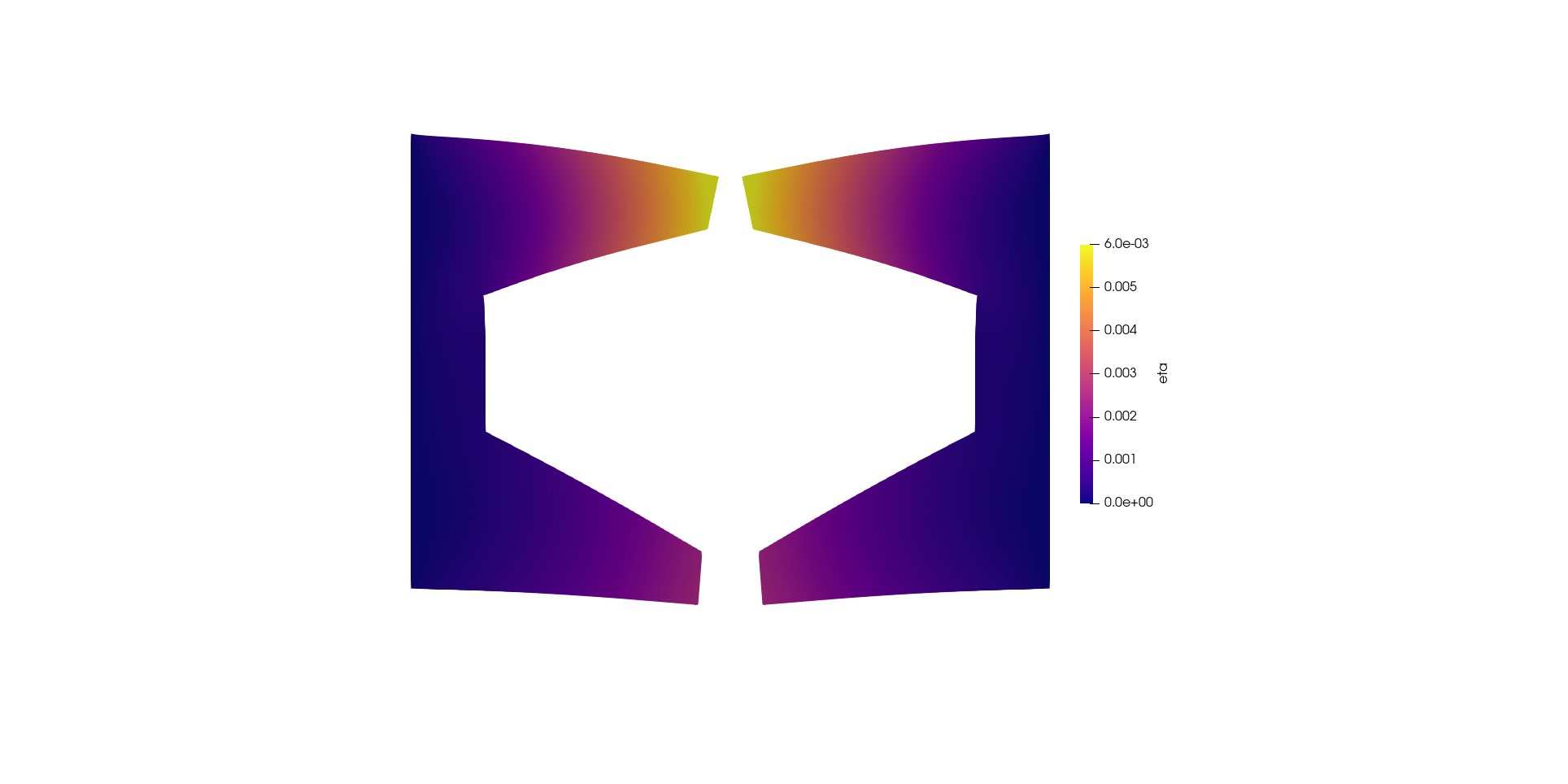}
\\[-0.6em]

\pfig{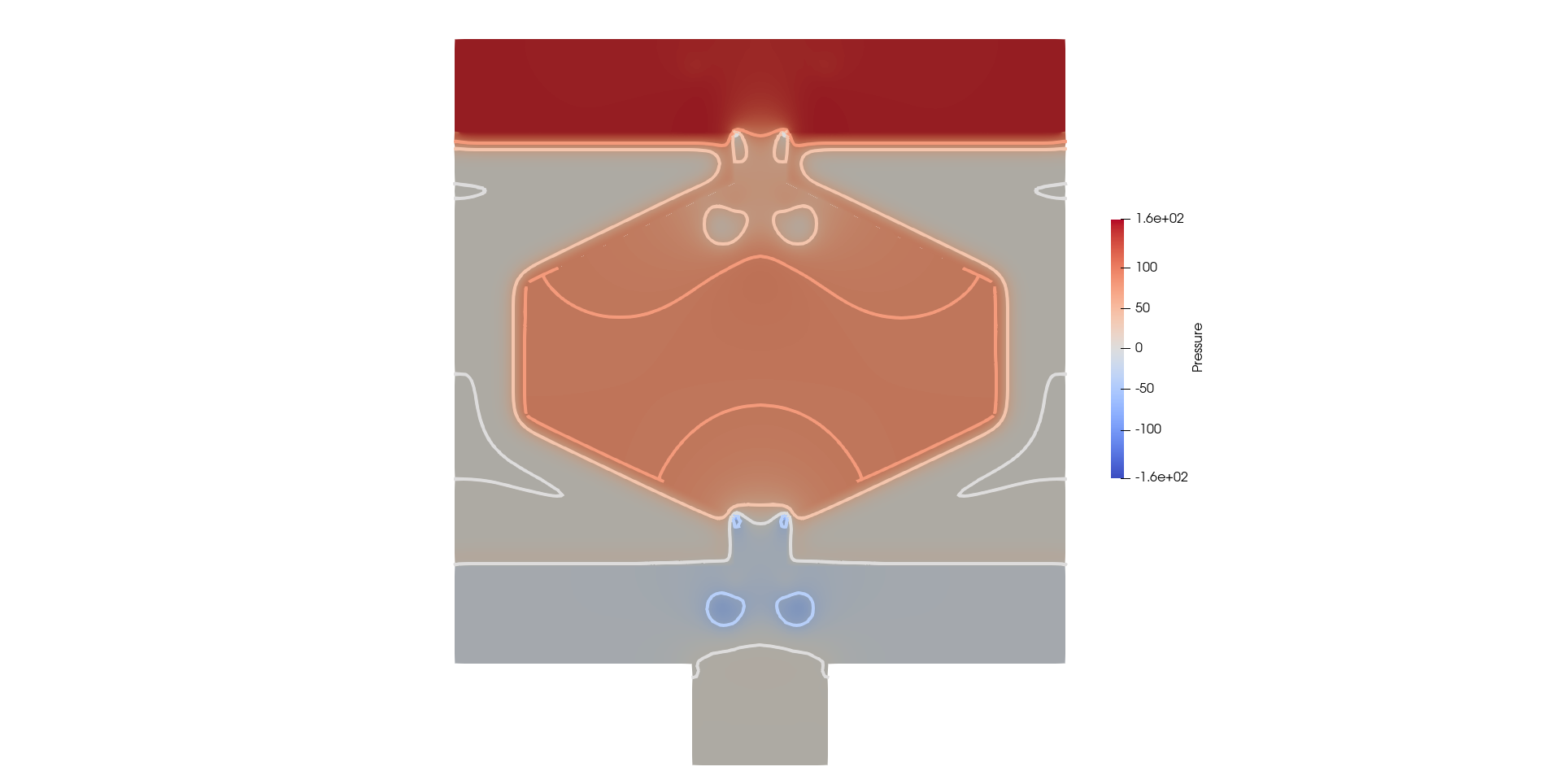}
&
\pfig{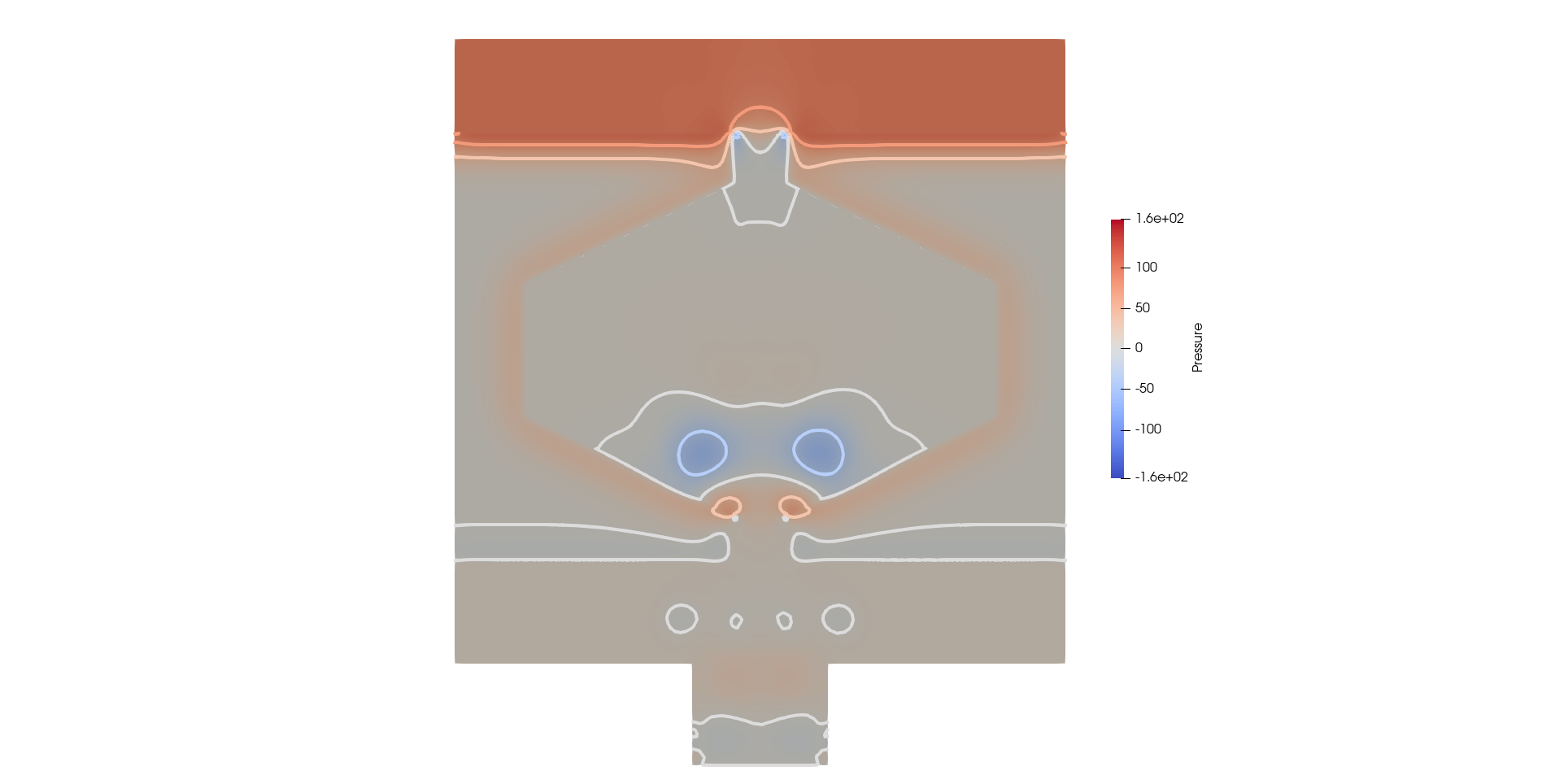}
&
\pfig{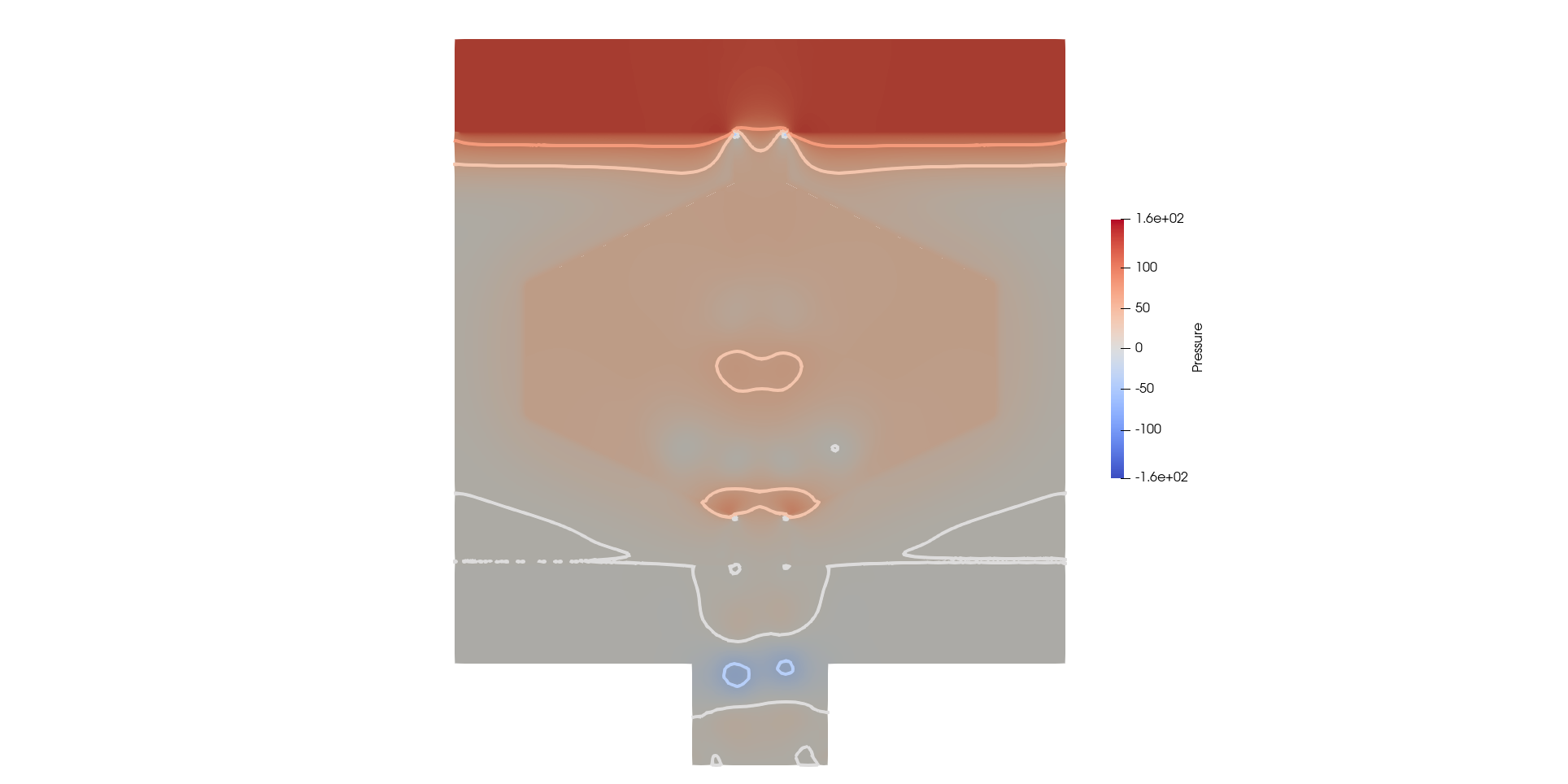}
\\[-0.6em]

\ufig{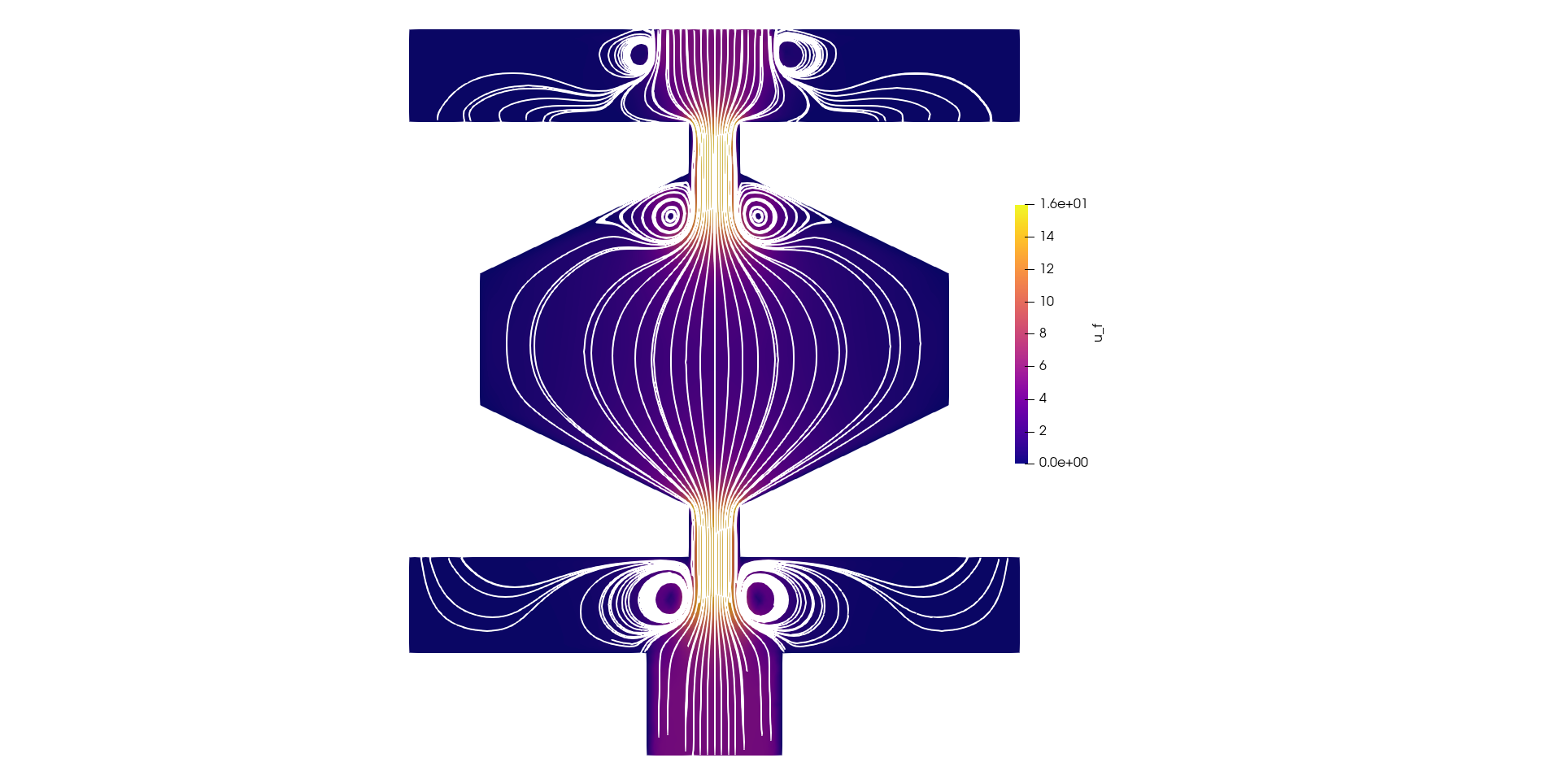}
&
\ufig{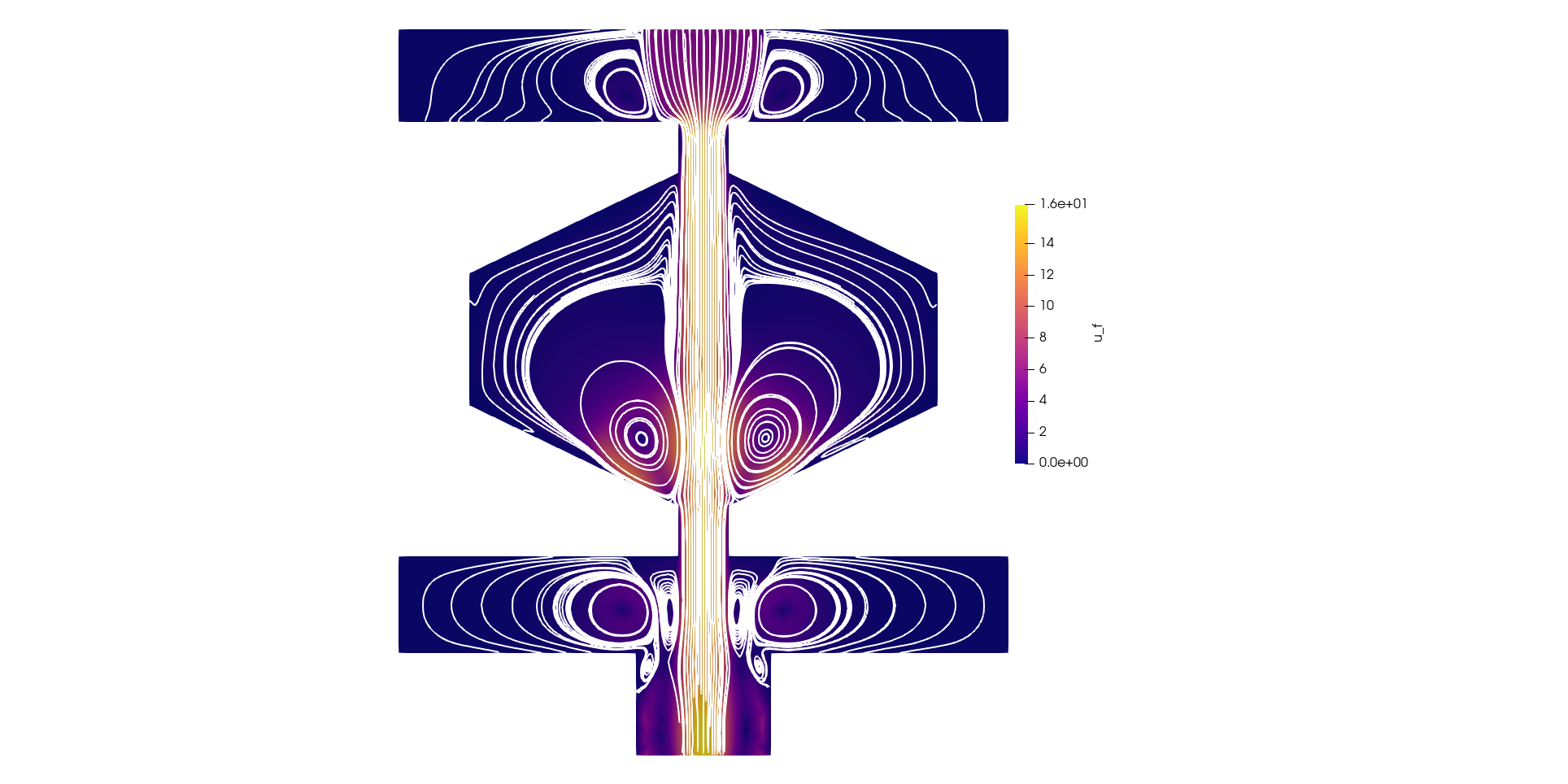}
&
\ufig{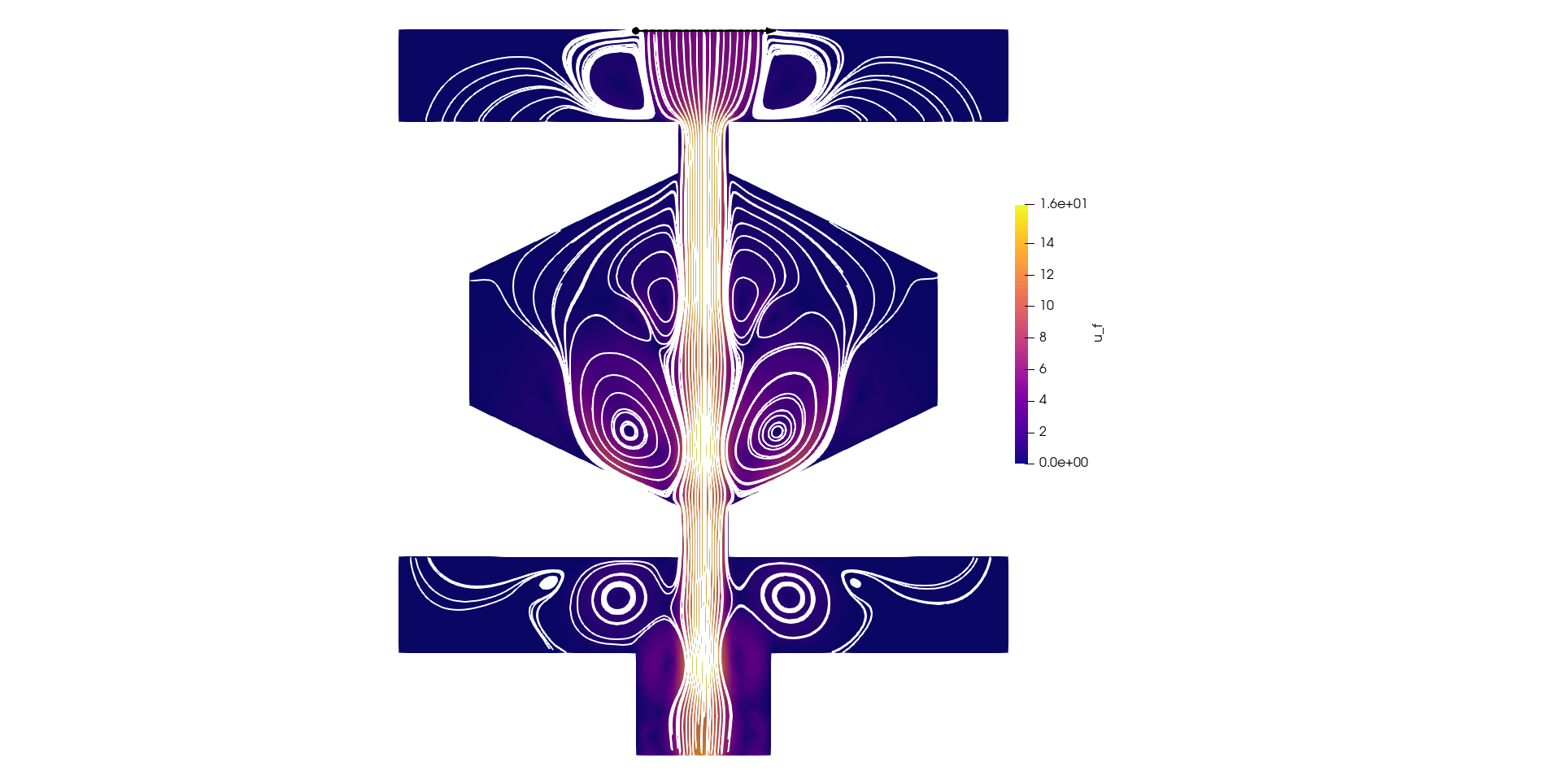}
\end{tabular}

\caption{Evolution of the coupled solution in the deformable chamber
computed with the first-order Robin partitioned scheme. From top to
bottom, the rows show the poroelastic displacement magnitude
$|\boldsymbol{\eta}_p|$, the fluid pressure $p_f$, and the fluid
velocity magnitude $|\mathbf{u}_f|$ together with instantaneous
streamlines, respectively. For visualization, the deformation in the
displacement panels is magnified by a factor of $200$.}
\label{fig:nonlinear_chamber_fields}
\end{figure}

\FloatBarrier

\subsection{Flow through a long channel with a deformable poroelastic segment}
\label{subsubsec:nonlinear_long_channel}

The second example considers a long fluid channel containing two
localized poroelastic segments near the inlet, one on each side of the
fluid region. Its initial configuration is shown in
Figure~\ref{fig:nonlinear_long_channel_geometry}. Deformation of
the two poroelastic components produces a moving constriction near the
upstream part of the channel, whereas the extended downstream region
allows the pressure disturbance and the nonlinear flow structures to
develop over a comparatively long distance.

\begin{figure}[!htbp]
\centering
\includegraphics[width=0.96\textwidth]
{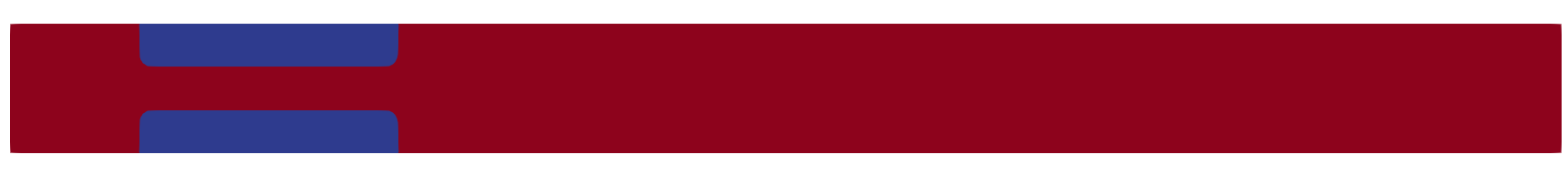}
\caption{Initial computational geometry for the long channel with
localized deformable poroelastic segments. The red region represents
the fluid domain and the blue regions represent the poroelastic media.
The fluid and poroelastic meshes contain $64\,670$ and $31\,220$
cells, respectively.}
\label{fig:nonlinear_long_channel_geometry}
\end{figure}

\FloatBarrier

\paragraph{Boundary conditions.}
Similar to the boundary setting in the previous example, the left
boundary of the fluid domain is prescribed as the inlet,
$\Gamma_f^{\mathrm{in}}
=\{(x,y):x=0,\ 0<y<1\}$,
where the following velocity condition is imposed:
\begin{equation*}
\mathbf{u}_f
=
\begin{cases}
\bigl(20\sin(\frac{1}{2}\pi t)y(1 - y),0\bigr)^T,
& 0\leq t<1,\\
(20y(1 - y),0)^T,
& t\geq1.
\end{cases}
\end{equation*}
The right boundary of the fluid domain is taken as the outlet,
$\Gamma_f^{\mathrm{out}}
=\{(x,y):x=12,\ 0<y<1\}$,
where a stress-free boundary condition is prescribed.
The upper poroelastic component has a thickness of $0.3$, and the
lower component is arranged symmetrically with respect to the upper one. 
The upper and lower boundaries of the fluid and poroelastic
domains are treated in the same manner as in the previous example.

The second-order Robin partitioned scheme is employed in this example,
and the first BDF2 step is initialized by the first-order partitioned
method. The numerical solution is examined at the representative times
$t=1$, $2.5$, and $5~\mathrm{s}$.

Figure~\ref{fig:nonlinear_long_channel_fluid_fields} presents the
evolution of the fluid pressure and velocity fields. The left column
shows the fluid pressure $p_f$, while the right column shows the
velocity magnitude $|\mathbf{u}_f|$ together with instantaneous
streamlines. From top to bottom, the rows correspond to $t=1$, $2.5$,
and $5~\mathrm{s}$.

At the initial observation time, the pressure variation and the main
fluid acceleration remain concentrated near the deformable
constriction. As the simulation proceeds, the pressure disturbance
propagates into the downstream part of the channel and develops a more
spatially nonuniform distribution. The velocity field initially
exhibits a relatively organized high-speed region emerging from the
constriction. At the later observation times, elongated recirculating
regions and increasingly complex streamline patterns appear farther
downstream. These results illustrate the combined effects of nonlinear
convection, the localized moving boundary, and the long propagation
distance.

\begin{figure}[htbp]
\centering
\setlength{\tabcolsep}{0pt}

\newcommand{\longpfig}[1]{%
  \includegraphics[
    width=0.98\linewidth,
    trim=100bp 395bp 210bp 260bp,
    clip
  ]{#1}%
}

\newcommand{\longufig}[1]{%
  \includegraphics[
    width=0.98\linewidth,
    trim=100bp 395bp 210bp 260bp,
    clip
  ]{#1}%
}

\begin{tabular}{
@{}
>{\centering\arraybackslash}m{0.47\textwidth}
@{\hspace{0.015\textwidth}}
>{\centering\arraybackslash}m{0.47\textwidth}
@{}
}
\textbf{Fluid pressure $p_f$}
&
\textbf{Velocity magnitude $|\mathbf{u}_f|$}
\\[0.4em]

\longpfig{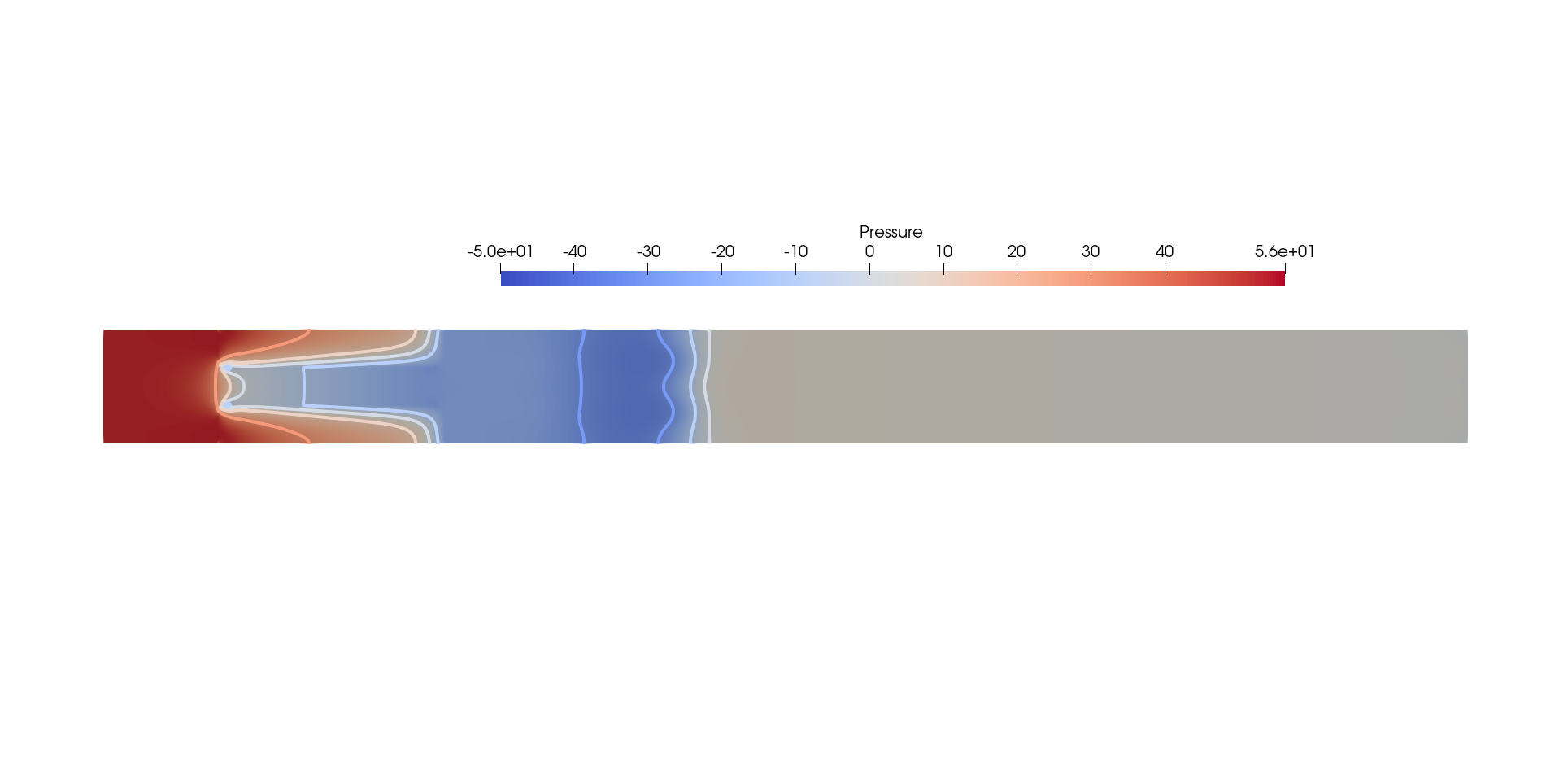}
&
\longufig{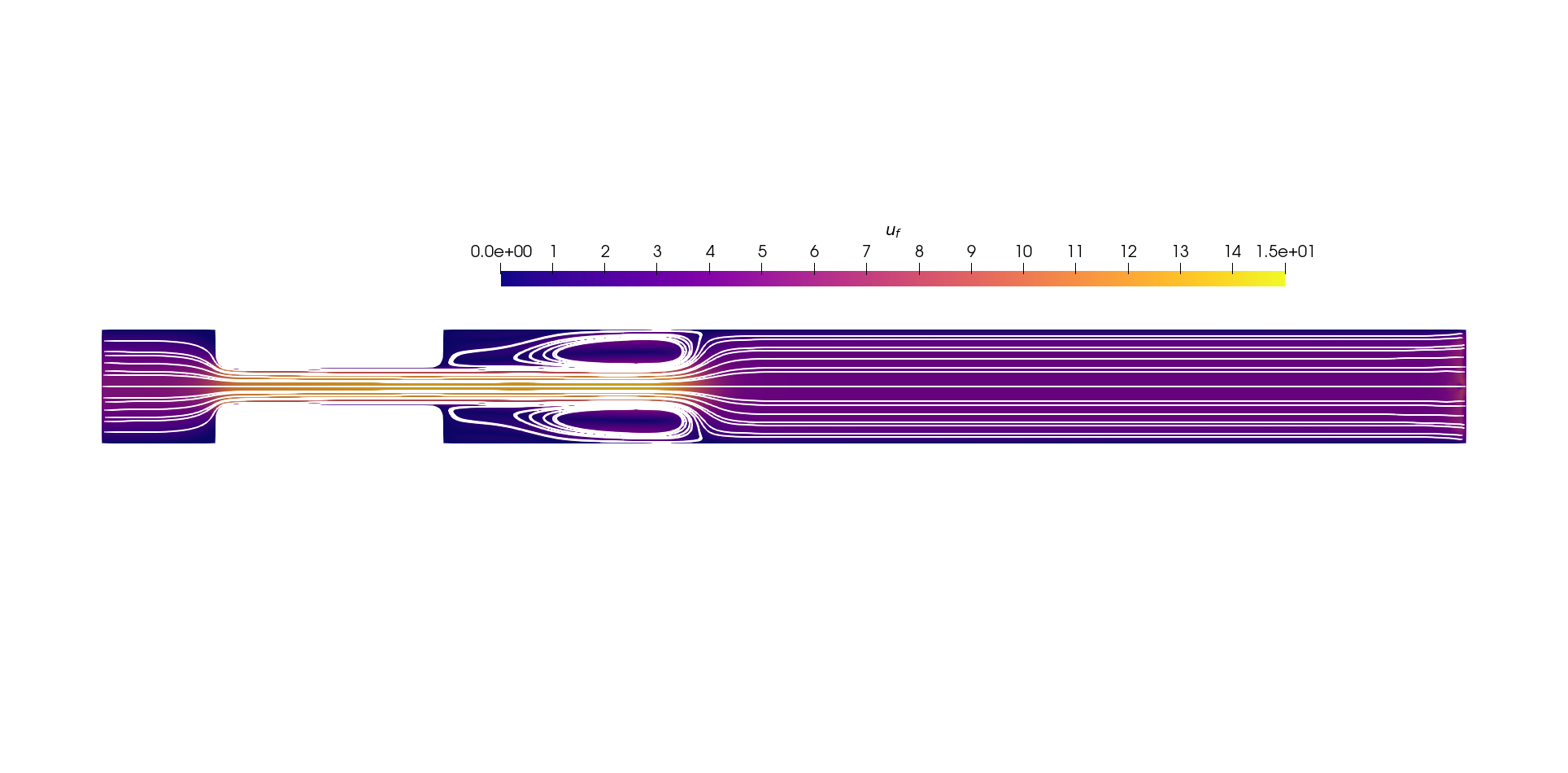}
\\[0.15em]

\longpfig{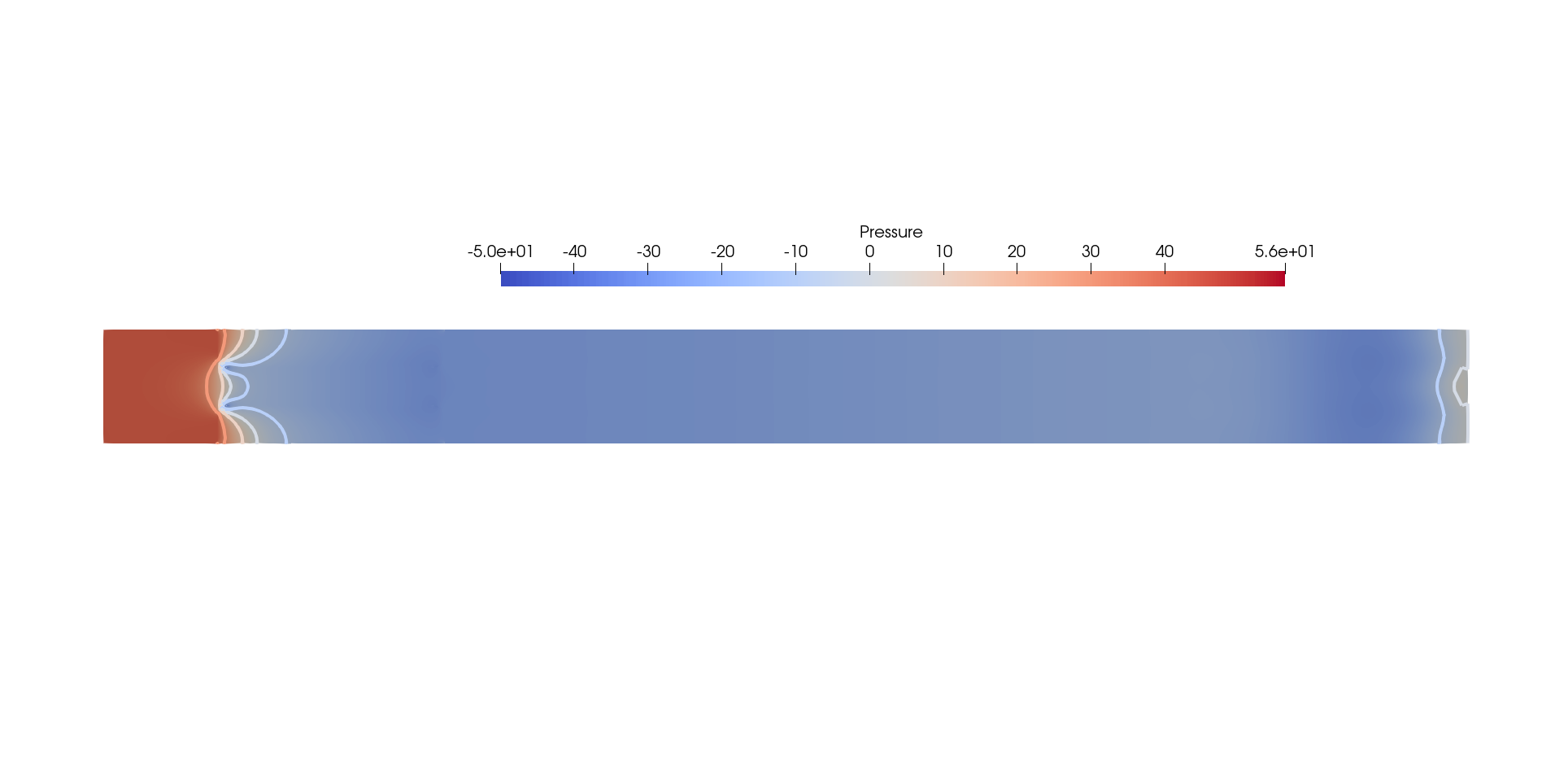}
&
\longufig{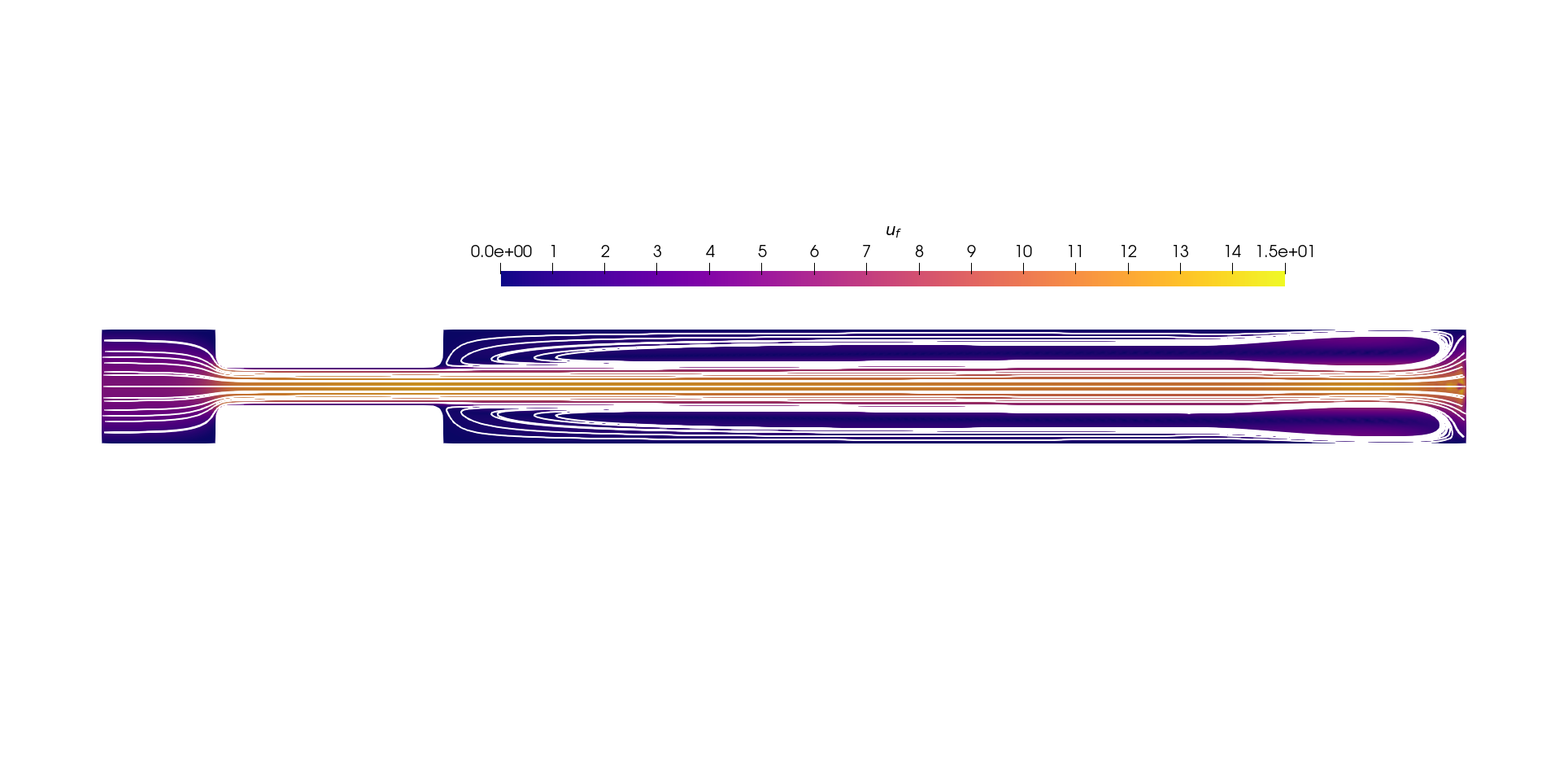}
\\[0.15em]

\longpfig{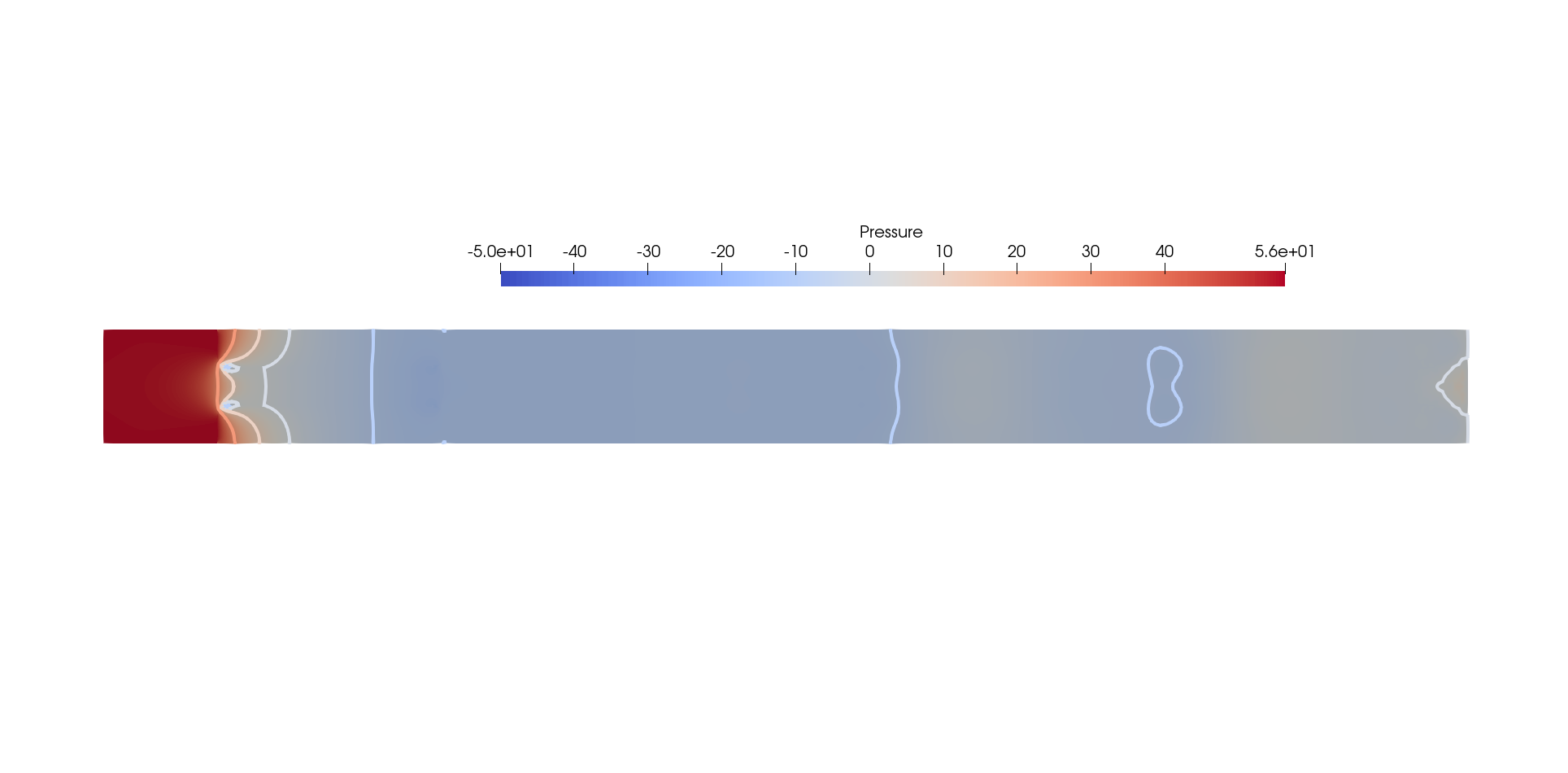}
&
\longufig{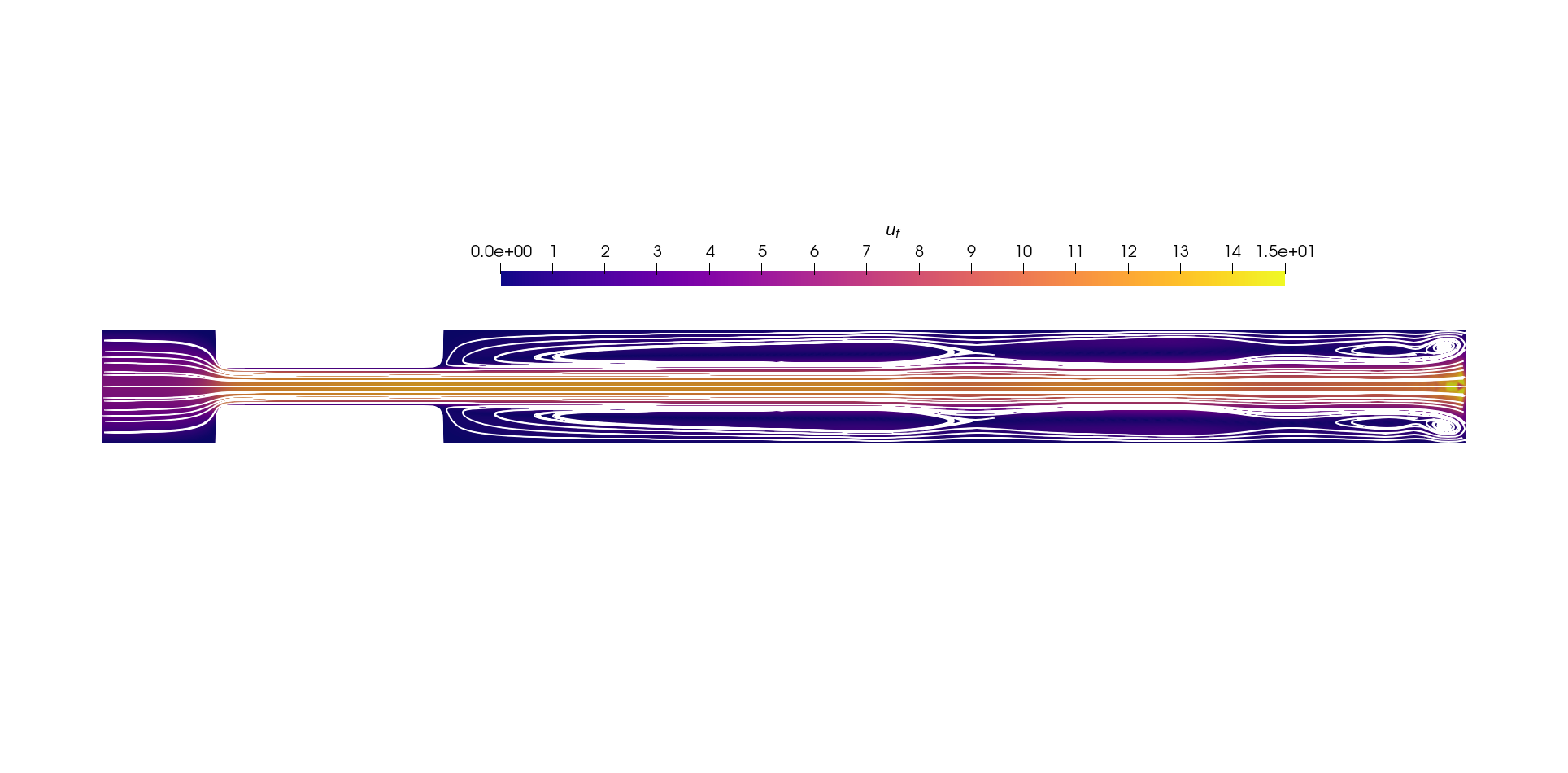}
\end{tabular}

\caption{Evolution of the fluid pressure and velocity fields in the
long deformable channel computed with the second-order Robin
partitioned scheme. The left and right columns show the fluid pressure
$p_f$ and the velocity magnitude $|\mathbf{u}_f|$ together with
instantaneous streamlines, respectively. From top to bottom, the rows
correspond to $t=1$, $2.5$, and $5$.}
\label{fig:nonlinear_long_channel_fluid_fields}
\end{figure}

\FloatBarrier

The spatial pattern of the poroelastic displacement remains
qualitatively similar at the selected observation times. Therefore,
rather than repeating three nearly identical displacement panels, a
representative displacement field at $t=5~\mathrm{s}$ is displayed in
Figure~\ref{fig:nonlinear_long_channel_displacement}. The deformation
is localized in the two poroelastic components near the inlet, with
the largest displacement occurring close to the moving constriction.

\begin{figure}[!htbp]
\centering
\includegraphics[
  width=0.82\textwidth,
  trim=300bp 200bp 200bp 200bp,
  clip
]{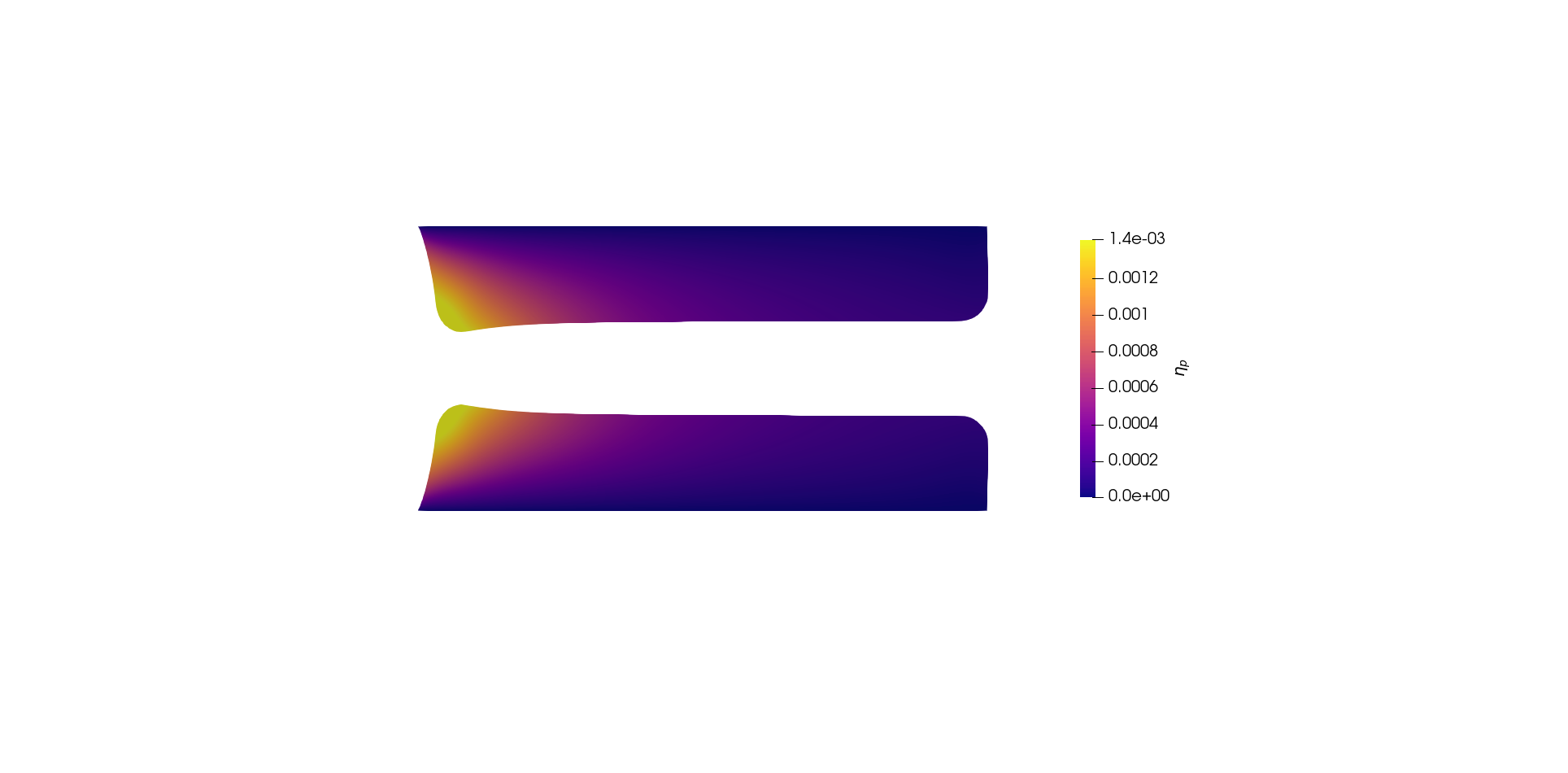}
\caption{Poroelastic displacement magnitude
$|\boldsymbol{\eta}_p|$ at $t=2.5$ in the long deformable channel,
computed with the second-order Robin partitioned scheme. For
visualization, the deformation is magnified by a factor of $1000$.}
\label{fig:nonlinear_long_channel_displacement}
\end{figure}

\FloatBarrier

In a word, the computations remain stable over the full simulation interval and
capture the coupled evolution of the moving poroelastic segment, the
fluid pressure, and the nonlinear velocity field. The more complicated
downstream structures observed at the latest time may also be
influenced by the outflow boundary treatment and the finite length of
the computational channel. A systematic investigation based on
additional mesh, time-step, and outlet-location studies is left for
future work.

\FloatBarrier

\subsection{Three-dimensional pressure-wave propagation in a deformable\\ poroelastic tube}
\label{subsubsec:nonlinear_3d_pressure_wave}

As a final test, we consider the fully three-dimensional counterpart of the
pressure-wave benchmark introduced in
Subsection~\ref{subsec:pressure_wave_benchmark}. In contrast to the
two-dimensional half-channel configuration, the complete circumferential
directions of both the fluid channel and the poroelastic wall are resolved.
In the reference configuration, the fluid and poroelastic domains are
\begin{equation*}
\widehat{\Omega}_f
=
\left\{
(x,y,z)\in\mathbb{R}^3:
0<x<H,\quad
y^2+z^2<R^2
\right\},
\end{equation*}
and
\begin{equation*}
\widehat{\Omega}_p
=
\left\{
(x,y,z)\in\mathbb{R}^3:
0<x<H,\quad
R^2<y^2+z^2<(R+r_p)^2
\right\},
\end{equation*}
respectively, with $H=6~\mathrm{cm},~
R=0.5~\mathrm{cm},~
r_p=0.1~\mathrm{cm}$. The fluid--poroelastic interface is the cylindrical surface
$y^2+z^2=R^2$.

The boundary and initial conditions are the three-dimensional counterparts
of those used in Subsection~\ref{subsec:pressure_wave_benchmark}. In
particular, the same time-dependent pressure pulse is prescribed on the
fluid inlet at $x=0$, while a homogeneous traction condition is imposed at
the fluid outlet at $x=H$. The poroelastic wall is clamped on its two
annular end faces, and the Darcy velocity satisfies the impermeability
condition on the two end faces and the exterior cylindrical wall. The system
is initially at rest. Since the complete circumferential deformation of the
wall is resolved directly in the three-dimensional geometry, the auxiliary
linear spring term introduced in the two-dimensional benchmark to represent
circumferential recoil is not included here.

All remaining material parameters are taken from
Table~\ref{tab:pressure_wave_parameters}. The first-order Robin partitioned
scheme is employed with $
\Delta t=5\times10^{-5}~\mathrm{s},~L=500$. And the fluid and poroelastic domains are discretized by tetrahedral meshes
containing $110\,935$ and $110\,036$ cells, respectively. Numerical
results are displayed at four representative times,
$t=0.0035~\mathrm{s}$, $0.007~\mathrm{s}$,
$0.0105~\mathrm{s}$, and $0.014~\mathrm{s}$,
which illustrate successive stages of the pressure-wave propagation.

To visualize the three-dimensional rotation-dominated flow structures, we
use the $Q$-criterion
\begin{equation*}
Q
=
\frac{1}{2}
\left(
\|\boldsymbol{W}_f\|_F^2
-
\|\mathbf{S}_f\|_F^2
\right),
\end{equation*}
where
$
\mathbf{S}_f
=
\frac{1}{2}
\left(
\nabla\mathbf{u}_f+\nabla\mathbf{u}_f^T
\right),
\qquad
\boldsymbol{W}_f
=
\frac{1}{2}
\left(
\nabla\mathbf{u}_f-\nabla\mathbf{u}_f^T
\right).
$
Positive values of $Q$ indicate regions in which the local rotational
contribution dominates the strain-rate contribution. The isosurface
$Q=1000$ is used in the following visualizations.

Figure~\ref{fig:nonlinear_3d_velocity_displacement} presents the coupled
evolution of the fluid velocity, wall deformation, and
rotation-dominated structures. The deformed poroelastic wall is colored by
the displacement magnitude $|\boldsymbol{\eta}_p|$, while the fluid region
is rendered semi-transparently and colored by the velocity magnitude
$|\mathbf{u}_f|$. Isosurfaces at $Q=1000$ are superimposed to identify the
principal three-dimensional rotation-dominated regions. For visualization, poroelastic deformation is magnified by a factor of $5$, while the
color scale represents the actual displacement magnitude.

\begin{figure}[!htbp]
\centering
\setlength{\tabcolsep}{2pt}

\newcommand{\threeduf}[1]{%
  \includegraphics[
    width=0.455\textwidth,
    trim=330bp 70bp 210bp 160bp,
    clip
  ]{#1}%
}

\begin{tabular}{@{}cc@{}}
\threeduf{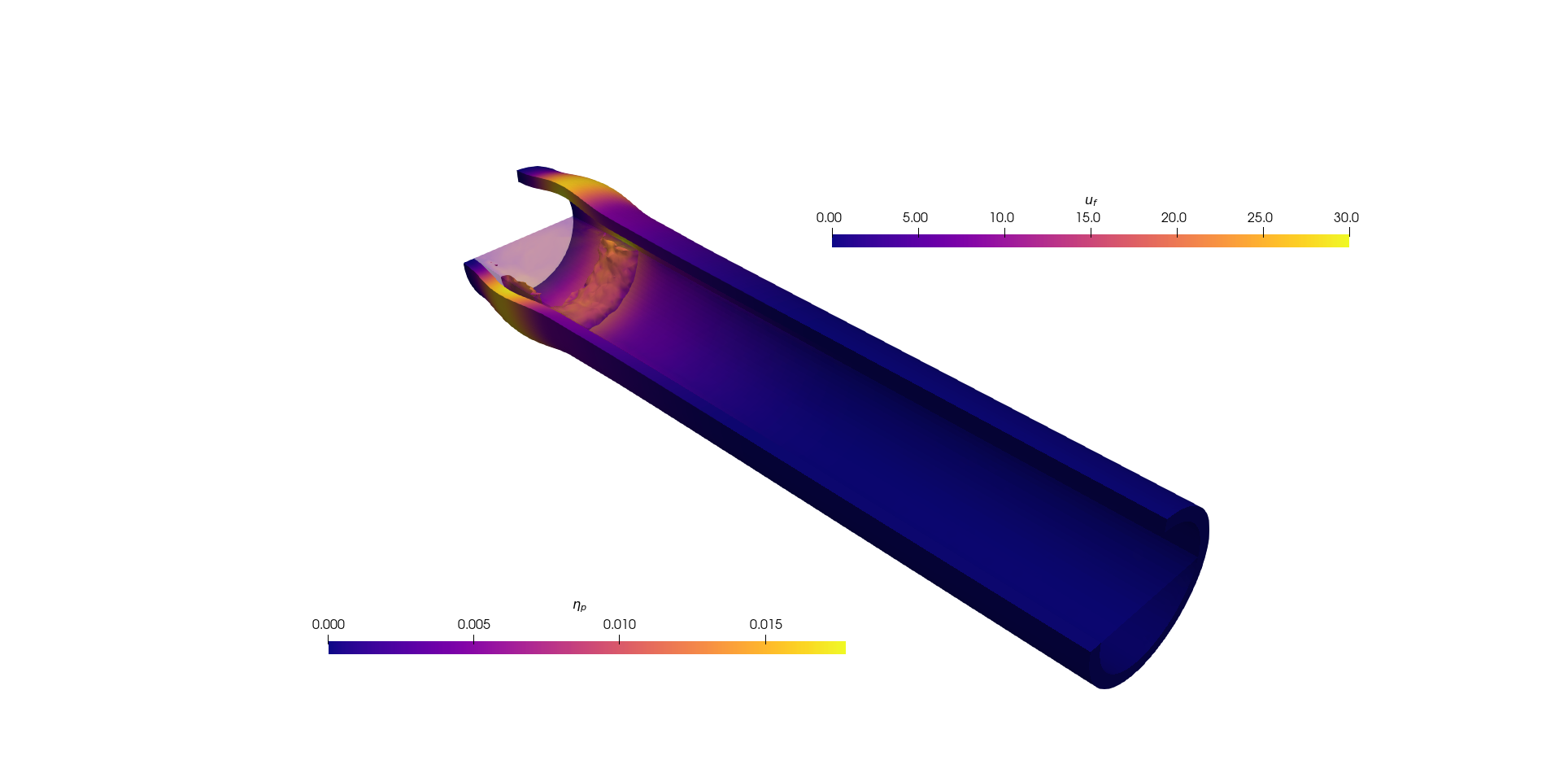}
&
\threeduf{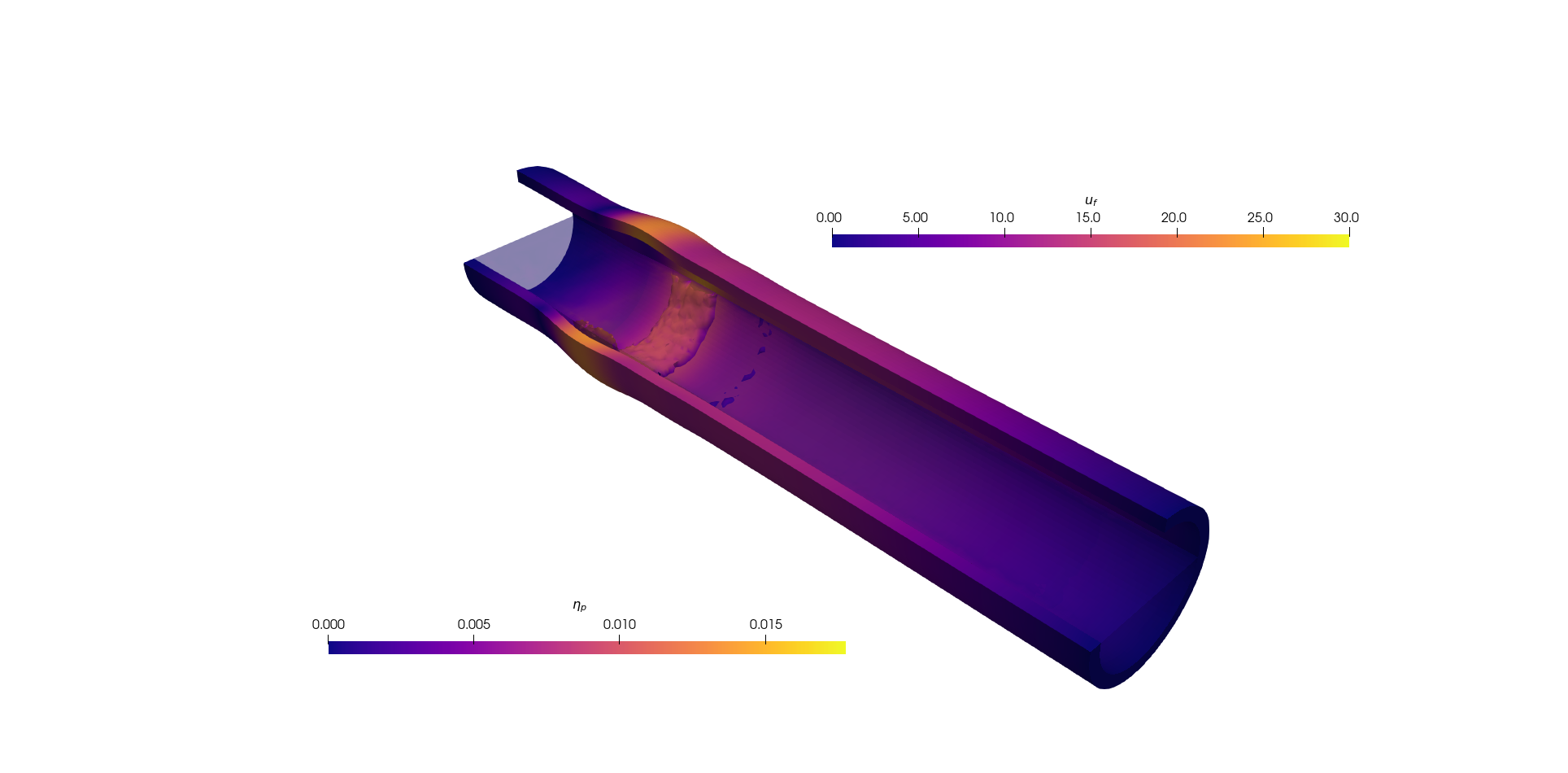}
\\[-0.35em]
{\small (a) $t=0.0035~\mathrm{s}$}
&
{\small (b) $t=0.007~\mathrm{s}$}
\\[0.45em]
\threeduf{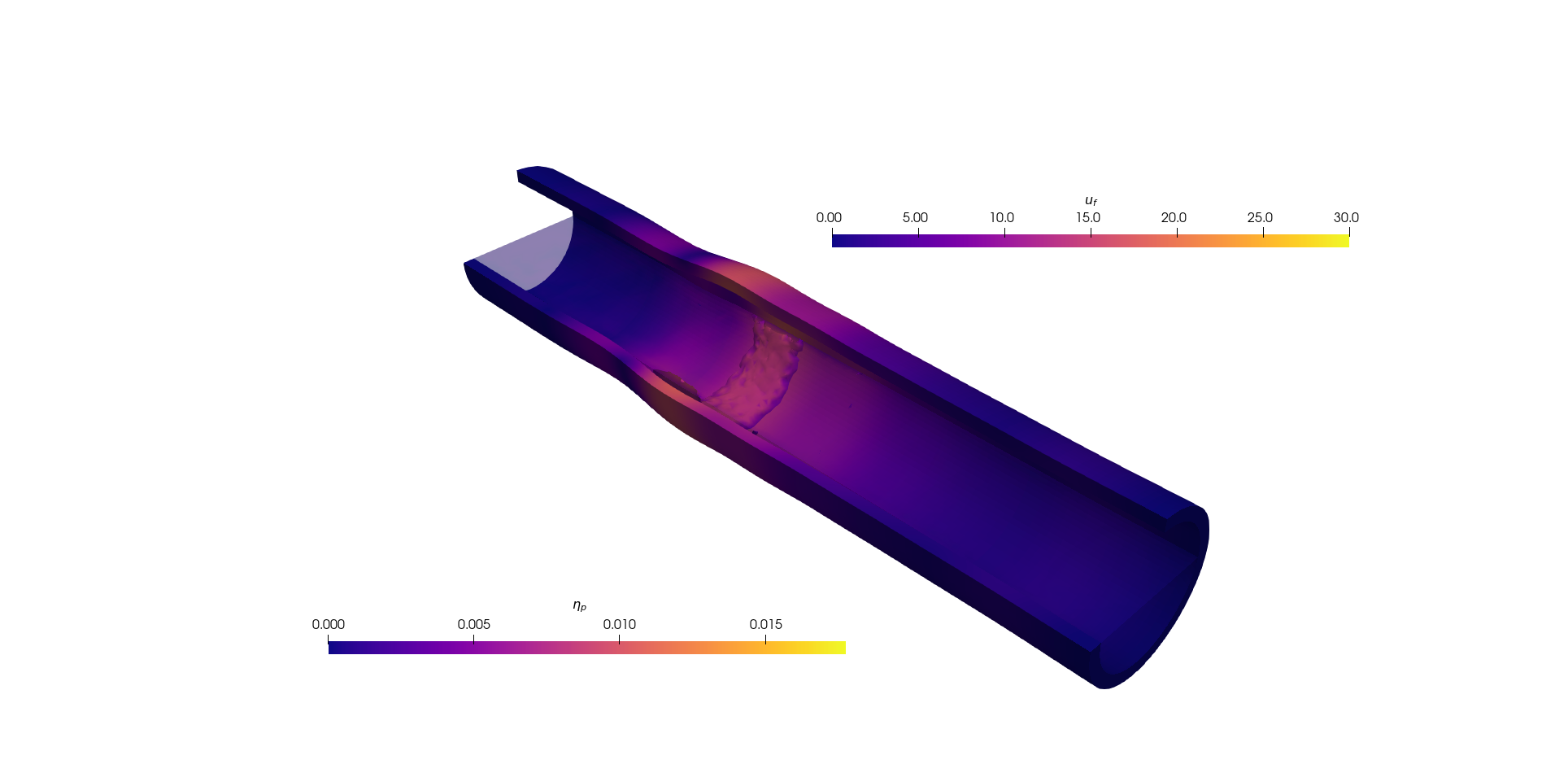}
&
\threeduf{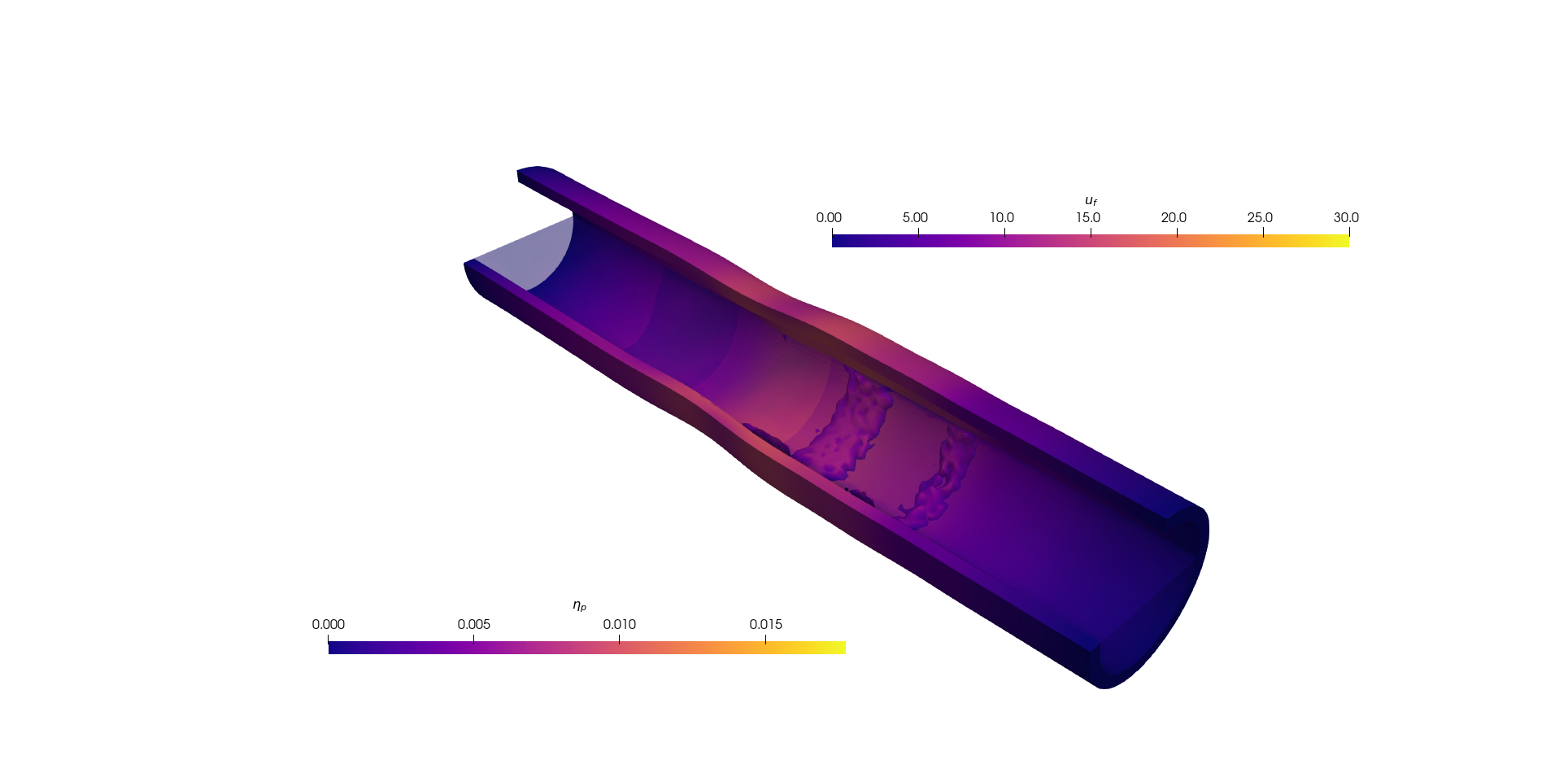}
\\[-0.35em]
{\small (c) $t=0.0105~\mathrm{s}$}
&
{\small (d) $t=0.014~\mathrm{s}$}
\end{tabular}

\caption{Three-dimensional evolution of the velocity magnitude
$|\mathbf{u}_f|$, poroelastic displacement magnitude
$|\boldsymbol{\eta}_p|$, and $Q$-criterion isosurfaces computed with the
first-order Robin partitioned scheme. The fluid region is rendered
semi-transparently, and the isosurface $Q=1000$ identifies
rotation-dominated regions. The displayed poroelastic deformation is
magnified by a factor of $5$, while the displacement color scale shows the
actual magnitude.}
\label{fig:nonlinear_3d_velocity_displacement}
\end{figure}

At the earliest observation time, the dominant fluid motion and wall
deformation are concentrated near the upstream region, and the corresponding
$Q$-criterion isosurfaces are localized close to the deformed
fluid--poroelastic interface. As the pressure disturbance propagates
downstream, the region of pronounced deformation moves along the tube and
the associated rotation-dominated structures develop farther downstream.
At the latest observation time, both expansion and contraction of the
poroelastic wall are captured, and two distinct $Q$-criterion isosurfaces
appear near the corresponding deformed regions. These results therefore
illustrate the close interaction between the moving poroelastic wall and the
local three-dimensional fluid dynamics.

The corresponding pressure evolution is shown in
Figure~\ref{fig:nonlinear_3d_pressure_fields}. The fluid region is colored by
the fluid pressure $p_f$, whereas the poroelastic wall is colored by the pore
pressure $p_p$. Identical color ranges are used at all displayed times to
facilitate direct comparison of the pressure-wave propagation.

\begin{figure}[!htbp]
\centering
\setlength{\tabcolsep}{2pt}

\newcommand{\threedp}[1]{%
  \includegraphics[
    width=0.455\textwidth,
    trim=320bp 65bp 190bp 160bp,
    clip
  ]{#1}%
}

\begin{tabular}{@{}cc@{}}
\threedp{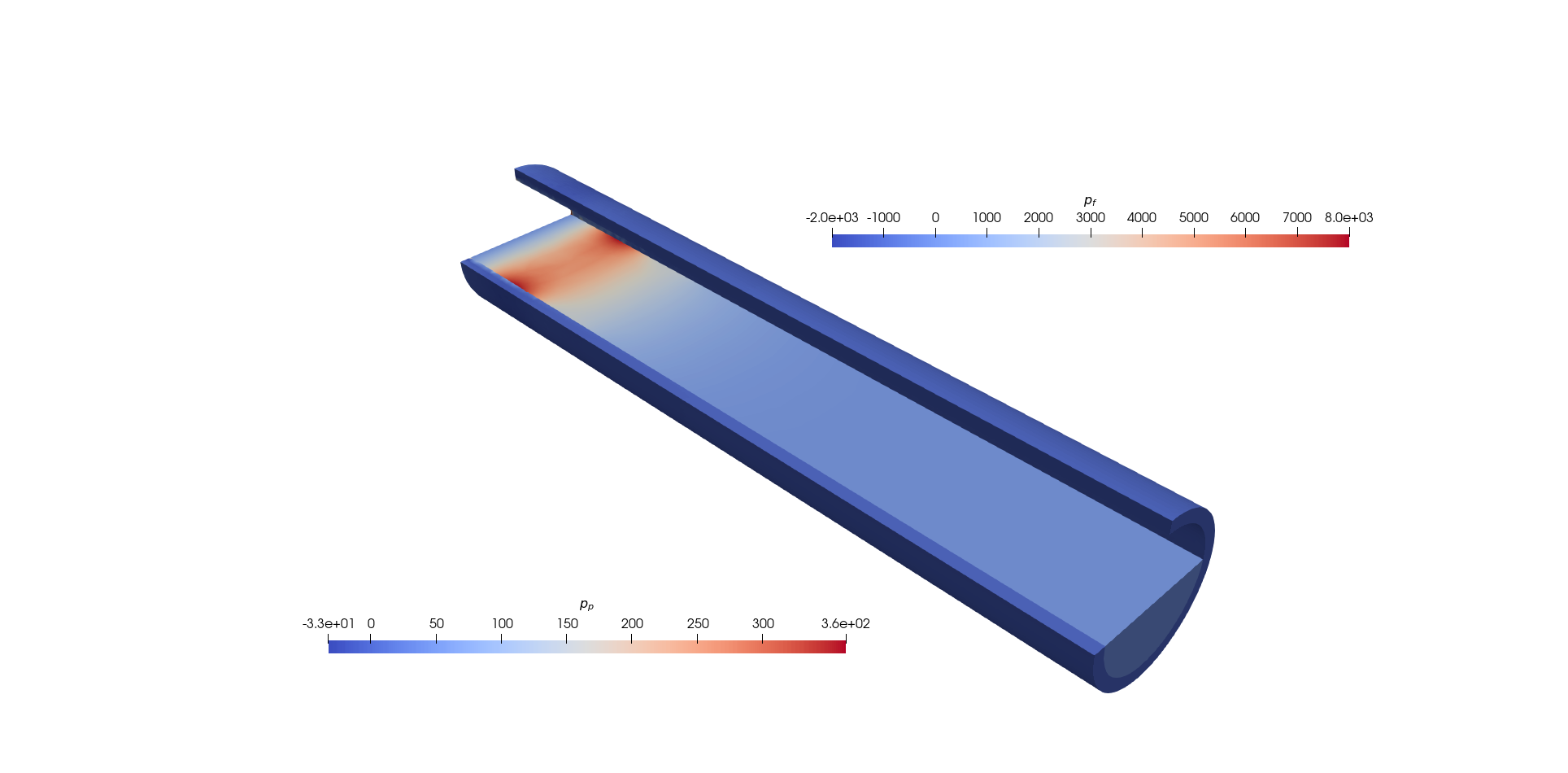}
&
\threedp{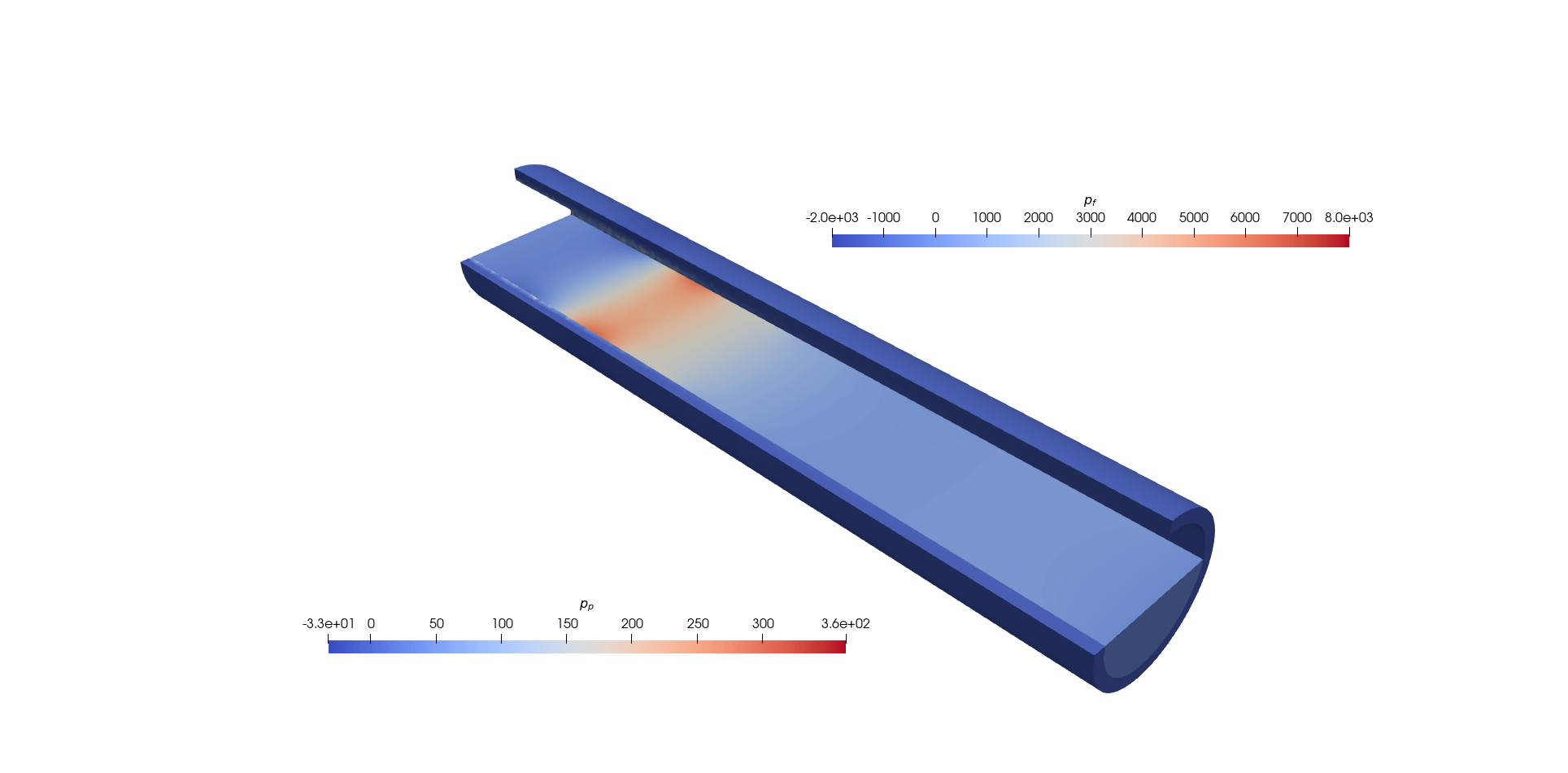}
\\[-0.35em]
{\small (a) $t=0.0035~\mathrm{s}$}
&
{\small (b) $t=0.007~\mathrm{s}$}
\\[0.45em]
\threedp{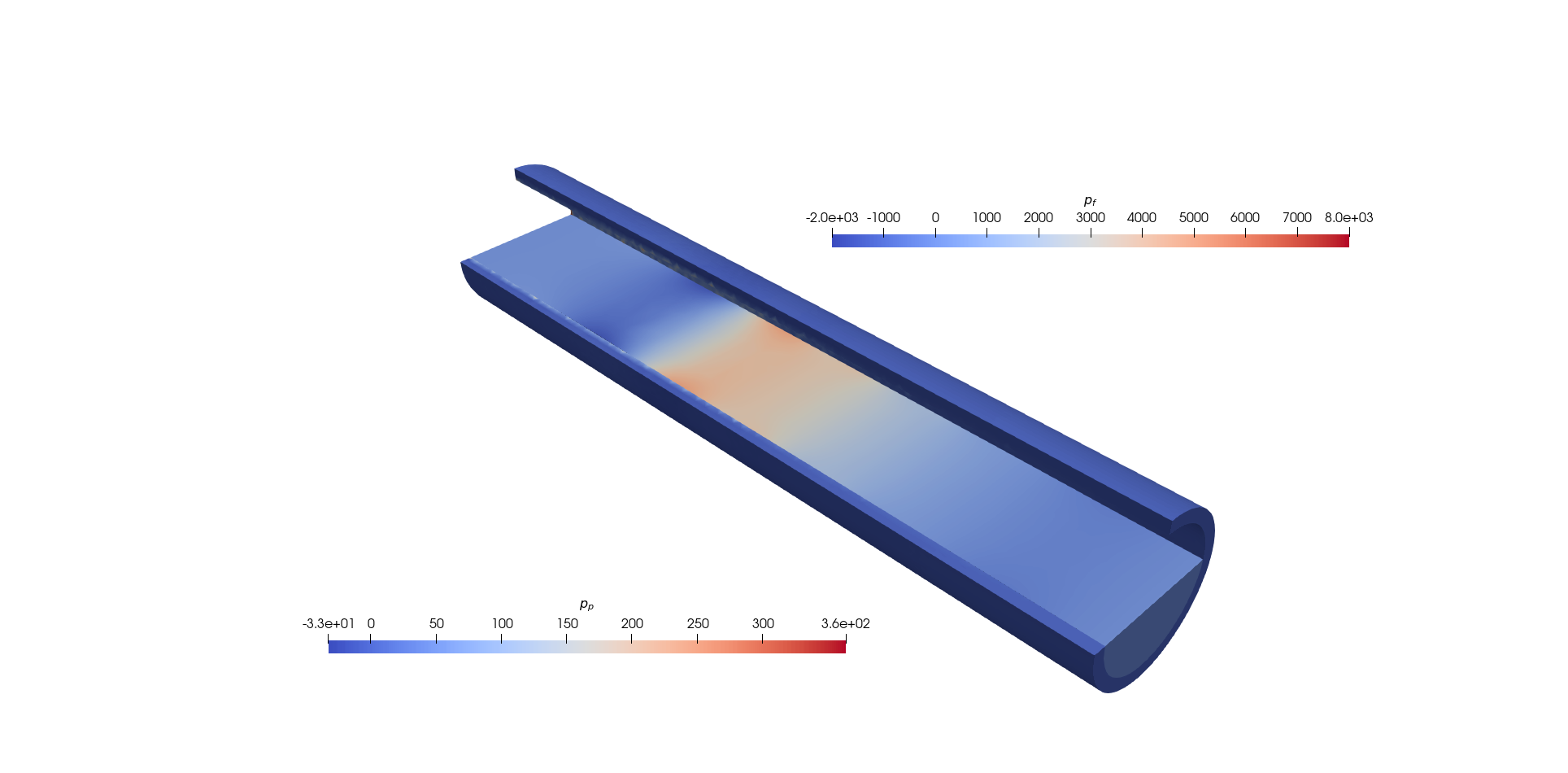}
&
\threedp{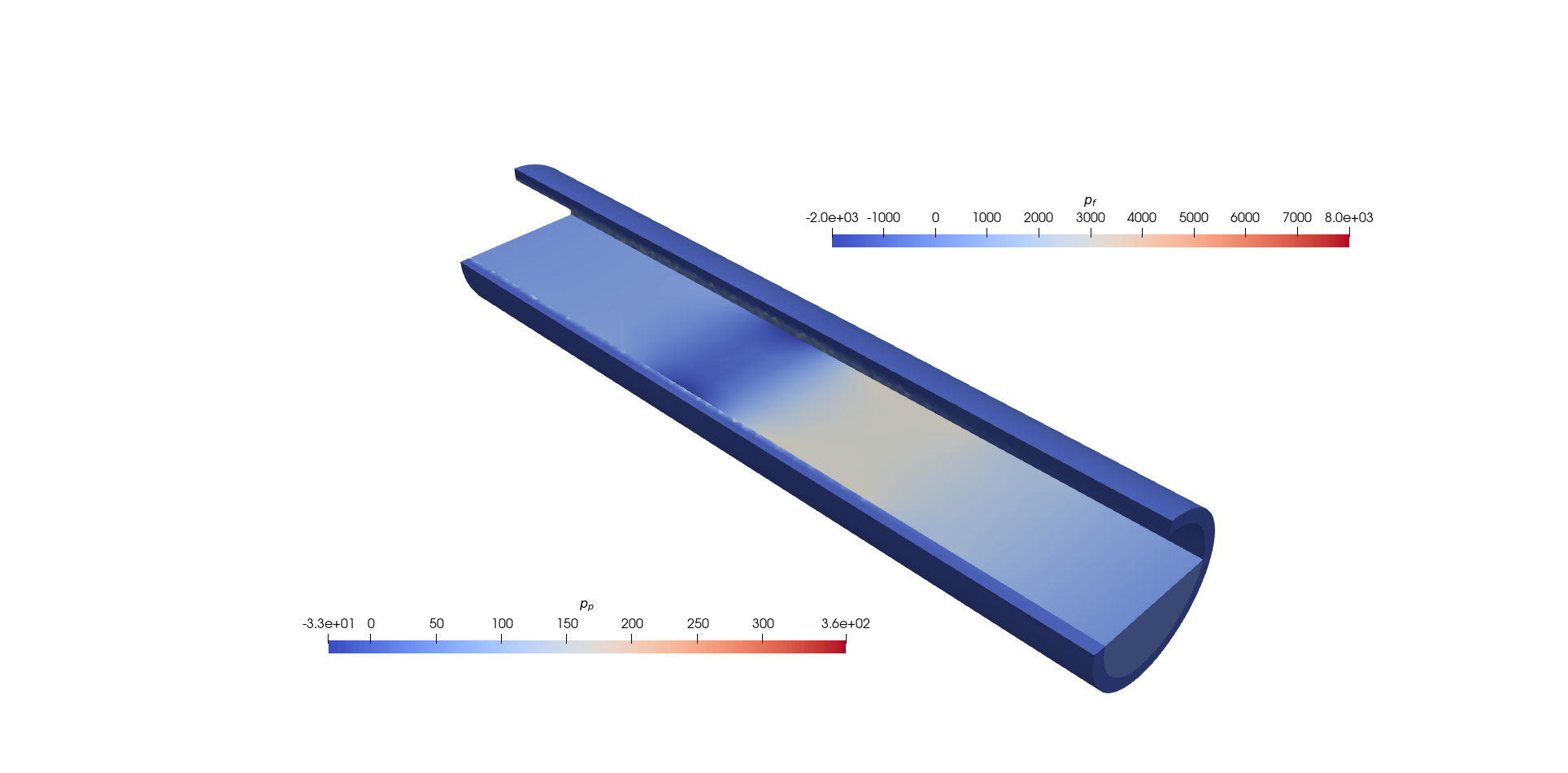}
\\[-0.35em]
{\small (c) $t=0.0105~\mathrm{s}$}
&
{\small (d) $t=0.014~\mathrm{s}$}
\end{tabular}

\caption{Three-dimensional propagation of the fluid pressure $p_f$ and
pore pressure $p_p$ for the pressure-wave problem. Identical color ranges
are used in all panels to facilitate direct comparison of the pressure-wave
location and amplitude.}
\label{fig:nonlinear_3d_pressure_fields}
\end{figure}

The pressure field exhibits a clear traveling-wave pattern. Initially, the
largest fluid pressure is confined to the upstream part of the tube. At the
subsequent observation times, the pressure disturbance moves progressively
downstream, while the pore-pressure response develops in the neighboring
poroelastic wall and follows the same propagation direction. Together with
Figure~\ref{fig:nonlinear_3d_velocity_displacement}, these results show that
the proposed partitioned strategy captures the coupled evolution of
pressure-wave propagation, wall deformation, and three-dimensional
rotation-dominated flow structures on the moving fluid domain.
\FloatBarrier

\bibliographystyle{plain}
\bibliography{reference}
\end{document}